\documentclass[11pt]{article}
\usepackage{amsmath,latexsym,amssymb,amsfonts,amsbsy, amsthm}
\allowdisplaybreaks[4]
\usepackage{graphicx}
\usepackage{empheq}
\usepackage{color}
\usepackage{ulem,cite}
\usepackage{amsthm}
\usepackage{appendix}

\numberwithin{equation}{section}
\newtheorem{thm}{Theorem}[section]

\newtheorem{pro}{Proposition}[section]

\newtheorem{Lemma}[pro]{Lemma}
\date{}

\begin{document}

\title{Boundary layer expansions of the steady MHD equations in a bounded domain}

\author{Dongfen Bian \footnote{ Email: {\tt biandongfen@bit.edu.cn}.}\ \ \ \  and \ \ \ \ Zhenjie Si \footnote{ Email: {\tt sizhenjie@bit.edu.cn, sizhenje@163.com}.}\\[1.26ex]
	\textit{\normalsize School of Mathematics and Statistics, Beijing Institute of Technology,}\\[0.32ex]
	\textit{\normalsize Beijing 100081, China}\\[0.32ex]
	\textit{\normalsize Beijing Key Laboratory on MCAACI, Beijing Institute of Technology,}\\[0.32ex]	
	\textit{\normalsize Beijing 102488, China}\\[0.32ex]
	}
\maketitle{}
\begin{abstract}
In this paper, we investigate the validity of boundary layer expansions for the MHD system in a rectangle. We describe the solution up to any order when the tangential magnetic field is much smaller than the tangential velocity field,
thereby extending \cite{DLX2021}.
\end{abstract}

~~~~~~{ \small {\bf MSC:} 76N10; 35Q30; 35R35}
\begin{center}
{ \small {\bf Key words:} the steady MHD equations, boundary layer expansions, any order,  bounded domain.}
\end{center}


\section{Introduction}



\subsection{The model}


We consider the following steady, incompressible magnetohydrodynamic system
\begin{align}\label{1.1}
\left\{\begin{array}{l}
U U_X+V U_Y-H H_X-G H_Y+P_X=\nu_{1}\varepsilon ( U_{X X}+U_{Y Y}),\\
U V_X+V V_Y-H G_X-G G_Y+P_Y=\nu_{2}\varepsilon ( V_{X X}+V_{Y Y}),\\
U H_X+V H_Y-H U_X-G U_Y=\nu_{3}\varepsilon ( H_{X X}+H_{Y Y}),\\
U G_X+V G_Y-H V_X-G V_Y=\nu_{4}\varepsilon ( G_{X X}+G_{Y Y}),\\
U_X+V_Y=H_X+G_Y=0,
\end{array}\right.
\end{align}
in the domain
$$
\Omega=\{(X, Y) \mid 0 \leq X \leq L, 0 \leq Y \leq 2\},
$$
where $\left(U,V\right)$ is the velocity field, $\left(H,G\right)$  the magnetic field, $P$  the pressure, $\nu_1 \varepsilon$, $\nu_2 \varepsilon$
the horizontal and vertical viscosities and $\nu_{3} \varepsilon$,  $\nu_{4}\varepsilon$ the horizontal and vertical magnetic resistivity coefficients.
We assume that $\Omega$ is limited by two plates moving with a velocity $u_b$, namely we consider the following boundary conditions
$$
U(X, 0)=U(X, 2)=u_b>0, \quad V(X, 0)=V(X, 2)=0,
$$
and
\begin{equation} \label{cond1}
(\partial_Y H, G) (X,0)=(\partial_Y H, G) (X,2)=0.
\end{equation}
The boundary conditions at $X = 0$ and $X = L$ will be made precise later.

When $\varepsilon = 0$,  $(\ref{1.1})$ reduces to the ideal magnetohydrodynamic system
\begin{align}\label{1.2}
\left\{\begin{array}{l}
U_{0} \partial_{X}U_{0}+V_{0} \partial_{Y}U_{0}-H_{0} \partial_{X}H_{0}-G_{0} \partial_{Y}H_{0}+\partial_{X}P_{0}=0,\\
U_{0} \partial_{X}V_{0}+V_{0} \partial_{Y}V_{0}-H_{0} \partial_{X}G_{0}-G_{0} \partial_{Y}G_{0}+\partial_{Y}P_{0}=0,\\
U_{0} \partial_{X}H_{0}+V_{0} \partial_{Y}H_{0}-H_{0} \partial_{X}U_{0}-G_{0} \partial_{Y}U_{0}=0,\\
U_{0} \partial_{X}G_{0}+V_{0} \partial_{Y}G_{0}-H_{0} \partial_{X}V_{0}-G_{0} \partial_{Y}V_{0}=0,\\
\partial_{X}U_{0}+\partial_{Y}V_{0}=\partial_{X}H_{0}+\partial_{Y}G_{0}=0,
\end{array}\right.
\end{align}
together with the boundary conditions
\begin{equation} \label{cond2}
V_0(X,0) = V_0(X,2) = 0, \qquad G_0(X,0) = G_0(X,2) = 0.
\end{equation}
We will consider particular solutions of this system, of the form
$$
(U_0,V_0,H_0,G_0) = \left(u_{e}^{0}(Y), 0, h_{e}^{0}(Y), 0\right),
$$
in which $u_{e}^{0}(\cdot),  h_{e}^{0}(\cdot)$ are smooth given functions, satisfying $u_e^0(0) = u_e^0(2) \ne u_b$, and
 symmetric with respect to $Y=1$, that is, for any $0 \le Y \le 1$
$$
u_{e}^{0}(1-Y)=u_{e}^{0}(1+Y),
\qquad h_{e}^{0}(1-Y)=h_{e}^{0}(1+Y) .
$$
The aim of this article is to construct a sequence of solution to (\ref{1.1}) which converges to $(U_0,V_0,H_0,G_0)$ as $\varepsilon \to 0$.
As (\ref{cond1}) and (\ref{cond2}) are different, we expect that the solution of (\ref{1.1}) has a boundary layer behavior at $Y = 0$ and $Y = 2$ \cite{Prandtl1904}.

When the magnetic field $(H,G)$  identically vanishes, (\ref{1.1}) reduces to the classical time independent Navier-Stokes equations.
For these equations, in $\Omega = [0,L] \times \mathbb{R}_+$, using energy estimates,  Guo and Nguyen \cite{GN2017} have proved the validity of the Prandtl boundary layer expansion, together with $L^\infty$ bounds on the error. This pioneering work has then be extended to the case $\Omega = [0,L] \times [0,2]$ in \cite{LD2020}. {{For more results on moving boundaries, please refer to \cite{Iyer2017,Iyer2019,Iyer2019-1,Iyer2020,Iyer2019-2}.
For the no-slip boundary condition, namely in the case $u_b = 0$, we in particular refer to the works of Guo and Iyer \cite{GI2023CPAM,GI2023QAM,GI2021}.}}

For the unsteady MHD equations, Xie, Luo and Li \cite{XLL2014} proved the existence of the boundary layer under the non-characteristic boundary conditions in the three dimensional case, and obtained that the solution of the viscous MHD equations converges to the solution of ideal MHD equations when the viscosity goes to zero.
With the no-slip boundary condition,  Wang and Xin \cite{WX2017} studied the initial boundary data problem of the boundary layer in the two and three dimensional cases.
Liu, Xie and Yang \cite{LXY2017} proved the validity of the MHD boundary layer expansion under non-degenerate boundary conditions on the tangential magnetic field. For more results, we refer to the works \cite{CRWZ2020,DLN2021,GVP2017,GGH2017,GJ2021,HLY2019,LLZ2015,LZ2018,LXY2019,WM2018,WW2019AMS,WW2019MMAS,XY2018SJMA} and the references therein.

For the steady MHD equations,  Gao,  Li and Yao \cite{GLY2024} proved  the high regularity and the asymptotic behavior of the Prandtl-Hartmann equations
using energy methods in Sobolev spaces. Liu, Yang and Zhang \cite{LYZ2023} proved the validity of the MHD  boundary layer expansion in the case of non-degenerate tangential magnetic field on the half plane. Under an assumption of  degeneracy on the tangential magnetic field, Ding, Ji and Lin \cite{DJL2022} obtained the stability of the Prandtl layer expansion when the outer ideal MHD flow is of the form $(1, 0, \sigma, 0) (\sigma \geq 0)$. 
When the outer ideal MHD flow is a shear flow, Ding, Lin and Xie  \cite{DLX2021} established the validity of the Prandtl boundary layer expansion and gave  $L^{\infty}$  estimates of the error terms. Later, Ding, Ji and Lin \cite{DJL2021} extended these results to the case of non-shear flows. In addition, Ding and Wang \cite{DW2023} generalized the results of \cite{Iyer2017} to  MHD flows.

In this paper, inspired by the methods used in \cite{LD2020} for the zeroth and first order expansions of the Navier-Stokes equations, we consider the validity of boundary layer expansions for the MHD system on $[0, L] \times [0,2]$ when the tangential magnetic field is much smaller than the tangential velocity field. 
We also precise the expansions of \cite{DLX2021,LD2020} by going up to any order.

We restrict ourselves to symmetric solutions, which satisfy, for $Y \in \left[0,1\right]$,
\begin{align}\label{1.3}
\begin{aligned}
U\left(1+Y\right)=U\left(1-Y\right), \quad V\left(1+Y\right)=-V\left(1-Y\right), \\
H\left(1+Y\right)=H\left(1-Y\right), \quad G\left(1+Y\right)=-G\left(1-Y\right).
\end{aligned}
\end{align}
Using this symmetry, it is sufficient to study the solutions in
$$
\Omega_{1}=\left\{\left.(X,Y)\right| 0 \leq X \leq L, 0\leq Y \leq 1\right\},
$$
with  boundary conditions
$$
(U, V, \partial_Y H, G)(X, 0)=(u_b, 0, 0, 0), \quad (\partial_Y U, V, \partial_Y H,G)(X,1)=(0,0,0,0).
$$
The first step is to introduce the rescaled variable
$$
x=X, \quad y=\frac{Y}{\sqrt{\varepsilon}},
$$
and the related unknown functions
\begin{align}\label{1.4}
\begin{array}{l}
U^{\varepsilon}(x,y)=U(X,Y), \quad V^{\varepsilon}(x,y)= \varepsilon^{-1/2} V(X,Y),\\
H^{\varepsilon}(x,y)=H(X,Y), \quad G^{\varepsilon}(x,y)=\varepsilon^{-1/2} G(X,Y), \quad P^{\varepsilon}(x,y)=P(X,Y).
\end{array}
\end{align}
This leads to
\begin{align}\label{1.5}
\left\{\begin{array}{l}
U^{\varepsilon} U^{\varepsilon}_x+V^{\varepsilon} U^{\varepsilon}_y-H^{\varepsilon} H^{\varepsilon}_x-G^{\varepsilon} H^{\varepsilon}_y+P^{\varepsilon}_x=\nu_{1}(\varepsilon U^{\varepsilon}_{x x}+U^{\varepsilon}_{y y}),\\
U^{\varepsilon} V^{\varepsilon}_x+V^{\varepsilon} V^{\varepsilon}_y-H^{\varepsilon} G^{\varepsilon}_x-G^{\varepsilon} G^{\varepsilon}_y+\varepsilon^{-1} P^{\varepsilon}_y=\nu_{2}(\varepsilon V^{\varepsilon}_{x x}+V^{\varepsilon}_{y y}),\\
U^{\varepsilon} H^{\varepsilon}_x+V^{\varepsilon} H^{\varepsilon}_y-H^{\varepsilon} U^{\varepsilon}_x-G^{\varepsilon} U^{\varepsilon}_y=\nu_{3}(\varepsilon H^{\varepsilon}_{x x}+H^{\varepsilon}_{y y}),\\
U^{\varepsilon} G^{\varepsilon}_x+V^{\varepsilon} G^{\varepsilon}_y-H^{\varepsilon} V^{\varepsilon}_x-G^{\varepsilon} V^{\varepsilon}_y=\nu_{4}(\varepsilon G^{\varepsilon}_{xx}+G^{\varepsilon}_{yy}),\\
U^{\varepsilon}_x+V^{\varepsilon}_y=H^{\varepsilon}_x+G^{\varepsilon}_y=0,
\end{array}\right.
\end{align}
 in the domain
 $$
 \Omega_{\varepsilon}=\left\{(x,y) \left| \right. 0\leq x \leq L, 0\leq y \leq \varepsilon^{-1/2} \right\},
 $$
together with the boundary conditions
\begin{align}\label{1.6}
(U^{\varepsilon}, V^{\varepsilon}, \partial_{y}H^{\varepsilon}, G^{\varepsilon})(x,0)=(u_{b},0, 0, 0),
\quad (\partial_{y}U^{\varepsilon}, V^{\varepsilon}, \partial_{y}H^{\varepsilon}, G^{\varepsilon})(x, \varepsilon^{-1/2}) =(0,0,0,0).
\end{align}
We now expand the solutions $(U^{\varepsilon}, V^{\varepsilon},H^{\varepsilon}, G^{\varepsilon},P^{\varepsilon})$ in $\varepsilon$
\begin{align}\label{1.7}
U^{\varepsilon}(x, y)=&u_e^0(\sqrt{\varepsilon}y)+u_p^0(x,y)+\sum_{i=1}^{n}\sqrt{\varepsilon}^{i}(u_e^i(x,\sqrt{\varepsilon} y)+u_p^i(x,y))+\varepsilon^{\gamma+\frac{n}{2}} u^{\varepsilon}(x, y), \nonumber\\
V^{\varepsilon}(x, y)=&\sum_{i=0}^{n-1}\sqrt{\varepsilon}^{i}(v_e^{i+1}(x,\sqrt{\varepsilon} y)+v_p^i(x,y))+\varepsilon^{\frac{n}{2}} v_p^n(x,y)
+\varepsilon^{\gamma+\frac{n}{2}} v^{\varepsilon}(x, y), \nonumber\\
H^{\varepsilon}(x, y)=&h_e^0(\sqrt{\varepsilon}y)+h_p^0(x,y)+\sum_{i=1}^{n}\sqrt{\varepsilon}^{i}(h_e^i(x,\sqrt{\varepsilon} y)+h_p^i(x,y))+\varepsilon^{\gamma+\frac{n}{2}} h^{\varepsilon}(x, y), \\
G^{\varepsilon}(x, y)=&\sum_{i=0}^{n-1}\sqrt{\varepsilon}^{i}(g_e^{i+1}(x,\sqrt{\varepsilon} y)+g_p^i(x,y))+\varepsilon^{\frac{n}{2}} g_p^n(x,y)
+\varepsilon^{\gamma+\frac{n}{2}} g^{\varepsilon}(x, y), \nonumber\\
P^{\varepsilon}(x, y)=&\sum_{i=1}^{n}\sqrt{\varepsilon}^i (p_e^i(x,\sqrt{\varepsilon}y)+ p_p^i(x, y))
+\varepsilon^{\frac{n+1}{2}} p_p^{n+1}(x, y)
+\varepsilon^{\gamma+\frac{n}{2}} p^{\varepsilon}(x, y),\nonumber
\end{align}
where $(u_e^i, v_e^i, h_e^i, g_e^i, p_e^i)$ and $(u_p^i, v_p^i, h_p^i, g_p^i, p_p^i) $ $(i=0,1,2,\ldots,n)$ denote the interior and
boundary layer parts, and where $(u^{\varepsilon}, v^{\varepsilon}, h^{\varepsilon}, g^{\varepsilon}, p^{\varepsilon})$ are error terms.
These approximate solutions are constructed in Sections $2$ to $3$.


\subsection{Boundary conditions and compatibility conditions}


Let us now detail the boundary conditions and compatibility conditions of the various orders of the asymptotic expansion.
First,  we choose the zeroth order ideal inner MHD profile to be $(u_e^0(Y), 0, h_e^0(Y), 0)$.
At $y = 0$, the zeroth order interior flow does not satisfy the boundary condition since, by assumption, $u_e^0(0):=u_e\neq u_b$.
Similarly, $h_e^0(0):=h_e \neq 0$. We thus add a boundary layer $u_p^0$ on the velocity, and a boundary layer $h_p^0$ on the magnetic field.
We have to prescribe the values of $u_p^0$ and $h_p^0$ at $x = 0$. For this we choose two smooth and decaying functions $\tilde u_0$ and
$\tilde h_0$.

The boundary conditions on the zeroth order MHD boundary layer profile are thus
\begin{align}\label{1.11}
\begin{array}{lll}
u_p^0(x, 0)+u_e=u_b, \quad \partial_{y} h_p^0(x,0)=0, \\
\partial_{y} u_p^0(x,\varepsilon^{-1/2})
=v_p^0(x,\varepsilon^{-1/2})=\partial_{y} h_p^0(x,\varepsilon^{-1/2})
=g_p^0(x,\varepsilon^{-1/2})=0,\\
u_p^0(0,y)=\tilde{u}_{0}(y), \quad h_p^0(0,y)=\tilde{h}_{0}(y).
\end{array}
\end{align}
For the $i$-th order ideal inner MHD profiles, the boundary conditions are, for $i = 1,2,\ldots,n$,
\begin{align}\label{1.12}
\begin{array}{llll}
v_e^i(x,0)+v_p^{i-1}(x,0)=0, \quad g_e^i(x,0)+g_p^{i-1}(x,0)=0,
\quad u_e^i(0,Y)=u_b^i(Y),\\
v_e^i(0,Y)=V_{b 0}^i(Y), \quad h_e^i(0,Y)=h_b^i(Y),
\quad g_e^i(0,Y)=G_{b 0}^i(Y),\\
v_e^i(L,Y)=V_{b L}^i(Y), \quad g_e^i(L,Y)=G_{b L}^i(Y),\\
u_{e Y}^i(x,1)=v_e^i(x,1)=h_{e Y}^i(x,1)=g_e^i(x,1)=0,
\end{array}
\end{align}
where $u_b^i(Y)$, $h_b^i(Y)$, $V_{b0}^i(Y)$, $G_{b0}^i(Y)$, $V_{bL}^i(Y)$ and $G_{bL}^i(Y)$ are given functions.

For the $i$-th order MHD  boundary layer profiles, we have, for $i = 1,2,\ldots,n$,
\begin{align}\label{1.11.1}
\begin{array}{lll}
u_p^i(x, 0)+u_e^i(x,0)=0, \quad \partial_{y} h_p^i(x,0)+\partial_{Y}h_e^{i-1}(0)=0, \\
\partial_{y} u_p^i(x,\varepsilon^{-1/2})
=v_p^i(x,\varepsilon^{-1/2})
=\partial_{y} h_p^i(x,\varepsilon^{-1/2})
=g_p^i(x,\varepsilon^{-1/2})=0,\\
u_p^i(0,y)=\tilde{u}_{i}(y), \quad h_p^i(0,y)=\tilde{h}_{i}(y),
\end{array}
\end{align}
where $\tilde u_i$ and $\tilde h_i$ are given smooth and rapidly decreasing functions.

Morevoer, we also have $v_p^n(x,0)=g_p^n(x,0)=0$ for the $n$-th order boundary layer.
Finally, the boundary conditions of the remainder terms $\left(u^{\varepsilon}, v^{\varepsilon}, h^{\varepsilon}, g^{\varepsilon}, p^{\varepsilon}\right)$ are
\begin{align}\label{1.14}
& u^{\varepsilon}(x,0)=0, \quad v^{\varepsilon}(x,0)=0,
\quad \partial_{y}h^{\varepsilon}(x,0)=0, \quad g^{\varepsilon}(x,0)=0,\nonumber\\
&\partial_{y}u^{\varepsilon}(x,\varepsilon^{-1/2})=0, \quad v^{\varepsilon}(x,\varepsilon^{-1/2})=0, \quad\partial_{y}h^{\varepsilon}(x,\varepsilon^{-1/2})=0,
\quad g^{\varepsilon}(x,\varepsilon^{-1/2})=0,
\nonumber\\
& u^{\varepsilon}(0,y)=0, \quad v^{\varepsilon}(0,y)=0, \quad h^{\varepsilon}(0,y)=0, \quad g^{\varepsilon}(0,y)=0,\\
& (p^{\varepsilon}-2\nu_{1}\varepsilon u^{\varepsilon}_{x})(L,y)=0,
\quad (u^{\varepsilon}_y + \nu_{2}\varepsilon v^{\varepsilon}_x)(L,y)=h^{\varepsilon}(L,y)
=\partial_{x}g^{\varepsilon}(L,y)=0.\nonumber
\end{align}
To solve this system $(\ref{1.5})$, we need the following  compatibility conditions
\begin{align}\label{1.15}
&(V_{b 0}^{i},G_{b 0}^{i})(0)=-(v_p^{i-1}, g_p^{i-1})(0,0), \quad (V_{b L}^{i},G_{b L}^{i})(0)=-(v_p^{i-1}, g_p^{i-1})(L, 0), \nonumber\\
& V_{b 0}^{i}(1)=V_{b L}^{i}(1)=0, \quad
 G_{b 0}^{i}(1)=G_{b L}^{i}(1)=0.
\end{align}


\subsection{Main result}


We now present the main result in this paper. We will assume that the tangential velocity dominates the horizontal magnetic field. 
More precisely, we assume the following assumption, called (H).

\begin{itemize}

\item (H) There exists $C_0 > 1$, such that, for any $0 \leq Y \leq 1$,
 $$
 u_{e}^{0}(Y) \geq C_0  h_{e}^{0}(Y),
 $$
 $$
 {{u_{e}+\tilde{u}_{0}(y) > h_{e} + \tilde{h}_{0}(y) \geq \theta_{0} > 0}}
 $$
 and
\begin{align}\label{1.16}
|u_{e}^{0}(\sqrt{\varepsilon}y)+\tilde{u}_{0}(y)| \geq C_0 |h_{e}^{0}(\sqrt{\varepsilon}y)+\tilde{h}_{0}(y)| .
\end{align}
\end{itemize}

\begin{thm}\label{b1.1}
{{
Let $u_{b} > 0$ and let $u_{e}^{0}(Y), h_{e}^{0}(Y)$ be given smooth positive functions satisfying $\partial_{Y}u_{e}^{0}(1)=\partial_{Y}h_{e}^{0}(1)=0$,
together with the boundary conditions. Suppose that
$$\|\langle Y \rangle \partial_{Y}(u_e^0, h_e^0)\|_{L^{\infty}}\lesssim \sigma_{1},$$
for some sufficiently small $\sigma_{1}>0$.
Assume that (H) holds true for some $C_0$ large enough.
 Assume moreover that, for any $0 \le Y \le 1$, }}
\begin{align}\label{1.17.1}
& |\langle y \rangle^{l+1} \partial_{y}(u_{e}+\tilde{u}_{0}, h_{e}+ \tilde{h}_{0})(y)| \leq \frac{1}{2} \sigma_{0},\nonumber\\
& |\langle y \rangle^{l+1} \partial_{y}^{2}(u_{e}+\tilde{u}_{0}, h_{e}+ \tilde{h}_{0})(y)| \leq \frac{1}{2} \theta_{0}^{-1},
\end{align}
where $\sigma_0$ and $\theta_0$ are small enough.
Then there exists a constant $L_{0}>0$, which depends only on the prescribed data, such that for $0 < L \leq L_{0}$ and $\gamma \in (0, \frac{1}{16})$,
provided $\varepsilon$ is small enough, there exists a solution $(U^{\varepsilon}, V^{\varepsilon}, H^{\varepsilon}, G^{\varepsilon}, P^{\varepsilon})$  to the system $(\ref{1.5})$ on  $\Omega_{\varepsilon}$.
Moreover, the error terms $\left(u^{\varepsilon}, v^{\varepsilon}, h^{\varepsilon}, g^{\varepsilon}\right)$ satisfy
\begin{align}\label{1.18}
& \|\nabla_\varepsilon u^\varepsilon\|_{L^2(\Omega_{\varepsilon})}+\|\nabla_\varepsilon v^\varepsilon\|_{L^2(\Omega_{\varepsilon})}
+\|\nabla_\varepsilon
h^\varepsilon\|_{L^2(\Omega_{\varepsilon})}
+\|\nabla_\varepsilon g^\varepsilon\|_{L^2(\Omega_{\varepsilon})}\nonumber\\
&+\varepsilon^{\frac{\gamma}{8}}\|u^\varepsilon\|_{L^{\infty}(\Omega_{\varepsilon})}
+\varepsilon^{\frac{1}{2}+\frac{\gamma}{8}}\|v^\varepsilon\|_{L^{\infty}(\Omega_{\varepsilon})}
+\varepsilon^{\frac{\gamma}{8}}\|h^\varepsilon\|_{L^{\infty}(\Omega_{\varepsilon})}
+\varepsilon^{\frac{1}{2}+\frac{\gamma}{8}}\|g^\varepsilon\|_{L^{\infty}(\Omega_{\varepsilon})} \leq C_1,
\end{align}
where the positive constant $C_1$ depends only on norms of $u_e^0$ and $h_e^0$.
\end{thm}

This theorem extends \cite{DLX2021,LD2020} in two directions. First we consider the system $(\ref{1.1})$ in a bounded domain, which is different from that in \cite{LD2020}.  Additionally, in contrast to \cite{DLX2021} , we take into account the influence of the magnetic field. 
Second we justify the asymptotic expansion at any order, whereas \cite{DLX2021,LD2020}  just investigate the leading order.

%
%
%


\subsection{Notations}

{{
Throughout this paper, we use the notations $\langle y\rangle= \sqrt{1+y^2}$
and $I_{\varepsilon}= \left[0, \varepsilon^{-1/2}\right]$. Moreover, $\bar{f}=f(x,0)$ denotes the boundary value of a function $f$ at $y=0$ in Sections 2 to 4.}}
For $l \in \mathbb{R}$, we introduce the following $L^2$ weighted space
$$
L_l^2:=\left\{f(x, y):[0, L] \times[0, \infty) \rightarrow \mathbb{R},\|f\|_{L_l^2}^2=\int_0^{\infty}\langle y\rangle^{2 l}|f(y)|^2 \mathrm{~d} y<\infty\right\}.
$$
For $\alpha=(k_1, k_2) \in \mathbb{N}^2$, we define the Sobolev spaces
$$
H_l^m:=\left\{f(x, y):[0, L] \times[0, \infty) \rightarrow \mathbb{R},\|f\|_{H_l^m}^2<\infty\right\}
$$
and  denote its norm by
$$
\|f\|_{H_l^m}^2=\sum_{|\alpha| \leq m}\|\langle y\rangle^{l+k_2} D^\alpha f\|_{L_{y}^2(0, \infty)}^2,
$$
where $D^\alpha=\partial_x^{k_1} \partial_y^{k_2}$.

\subsection{Organization of the paper}

 The rest of this article is as follows. In Section $2$, we study the leading order.
 In Section $3$, we construct ideal inner MHD profiles and MHD boundary layer profiles at any order in domains $\Omega_1$ and $\Omega_{\varepsilon}$, respectively. 
 In Section $4$, we establish  energy  and $L^{\infty}$ estimates of the remainder $(u^{\varepsilon},v^{\varepsilon},h^{\varepsilon},g^{\varepsilon},p^{\varepsilon})$
 and prove Theorem $\ref{b1.1}$. Finally, in Appendix A, we derive the equations of $(u_e^i, v_e^i, h_e^i, g_e^i, p_e^i)$ and $(u_p^i, v_p^i, h_p^i, g_p^i, p_p^i), (0\leq i\leq n)$.


\section{The zeroth order MHD boundary layer profile}


In this section, we consider the zeroth order MHD boundary layer profile $(u_p^0, v_p^0, h_p^0, g_p^0,0)$, 
which satisfies (see Appendix)
\begin{align}\label{2.6.1}
\left\{\begin{array}{l}
(u_e+u_p^0) \partial_x u_p^0+(v_p^0+\overline{v_e^1}) \partial_y u_p^0-(h_e+h_p^0) \partial_x h_p^0-(g_p^0+\overline{g_e^1}) \partial_y h_p^0= \nu_{1}\partial_y^2 u_p^0, \\
(u_e+u_p^0) \partial_x h_p^0+(v_p^0+\overline{v_e^1}) \partial_y h_p^0-(h_e+h_p^0) \partial_x u_p^0-(g_p^0+\overline{g_e^1}) \partial_y u_p^0= \nu_{3}\partial_y^2 h_p^0, \\
(v_p^0, g_p^0)(x, y)=\int_y^{\varepsilon^{-1/2}} \partial_x (u_p^0, h_p^0)(x, z) \mathrm{d} z, \\
(v_e^1, g_e^1)(x, 0)=-\int_0^{\varepsilon^{-1/2}} \partial_x(u_p^0, h_p^0)(x, z) \mathrm{d} z, \\
(u_p^0, \partial_y h_p^0)(x, 0)=(u_b-u_e, 0), \\
(v_p^0, g_p^0)(x, 0)=-(v_e^1, g_e^1)(x, 0),
(u_p^0, h_p^0)(0, y)=(\tilde{u}_0(y), \tilde{h}_0(y)).
\end{array}\right.
\end{align}

\subsubsection*{Step 1:}

We first solve $(\ref{2.6.1})$ on $[0,L] \times \mathbb{R}_+$ and denote  by $(u_p^{0,\infty},v_p^{0,\infty}, h_p^{0,\infty}, g_p^{0,\infty})$
the corresponding solution.

\begin{pro}\label{b2.1}
Let $m \geq 2n+3, (n\geq 3)$ and $l \geq 0$ and let $u_e^0$ and $h_e^0$ be smooth functions satisfying the assumptions of Theorem $\ref{b1.1}$.
Assume  moreover that $\partial_{Y}u_e^0, \partial_{Y}h_e^0$ and their derivatives decay fast at infinity. Then, the solutions $u_p^{0,\infty}, v_p^{0,\infty}, h_p^{0,\infty} ~\text{and}~ g_p^{0,\infty}$ to the system $(\ref{2.6.1})$ satisfy
\begin{align}\label{2.9}
 \sup _{0 \leq x \leq L}\|\langle y\rangle^l D^\alpha (u_p^{0,\infty}, v_p^{0,\infty}, h_p^{0,\infty}, g_p^{0,\infty})\|_{L^2(0, \infty)} \leq C (l, m, \tilde{u}_{0}, \tilde{h}_{0}),
\end{align}
for any $\alpha, l ~\text{and}~ m$ with $|\alpha|\leq m$.
\end{pro}

\begin{proof}

Let us introduce
\begin{align}\label{2.11}
\left\{\begin{array}{l}
u=u_p^{0,\infty}+u_e-u_e \phi(y)-u_b(1-\phi(y)), \\
v=v_e^1(x, 0)+v_p^{0,\infty}, \\
h=h_p^{0,\infty}+h_e-h_e \phi(y), \\
g=g_e^1(x, 0)+g_p^{0,\infty},
\end{array}\right.
\end{align}
in which $\phi (y)$ is the following  cut-off function
\begin{align}\label{2.10}
\phi (y)= \left\{
\begin{aligned}
&1, \quad y \geq 2R_{0},\\
&0,  \quad  y \in [0, R_0],
\end{aligned}\right.
\end{align}
where $R_0 > 0$.
Putting $(\ref{2.11})$ into $(\ref{2.6.1})$,  we obtain
\begin{align}\label{2.12}
\left\{\begin{array}{l}
{[(u+u_e \phi(y)+u_b(1-\phi(y))) \partial_x+v \partial_y] u-[(h+h_e \phi(y)) \partial_x+g \partial_y] h} \\
\quad \qquad \qquad - \nu_{1}\partial_y^2 u-g h_e \phi^{\prime}(y)+v(u_e-u_b) \phi^{\prime}(y)=a_1, \\
{[(u+u_e \phi(y)+u_b(1-\phi(y))) \partial_x+v \partial_y] h-[(h+h_e \phi(y)) \partial_x+g \partial_y] u} \\
\qquad \qquad \quad - \nu_{3}\partial_y^2 h-g(u_e-u_b) \phi^{\prime}(y)+v h_e \phi^{\prime}(y)=a_2, \\
\partial_x u+\partial_y v=\partial_x h+\partial_y g=0, \\
\left.(u, v, \partial_y h, g)\right|_{y=0}=(0,0,0,0), \\
(u, h) \rightarrow(0,0) \text { as } y \rightarrow \infty, \\
\left.(u,h)\right|_{x=0}=(\tilde{u}_0(y)+(u_e-u_b)(1-\phi(y)), \tilde{h}_0(y)+h_e(1-\phi(y))) := (u_0, h_0)(y),
\end{array}\right.
\end{align}
where
$$
a_1=\nu_{1}(u_b-u_e) \phi^{\prime \prime}(y), \quad a_2= \nu_{3}h_e \phi^{\prime \prime}(y) .
$$
The proposition is a direct consequence of the following lemma.
\end{proof}

{{
\begin{Lemma}\label{b2.2}
Under the assumptions of Proposition $\ref{b2.1}$ and Theorem $\ref{b1.1}$, there exists a smooth solution $\left(u, v, h, g\right)$ to $(\ref{2.12})$, such that
\begin{align}\label{2.13}
\sup _{0 \leq x \leq L} \|(u, h)\|_{H_l^m(0, \infty)}
+\|\partial_y(u, h)\|_{L^2(0, L; H_l^m(0, \infty))} \leq C(l, m, u_0, h_0),
\end{align}
where $C(l, m, u_0, h_0)$ depends only on $l, m, u_0 ~\text{and}~ h_0$.
Furthermore, for $(x, y) \in [0, L] \times[0, \infty)$,
under assumption (H) and provided $C_0$ is large enough,
$(u_p^{0,\infty},h_p^{0,\infty})$ satisfies the following estimates
$$
h_e+h_p^{0,\infty}(x, y) \geq \frac{\theta_0}{2}>0,
$$
$$
\left|u_e^0(\sqrt{\varepsilon} y)+u_p^{0,\infty}(x, y)\right|\geq C_0\left|h_e^0(\sqrt{\varepsilon} y)+h_p^{0,\infty}(x, y)\right|,
$$
and
$$
 \left|\langle y\rangle^{l+1} \partial_y\left(u_e+u_p^{0,\infty}, h_e+h_p^{0,\infty}\right)(x, y)\right| \leq \sigma_0,
$$
$$
 \left|\langle y\rangle^{l+1} \partial_y^2\left(u_e+u_p^{0,\infty}, h_e+h_p^{0,\infty}\right)(x, y)\right| \leq \theta_0^{-1} .
$$
\end{Lemma}}}

\begin{proof}
By using the divergence free conditions, the second equation of $(\ref{2.12})$ may be rewritten
\begin{align}\label{2.15}
\partial_y[v (h+h_e \phi)
-g (u+u_e \phi + u_b (1-\phi))] - \nu_{3}\partial_y^2 h = \nu_{3}\phi^{\prime \prime}(y) h_e.
\end{align}
Integrating  $(\ref{2.15})$ from $0$ to $y$, we obtain
\begin{align}\label{2.16}
v (h+h_e \phi)
-g (u+u_e \phi + u_b (1-\phi))- \nu_{3}\partial_y h = \nu_{3}\phi^{\prime}(y) h_e,
\end{align}
in which we have used the following boundary conditions
$$
\left. v\right|_{y=0}=\left. g\right|_{y=0}
=\left.\partial_{y}h\right|_{y=0}=\left. \phi^{\prime}\right|_{y=0}=0.
$$
We introduce the stream function ${\psi}$, such that
$$
{\psi}=\int_0^y h(x, \theta) \mathrm{d} \theta, \quad \left. \psi\right|_{y=0}=0.
$$
Then we can obtain via the divergence free condition
$$
\partial_x {\psi}=-g, \quad \partial_y {\psi}=h.
$$
The equation (\ref{2.16}) may be rewritten as
\begin{align}\label{2.16.1}
[(u+u_e \phi +u_b (1-\phi))\partial_{x} +v \partial_{y}] \psi + v h_e \phi -\nu_{3}\partial_{y}^{2} \psi =\nu_{3}h_e \phi^{\prime}.
\end{align}
Here we just sketch the outline about the proof of the Lemma under assumption (H). First, by  standard energy methods, we obtain  weighted $L^2$ estimates of $\partial_x^{k_1} \partial_{y}^{k_2}(u,h)$, for $k_1+k_2\leq m$ and $k_1\leq m-1 ~(m \geq 2n+3)$. 
Second, we focus on  weighted $L^2$ estimates of the $m$-order tangential derivatives $\partial_x^{m}(u, h)$. For this we  introduce the new unknowns
$$
u^{m}:=\partial_x^{m} u-\frac{\partial_y u + (u_e-u_b) \phi^{\prime}}{h+h_e \phi} \partial_x^{m} {\psi},
\quad h^{m}:=\partial_x^{m} h-\frac{\partial_y h + h_e \phi^{\prime}}{h+h_e \phi} \partial_x^{m} {\psi}.
$$
By applying $\partial_{x}^{m}$ to $(\ref{2.12})$ and $(\ref{2.16.1})$ respectively and combining them, we obtain  equations of $u^{m} ~\text{and}~ h^{m}$, in which the terms involving $\partial_x^{m}(v, g)$ vanish. We can then obtain  weighted $L^2$ estimates for $u^{m}$ and $h^{m}$, using 
$h+h_e\phi \geq \frac{\theta_0}{2}>0$.
To derive the estimates for $\partial_{x}^{m}(u, h)$, we establish the equivalence of $L^2$ norms between $\partial_x^{m}(u, h)$ and $(u^{m}, h^{m})$. This leads to the desired estimates of $\partial_x^{m}(u, h)$.

Finally, we obtain the existence of the solution $(u, v, h, g)$ to the problem $(\ref{2.12})$ through  the classical Picard iteration and fixed point theorem. For more details, we  refer to \cite{DLX2021} and \cite{LXY2019}.
\end{proof}

\subsubsection*{Step 2:}

We cut  the domain from $\mathbb{R}_+$ to $I_{\varepsilon}$, and obtain  estimates on $[0,L] \times I_{\varepsilon}$.
\begin{pro}\label{b2.4}
(Estimates on $[0,L] \times I_{\varepsilon}$).\\
Under the assumptions in Theorem $\ref{b1.1}$, there exist smooth functions $(u_p^0, v_p^0, h_p^0, g_p^0)$ on $\Omega_{\varepsilon}$, which satisfy the following inhomogeneous system
\begin{align}\label{2.17}
\left\{\begin{array}{l}
(u_e+u_p^0) u_{p x}^0+(v_p^0+v_e^1(x, 0)) u_{p y}^0-(h_e+h_p^0) h_{p x}^0-(g_p^0+g_e^1(x, 0)) h_{p y}^0-\nu_{1}u_{p y y}^0=R_p^{1, 0}, \\
(u_e+u_p^0) h_{p x}^0+(v_p^0+v_e^1(x, 0)) h_{p y}^0-(h_e+h_p^0) u_{p x}^0-(g_p^0+g_e^1(x, 0)) u_{p y}^0-\nu_{3}h_{p y y}^0=R_p^{3, 0}, \\
u_{p x}^0+v_{p y}^0=h_{p x}^0+g_{p y}^0=0, \\
(u_p^0, \partial_y h_p^0)(x, 0)=(u_b-u_e, 0), u_{p y}^0(x, \varepsilon^{-1/2})=h_{p y}^0(x, \varepsilon^{-1/2})=0,\\
(v_p^0+v_e^1)(x, 0)=(g_p^0+g_e^1)(x, 0)=0, v_p^0(x, \varepsilon^{-1/2})=g_p^0(x, \varepsilon^{-1/2})=0,
\end{array}\right.
\end{align}
such that for any given $l, m \in \mathbb{N}$,
\begin{align}\label{2.18}
\sup_{x \in[0, L]}\|\langle y\rangle^l D^{\alpha}(u_p^0, v_p^0, h_p^0, g_p^0)\|_{L^2(I_{\varepsilon})} \leq C (l,m, \tilde{u}_0, \tilde{h}_0),
\end{align}
where $R_p^{1, 0}$ and $R_p^{3, 0}$ are higher order terms of $\sqrt{\varepsilon}$.
\end{pro}
 \begin{proof}
 $u_p^{0,\infty} ~\text{and}~ h_p^{0,\infty}$ have already been constructed in Proposition $\ref{b2.1}$, and $v_p^{0,\infty}$ and $g_p^{0,\infty}$ are directly obtained through the divergence free conditions. Next, we first introduce the cut-off function $\chi(\cdot)$ with support in $[0,1]$, satisfying $\chi(0)=1, \chi(1)= \chi^{\prime}(0)=\chi^{\prime}(1)=0$, and define
\begin{align}\label{2.19}
\left\{\begin{aligned}
& u_p^0(x, y):=\chi(\sqrt{\varepsilon} y) u_p^{0,\infty}(x, y)-\sqrt{\varepsilon} \chi^{\prime}(\sqrt{\varepsilon} y) \int_y^{\infty} u_p^{0,\infty}(x, \theta) \mathrm{d} \theta, \\
& v_p^0(x, y):=\chi(\sqrt{\varepsilon} y) v_p^{0,\infty}(x, y),\\
&h_p^0(x, y):=\chi(\sqrt{\varepsilon} y) h_p^{0,\infty}(x, y)-\sqrt{\varepsilon} \chi^{\prime}(\sqrt{\varepsilon} y) \int_y^{\infty} h_p^{0,\infty}(x, \theta) \mathrm{d} \theta, \\
& g_p^0(x, y):=\chi(\sqrt{\varepsilon} y) g_p^{0,\infty}(x, y).
\end{aligned}
\right.
\end{align}
It is easy to check that $(u_p^0, v_p^0, h_p^0,g_p^0)$ satisfy $(\ref{2.17})$ with
\begin{align*}
R_p^{1,0}= \sqrt{\varepsilon}\triangle_0^1 +\varepsilon \triangle_0^2, \quad
R_p^{3,0}=\sqrt{\varepsilon}\triangle_0^3 +\varepsilon \triangle_0^4,
\end{align*}
where
\begin{align*}
\triangle_0^{1}  := &   \chi \int_0^y \chi^{\prime} \mathrm{d} \theta (u_p^{0,\infty} u_{p x}^{0,\infty}+v_p^{0,\infty} u_{p y}^{0,\infty}-h_p^{0,\infty} h_{p x}^{0,\infty}-g_p^{0,\infty} h_{p y}^{0,\infty})- \chi^{\prime} \chi ( u_{p x}^{0,\infty} \int_y^{\infty} u_p^{0,\infty} \mathrm{d} \theta \nonumber\\
& -h_{p x}^{0,\infty} \int_y^{\infty} h_p^{0,\infty} \mathrm{d} \theta)- \chi^{\prime} [v_p^{0,\infty}(u_e
+\chi u_p^{0,\infty})
-g_p^{0,\infty}(h_e+\chi h_p^{0,\infty})]-3 \nu_{1}  \chi^{\prime} u_{p y}^{0,\infty}\nonumber\\
& +2 \chi^{\prime} (u_p^{0,\infty} \int_0^y \chi v_{p y}^{0,\infty} \mathrm{d} \theta - h_p^{0,\infty} \int_0^y \chi g_{p y}^{0,\infty} \mathrm{d} \theta),\nonumber\\
\triangle_0^2:= & 2 \chi^{\prime} (u_p^{0,\infty} \int_0^y \chi^{\prime} v_p^{0,\infty} \mathrm{d} \theta-h_p^{0,\infty} \int_0^y \chi^{\prime} g_p^{0,\infty} \mathrm{d} \theta)
-3 \nu_{1}  \chi^{\prime \prime} u_p^{0,\infty}
+ (\chi^{\prime})^2 v_p^{0,\infty} \int_y^{\infty} u_p^{0,\infty} \mathrm{d} \theta\\
&-(\chi^{\prime})^2 g_p^{0,\infty} \int_y^{\infty} h_p^{0,\infty} \mathrm{d} \theta - \chi^{\prime \prime}[(\chi v_p^{0,\infty}-v_p^{0,\infty}(0)) \int_y^{\infty} u_p^{0,\infty} \mathrm{d} \theta -(\chi g_p^{0,\infty}-g_p^{0,\infty}(0)) \int_y^{\infty} h_p^{0,\infty} \mathrm{d} \theta] \\
&+ \nu_{1} \varepsilon^{1/2} \chi^{\prime \prime \prime} \int_y^{\infty} u_p^{0,\infty} \mathrm{d} \theta, \nonumber\\
\triangle_0^3:= &  \chi \int_0^y \chi^{\prime} \mathrm{d} \theta (u_p^{0,\infty} h_{p x}^{0,\infty}+v_p^{0,\infty} h_{p y}^{0,\infty}-h_p^{0,\infty} u_{p x}^{0,\infty}-g_p^{0,\infty} u_{p y}^{0,\infty})- \chi^{\prime} \chi ( h_{p x}^{0,\infty} \int_y^{\infty} u_p^{0,\infty} \mathrm{d} \theta \nonumber\\
& -u_{p x}^{0,\infty} \int_y^{\infty} h_p^{0,\infty} \mathrm{d} \theta)- \chi^{\prime}[g_p^{0,\infty}(u_e+\chi u_p^{0,\infty})
-v_p^{0,\infty}(h_e+\chi h_p^{0,\infty})]-3 \nu_{3} \chi^{\prime} h_{p y}^{0,\infty}\nonumber\\
&+2 \chi^{\prime} (h_p^{0,\infty} \int_0^y \chi v_{p y}^{0,\infty} \mathrm{d} \theta -u_p^{0,\infty} \int_0^y \chi g_{p y}^{0,\infty} \mathrm{d} \theta ),\\
\triangle_0^4:= & 2 \chi^{\prime} (h_p^{0,\infty} \int_0^y \chi^{\prime} v_p^{0,\infty} \mathrm{d} \theta-u_p^{0,\infty} \int_0^y \chi^{\prime} g_p^{0,\infty} \mathrm{d} \theta )-3\nu_{3}  \chi^{\prime \prime} h_p^{0,\infty} + (\chi^{\prime})^2 g_p^{0,\infty} \int_y^{\infty} u_p^{0,\infty} \mathrm{d} \theta\\
&-(\chi^{\prime})^2 v_p^{0,\infty} \int_y^{\infty} h_p^{0,\infty} \mathrm{d} \theta  - \chi^{\prime \prime}[(\chi v_p^{0,\infty}-v_p^{0,\infty}(0)) \int_y^{\infty} h_p^{0,\infty} \mathrm{d}\theta - (\chi g_p^{0,\infty}-g_p^{0,\infty}(0)) \int_y^{\infty} u_p^{0,\infty} \mathrm{d} \theta ]\\
& +\nu_{3}\varepsilon^{1 / 2} \chi^{\prime \prime \prime} \int_y^{\infty} h_p^{0,\infty} \mathrm{d} \theta.
\end{align*}
Finally, using the Proposition $\ref{b2.1}$ and definition $(\ref{2.19})$, we get $(\ref{2.18})$.
\end{proof}


\section{The $\varepsilon^{\frac{i}{2}}$-order ideal inner MHD profiles and MHD boundary layer  profiles $(1\leq i \leq n)$}

In this section, 
we establish estimates of any order ideal MHD profiles $(u_e^i,v_e^i,h_e^i,g_e^i,p_e^i)$
and boundary layer profiles $(u_p^i,v_p^i,h_p^i,g_p^i,p_p^i),(1\leq i \leq n)$.

\subsection{The $\varepsilon^{\frac{i}{2}}$-order ideal inner MHD profiles}

In this subsection, we estimate the $i$-th order ideal MHD profiles  $(u_e^i, v_e^i, h_e^i, g_e^i, p_e^i)(1\leq i \leq n, n\geq 3)$, which satisfy 
\begin{align}\label{3.9}
\left\{\begin{array}{l}
u_e^0 \partial_x u_e^i+v_e^i \partial_Y u_e^0-h_e^0 \partial_x h_e^i-g_e^i \partial_Y h_e^0+\partial_x p_e^i=f_1^i, \\
u_e^0 \partial_x v_e^i-h_e^0 \partial_x g_e^i+\partial_Y p_e^i=f_2^i, \\
u_e^0 \partial_x h_e^i+v_e^i \partial_Y h_e^0-h_e^0 \partial_x u_e^i-g_e^i \partial_Y u_e^0=f_3^i, \\
u_e^0 \partial_x g_e^i-h_e^0 \partial_x v_e^i=f_4^i, \\
\partial_x u_e^i+\partial_Y v_e^i=\partial_x h_e^i+\partial_Y g_e^i=0,
\end{array}\right.
\end{align}
with the following boundary conditions
\begin{align}\label{3.10}
\left\{\begin{array}{l}
(v_e^i, g_e^i)(x, 0)=-(v_p^{i-1}, g_p^{i-1})(x, 0), \quad(v_e^i, g_e^i)(x,1) =(0,0),\\
(v_e^i, g_e^i)(0, Y)=(V_{b 0}^{i}, G_{b 0}^{i})(Y),\quad(v_e^i, g_e^i)(L, Y)=(V_{b L}^{i}, G_{b L}^{i})(Y),
\end{array}\right.
\end{align}
and the compatibility conditions at the corners
\begin{align}\label{3.11}
\left\{\begin{array}{l}
(V_{b 0}^{i}, G_{b 0}^{i})(0)=-(v_p^{i-1}, g_p^{i-1})(0,0), \quad (V_{b L}^{i}, G_{b L}^{i})(0)=-(v_p^{i-1}, g_p^{i-1})(L,0),\\
(V_{b 0}^{i}, G_{b 0}^{i})(1)=(0,0), \quad (V_{b L}^{i}, G_{b L}^{i})(1)=(0,0),
\end{array}\right.
\end{align}
where
\begin{align*}
-f_1^i=&u_e^1 u_{e x}^{i-1}+u_e^{i-1} u_{e x}^{1}+v_e^1 u_{e Y}^{i-1}+ v_e^{i-1}u_{e Y}^{1}\\
&-h_e^1 h_{e x}^{i-1}-h_e^{i-1} h_{e x}^{1}-g_e^1 h_{e Y}^{i-1}- g_e^{i-1}h_{e Y}^{1}-\nu_1 \Delta u_e^{i-2},\\
-f_2^i=&v_e^{i-1}\partial_{Y}v_e^1+v_e^{1}\partial_{Y}v_e^{i-1}
       +u_e^{i-1}\partial_{x}v_e^1+u_e^{1}\partial_{x}v_e^{i-1}\\
       &-g_e^{i-1}\partial_{Y}g_e^1-g_e^{1}\partial_{Y}g_e^{i-1}
       -h_e^{i-1}\partial_{x}g_e^1-h_e^{1}\partial_{x}g_e^{i-1}-\nu_{2}\Delta v_e^{i-2},\\
-f_3^i=&u_e^1 h_{e x}^{i-1}+u_e^{i-1} h_{e x}^{1}+v_e^1 h_{e Y}^{i-1}+ v_e^{i-1}h_{e Y}^{1}\\
&-h_e^1 u_{e x}^{i-1}-h_e^{i-1} u_{e x}^{1}-g_e^1 u_{e Y}^{i-1}- g_e^{i-1}u_{e Y}^{1}-\nu_3 \Delta h_e^{i-2},\\
-f_4^i=&v_e^{i-1}\partial_{Y}g_e^1+v_e^{1}\partial_{Y}g_e^{i-1}
       +u_e^{i-1}\partial_{x}g_e^1+u_e^{1}\partial_{x}g_e^{i-1}\\
       &-g_e^{i-1}\partial_{Y}v_e^1-g_e^{1}\partial_{Y}v_e^{i-1}
       -h_e^{i-1}\partial_{x}v_e^1-h_e^{1}\partial_{x}v_e^{i-1}-\nu_{4}\Delta g_e^{i-2}.
\end{align*}
To solve the problem $(\ref{3.9})-(\ref{3.11})$,
using the divergence free conditions of the velocity and magnetic fields, we rewrite the third and the fourth equations of $(\ref{3.9})$ respectively, and obtain
{{
$$
\partial_Y(u_e^0 g_e^i-h_e^0 v_e^i)=-f_{3}^i,
$$
$$
\partial_x(u_e^0 g_e^i-h_e^0 v_e^i)=f_{4}^i.
$$
Integrating the above two equalities respectively yields
\begin{align*}
v_e^i=\frac{u_e^0}{h_e^0} g_e^i+\frac{u_e \overline{g_p^{i-1}}-h_e \overline{v_p^{i-1}}-u_e^0 G_{b0}^{i}+h_e^0V_{b0}^{i}}{2 h_e^0}+\frac{\int_0^Y f_3^i(x,\theta)\mathrm{d} \theta-\int_0^x f_4^i(s,Y)\mathrm{d} s}{2 h_e^0}:=\frac{u_e^0}{h_e^0} g_e^i+\frac{b_{i}^{1}(x, Y)}{h_e^0},
\end{align*}
and
\begin{align*}
g_e^i=\frac{h_e^0}{u_e^0} v_e^i+\frac{h_e \overline{v_p^{i-1}}-u_e \overline{g_p^{i-1}}+u_e^0 G_{b0}^{i}-h_e^0V_{b0}^{i}}{2 u_e^0}-\frac{\int_0^Y f_3^i(x,\theta)\mathrm{d} \theta-\int_0^x f_4^i(s,Y)\mathrm{d} s}{2 h_e^0}:=\frac{h_e^0}{u_e^0} v_e^i+\frac{b_{i}^{2}(x, Y)}{u_e^0},
\end{align*}
where
$$
b_{i}^{1}(x,Y)=\frac{1}{2}(u_e \overline{g_p^{i-1}}-h_e \overline{v_p^{i-1}}-u_e^0 G_{b0}^{i}+h_e^0V_{b0}^{i}+\int_0^Y f_3^i(x,\theta)\mathrm{d} \theta-\int_0^x f_4^i(s,Y)\mathrm{d} s)
$$
and $b_{i}^{2}(x,Y)=-b_{i}^{1}(x,Y)$. Using the divergence free conditions, it follows from the first and second equations in $(\ref{3.9})$ that
\begin{align}\label{3.9.1}
-u_e^0 \Delta v_e^i+\partial_Y^2 u_e^0 \cdot v_e^i+(h_e^0 \Delta g_e^i-\partial_Y^2 h_e^0 \cdot g_e^i)=f_{1 Y}^i-f_{2 x}^i,
\end{align}
which satisfies the boundary conditions $(\ref{3.10})$.

When $i=1$, we have $f_1^1=f_2^1=f_3^1=f_4^1=0$,  resulting in the elimination of the terms containing $f_j^1(j=1,2,3,4)$ in $(\ref{3.9})~\text{and}~ (\ref{3.9.1})$. Now, we consider the $\sqrt{\varepsilon}$-order ideal inner MHD flow. 

\subsubsection{The first order ideal inner MHD profile}

Let us rewrite $(\ref{3.9.1})$ in the following form
\begin{align*}
-u_e^0 \Delta v_e^1+\partial_Y^2 u_e^0 \cdot v_e^1+(h_e^0 \Delta g_e^1-\partial_Y^2 h_e^0 \cdot g_e^1)=0.
\end{align*}
To avoid the singularity at the corners of $\Omega_{1}$, we consider the following elliptic problem 
\begin{align}\label{3.14}
-u_e^0 \Delta v_e^1+\partial_Y^2 u_e^0 \cdot v_e^1+(h_e^0 \Delta g_e^1-\partial_Y^2 h_e^0 \cdot g_e^1)=E_b,
\end{align}
which satisfies the boundary conditions $(\ref{3.10})$.  Now, we construct $E_b$. In order to do this, we first introduce
\begin{align}\label{3.15}
\left\{\begin{array}{l}
B_v(x, Y):=\left(1-\frac{x}{L}\right) \frac{V_{b 0}^{1}(Y)}{v_p^0(0,0)} v_p^0(x, 0)+\frac{x}{L} \frac{V_{b L}^{1}(Y)}{v_p^0(L, 0)} v_p^0(x, 0), \\
B_g(x, Y):=\left(1-\frac{x}{L}\right) \frac{G_{b 0}^{1}(Y)}{g_p^0(0,0)} g_p^0(x, 0)+\frac{x}{L} \frac{G_{b L}^{1}(Y)}{g_p^0(L, 0)} g_p^0(x, 0),
\end{array}\right.
\end{align}
where $v_p^0(0,0), v_p^0(L, 0), g_p^0(0,0) ~\text{and}~ g_p^0(L, 0)$ are nonzero. When $v_p^0(0,0)=0$,  
we replace $\frac{V_{b 0}^{1}(Y)}{v_p^0(0,0)} v_p^0(x, 0)$ 
by $V_{b 0}^{1}(Y)-v_p^0(x, 0)(1-Y)$. 
The same argument can be applied when the other three terms $v_p^0(L, 0), g_p^0(0,0) ~\text{and}~ g_p^0(L, 0)$ are equal to zero respectively.
It is easy to show that $B_v (x,Y)$ and $B_g(x,Y)$ satisfy the boundary conditions $(\ref{3.10})$.
Then, we introduce the smooth function $F_e$
\begin{align}\label{3.16}
-u_e^0 \Delta B_v+\partial_Y^2 u_e^0 \cdot B_v+(h_e^0 \Delta B_g-\partial_Y^2 h_e^0 \cdot B_g)=F_e.
\end{align}
As $\left|\partial_{Y}^k\left(V_{bL}^{1}(Y)-V_{b0}^{1}(Y), G_{bL}^{1}(Y)-G_{b0}^{1}(Y)\right)\right|\leq CL$, we obtain $B_{v}, B_{g} \in W^{k,p}(\Omega_{1})$.
Therefore, for any $k \geq 0$, $p \in[1, \infty]$, we have
$$
\|F_e\|_{W^{k, p}(\Omega_1)} \leq C,
$$
where $C$ is a positive constant.

Let $\chi(\cdot)$ be a smooth function supported in $[0,1]$ and let us take
$$
E_b=\chi \left(\frac{Y}{\varepsilon^{\frac{3}{16(n-1)}}}\right)F_e(x,0).
$$
It follows that
$$\|\partial_Y^k E_b\|_{L^{p}(\Omega_{1})} \lesssim \varepsilon^{-\frac{3k}{16(n-1)}+\frac{3}{16(n-1)p}}.$$
 Let
\begin{align}\label{3.17.1}
v_e^1=B_v+\omega_1, \quad g_e^1=B_g+\omega_2.
\end{align}
Using $(\ref{3.14})$,  we obtain the following system
\begin{align}\label{3.17}
\left\{\begin{array}{l}
-u_e^0 \Delta \omega_1+\partial_Y^2 u_e^0 \cdot \omega_1+(h_e^0 \Delta \omega_2-\partial_Y^2 h_e^0 \cdot \omega_2)=E_b-F_e, \\
h_e^0\omega_1-u_e^0\omega_2=u_e^0 B_g-h_e^0 B_v+b_{1}^{1}, \\
\left.\omega_i\right|_{\partial \Omega_1}=0, \quad i=1,2.
\end{array}\right.
\end{align}
For the system $(\ref{3.17})$, we  get the following proposition.
\begin{pro}\label{b3.1}
Under the assumptions of Theorem $\ref{b1.1}$, suppose that, for any $k \geq 0,~ p\geq 1$, $F_e(x,Y) \in W^{k,p}(\Omega_1)$, and
$$
\left|\partial_{Y}^k (V_{bL}^{1}(Y)-V_{b0}^{1}(Y), G_{bL}^{1}(Y)-G_{b0}^{1}(Y)) \right|\leq CL.
$$
Then, there exists a unique smooth solution $(\omega_1, \omega_2)$ of the boundary value problem $(\ref{3.17})$, satisfying
\begin{align}\label{3.18}
& \|(\omega_1, \omega_2)\|_{L^{\infty}(\Omega_1)}
+\|(\omega_1, \omega_2)\|_{H^2(\Omega_1)} \leq C, \nonumber\\
& \|(\omega_1, \omega_2)\|_{H^{2+j}(\Omega_1)}\leq C \varepsilon^{-\frac{3j}{16(n-1)}+\frac{3}{32(n-1)}}, \quad j=1,2,\ldots,2n,
\end{align}
where the constant $C>0$ only depends on the given boundary data. 
Moreover, it holds that
\begin{align}\label{3.19}
& \|(\omega_1, \omega_2)\|_{W^{2, q}(\Omega_1)} \leq C, \nonumber\\
& \|(\omega_1, \omega_2)\|_{W^{2+j, q}(\Omega_1)} \leq C \varepsilon^{-\frac{3j}{16(n-1)}+\frac{3}{16(n-1)q}},
\end{align}
for $q \in (1, \infty)$.
\end{pro}

\begin{proof}
Under assumption  (H),  the system $(\ref{3.17})$ is non-degenerate.
%
Next, we prove these estimates of the solution for the system $(\ref{3.17})$. The proof is divided into  three steps.

\subsubsection*{Step 1: $H^1$ estimates of $\omega_{i}(i=1,2)$}

Multiplying $(\ref{3.17})_{1}$ by $\frac{\omega_1}{u_e^0}$ and integrating by parts, we get
\begin{align*}
& \iint_{\Omega_{1}}\left|\nabla \omega_1\right|^2 \mathrm{d}x \mathrm{d}Y+\iint_{\Omega_{1}} \frac{\partial_Y^2 u_e^0}{u_e^0}\left|\omega_1\right|^2 \mathrm{d}x \mathrm{d}Y\\
& =\iint_{\Omega_{1}} \frac{h_e^0}{u_e^0} \nabla \omega_1 \cdot \nabla \omega_2 \mathrm{d}x \mathrm{d}Y+\iint_{\Omega_{1}} \partial_Y\left(\frac{h_e^0}{u_e^0}\right) \partial_Y \omega_2 \cdot \omega_1 \mathrm{d}x \mathrm{d}Y\\
& \quad+\iint_{\Omega_{1}} \frac{\partial_Y^2 h_e^0}{u_e^0} \omega_1 \cdot \omega_2 \mathrm{d}x \mathrm{d}Y
+\iint_{\Omega_{1}} \frac{(E_b-F_e)}{u_e^0} \omega_1 \mathrm{d}x \mathrm{d}Y := \sum_{i=1}^4 I_{i} .
\end{align*}
For $I_{i}$ ($i=1,2,3,4$), we obtain
$$
\begin{aligned}
\left|I_{1}\right| & =\left|\iint_{\Omega_{1}} \frac{h_e^0}{u_e^0} \nabla \omega_2 \cdot \nabla \omega_1 \mathrm{d}x \mathrm{d}Y \right| \lesssim \left\|\frac{h_e^0}{u_e^0}\right\|_{L^{\infty}}\|\nabla \omega_1\|_{L^2(\Omega_1)}\|\nabla \omega_2\|_{L^2(\Omega_1)}, \\
\left|I_{2}\right| & = \left|\iint_{\Omega_{1}} \partial_Y\left(\frac{h_e^0}{u_e^0}\right) \partial_Y \omega_2 \cdot \omega_1 \mathrm{d}x \mathrm{d}Y\right|\\
& \lesssim \left(\left\|\frac{h_e^0}{u_e^0}\right\|_{L^{\infty}}
\left\|\partial_{Y}{u_e^0}\right\|_{L^{\infty}}
+\left\|\partial_{Y}{h_e^0}\right\|_{L^{\infty}}\right)\|\omega_1\|_{L^2(\Omega_1)}
\|\partial_Y \omega_2\|_{L^2(\Omega_1)} \\
& \lesssim L\left(\left\|\frac{h_e^0}{u_e^0}\right\|_{L^{\infty}}
\left\|\partial_{Y}{u_e^0}\right\|_{L^{\infty}}
+\left\|\partial_{Y}{h_e^0}\right\|_{L^{\infty}}\right)\|\partial_x \omega_1\|_{L^2(\Omega_1)}\|\partial_Y \omega_2\|_{L^2(\Omega_1)},\\
\left|I_{3}\right| & =\left|\iint_{\Omega_{1}} \frac{\partial_Y^2 h_e^0}{u_e^0} \omega_1 \cdot \omega_2 \mathrm{d}x \mathrm{d}Y\right| \\
& \lesssim \|\omega_1\|_{L^2(\Omega_1)}
 \|\omega_2\|_{L^2(\Omega_1)} \lesssim L^2 \|\partial_x \omega_1\|_{L^2(\Omega_1)}\|\partial_x \omega_2\|_{L^2(\Omega_1)},\\
\left|I_{4}\right|&=\left|\iint_{\Omega_{1}} \frac{E_b-F_e}{u_e^0} \omega_1 \mathrm{d}x \mathrm{d}Y\right| \lesssim L \|\partial_x \omega_1\|_{L^2(\Omega_1)}\|E_b-F_e\|_{L^2(\Omega_1)}
\lesssim L \|\partial_{x}\omega_{1}\|_{L^2({\Omega_{1}})}.
\end{aligned}
$$
Combining all the above estimates, we get
\begin{align}\label{3.20.2}
&\iint_{\Omega_{1}} \left|\nabla \omega_1\right|^2 \mathrm{d}x \mathrm{d}Y+\iint_{\Omega_{1}} \frac{\partial_{Y}^{2}u_e^0}{u_e^0}\left|\omega_1\right|^2 \mathrm{d}x \mathrm{d}Y \nonumber\\
\lesssim & \left(\left\|\frac{h_e^0}{u_e^0}\right\|_{L^{\infty}}
+\left\|\partial_{Y}{h_e^0}\right\|_{L^{\infty}}\right)\|\nabla \omega_1\|_{L^2(\Omega_1)}\|\nabla \omega_2\|_{L^2(\Omega_1)}
+L \|\nabla \omega_1\|_{L^2(\Omega_1)}.
\end{align}
Applying the arguments of  \cite{GN2017} (Page $19$), we have
\begin{align*}
\iint_{\Omega_{1}}(\left|\partial_Y \omega_1\right|^2+\frac{\partial_Y^2 u_e^0}{u_e^0}\left|\omega_1\right|^2)\mathrm{d}x \mathrm{d}Y=\iint_{\Omega_{1}}\left|u_e^0\right|^2
\left|\partial_Y\left(\frac{\omega_1}{u_e^0}\right)\right|^2 \mathrm{d}x \mathrm{d}Y \geq \beta_0 \iint_{\Omega_{1}}\left|\partial_Y \omega_1\right|^2 \mathrm{d}x \mathrm{d}Y.
\end{align*}
From the second equation of $(\ref{3.17})$, we obtain
$$
\nabla \omega_2
=\left(\begin{array}{c}
\frac{h_e^0}{u_e^0} \partial_x \omega_1 \\
\frac{h_e^0}{u_e^0} \partial_Y \omega_1
+\partial_Y\left(\frac{h_e^0}{u_e^0}\right) \omega_1
\end{array}\right)+\nabla\left(\frac{h_e^0}{u_e^0} B_v+\frac{b_{1}^{1}}{u_e^0}-B_g\right) .
$$
Thus, we have
\begin{align}\label{3.221}
\begin{aligned}
 \|\nabla \omega_2\|_{L^2(\Omega_1)}
& \lesssim \left\|\nabla \left(\frac{h_e^0}{u_e^0}B_{v}+\frac{b_{1}^{1}}{u_e^0}-B_g\right)\right\|_{L^2(\Omega_1)}
 +\left\|\frac{h_e^0}{u_e^0}\right\|_{L^{\infty}}\|\nabla\omega_1\|_{L^2(\Omega_1)}\\
& \quad \quad +L\left( \left\|\frac{h_e^0}{u_e^0}\right\|_{L^{\infty}}+\left\|\partial_Y h_e^0\right\|_{L^{\infty}}\right)\|\partial_x\omega_1\|_{L^2(\Omega_1)}\\
&\lesssim 1+\left\|\frac{h_e^0}{u_e^0}\right\|_{L^{\infty}}\|\nabla \omega_1\|_{L^2(\Omega_1)}+L\left(\left\|\frac{h_e^0}{u_e^0}\right\|_{L^{\infty}}+\left\|\partial_Y h_e^0\right\|_{L^{\infty}}\right)\|\partial_x \omega_1\|_{L^2(\Omega_1)}.
\end{aligned}
\end{align}
Combining this estimate with $(\ref{3.20.2})$ and using Young's inequality, we obtain
$$
\|\omega_1\|_{H^1(\Omega_{1})} \leq C,
$$
in which we have used the assumption (H). 
In addition, using the above estimate and $(\ref{3.221})$, we get
$$
\|\omega_2\|_{H^1(\Omega_{1})} \leq C .
$$

\subsubsection*{Step 2: $H^2$ and $L^{\infty}$ estimates of $\omega_{i}(i=1,2)$}

Note that $(\ref{3.17})$ may be rewritten
\begin{align}\label{3.21.1}
\left\{
\begin{array}{l}
-\Delta \omega_1+\frac{h_e^0}{u_e^0} \Delta \omega_1=M_e^{1},\\
\left. \omega_i \right|_{\partial{\Omega_1}}=0, \quad i=1,2,
\end{array}\right.
\end{align}
where
$
M_e^{1}:=\frac{1}{u_e^0}(E_b-F_e-\partial_Y^2 u_e^0 \cdot \omega_1+\partial_Y^2 h_e^0 \cdot \omega_2) .
$
We then have $\|M_e^{1}\|_{L^2(\Omega_1)} \leq C$. On the boundary $\left\{Y=0\right\}$, $M_{e}^{1}(x,0)=0$ and
\begin{align}\label{3.21.1.1}
-\partial_{Y}^{2}\omega_{1}+\frac{h_{e}^{0}(0)}{u_{e}^{0}} \partial_{Y}^{2}\omega_{1}=0.
\end{align}
Applying $\partial_Y$ to the first equation in $(\ref{3.21.1})$, then multiplying the resulted equation by $\partial_{Y} \omega_{1}$ and  integrating by parts, we have
\begin{align*}
&\iint_{\Omega_{1}} \left|\nabla \partial_{Y}\omega_{1}\right|^{2} \mathrm{d}x \mathrm{d}Y+\int_{0}^{L} \left(-\partial_{Y}^{2}\omega_{1}+\frac{h_{e}}{u_{e}}\partial_{Y}^{2} \omega_{2}\right)\partial_{Y}\omega_{1}(x,0)\mathrm{d}x \\
& -\int_{0}^{1}\left.\partial_{xY}\omega_{1} \cdot \partial_{Y}\omega_{1}\right|_{x=0}^{x=L} \mathrm{d}Y+\int_{0}^{1} \left.\frac{h_{e}^{0}}{u_{e}^{0}}\partial_{xY}\omega_{2} \cdot \partial_{Y}\omega_{1}\right|_{x=0}^{x=L} \mathrm{d}Y \\
=&\iint_{\Omega_{1}}\frac{h_{e}^{0}}{u_{e}^{0}}\nabla\partial_{Y}\omega_{1} \cdot \nabla\partial_{Y}\omega_{2}\mathrm{d}x \mathrm{d}Y+\iint_{\Omega_{1}} \partial_{Y}\left(\frac{h_{e}^{0}}{u_{e}^{0}}\right)\partial_{Y}^{2}\omega_{2} \cdot \partial_{Y}\omega_{1} \mathrm{d}x \mathrm{d}Y\\
& -\int_{0}^{L}\left.\partial_{Y}\left(\frac{h_{e}^{0}}{u_{e}^{0}}\right)
\partial_{Y}\omega_{1} \cdot \partial_{Y}\omega_{2}\right|_{Y=0}^{Y=1} \mathrm{d}x +\iint_{\Omega_{1}} \partial_{Y}\left(\frac{h_{e}^{0}}{u_{e}^{0}}\right)\nabla \omega_{2}\cdot \nabla \partial_{Y}\omega_{1}\mathrm{d}x \mathrm{d}Y\\
&+\iint_{\Omega_{1}}\partial_{Y}^{2}\left(\frac{h_{e}^{0}}{u_{e}^{0}}\right)\partial_{Y}\omega_{1}
\cdot \partial_{Y}\omega_{2}\mathrm{d}x \mathrm{d}Y-\iint_{\Omega_{1}} M_{e}^{1}\partial_{Y}^{2}\omega_{1}\mathrm{d}x \mathrm{d}Y.
\end{align*}
The last three terms on the left-hand side of the above equality vanish thanks to the boundary conditions $(\ref{3.21.1})_{2}$ and $(\ref{3.21.1.1})$. 
For the terms on the right-hand side, we obtain the following estimates
\begin{align*}
\left|\iint_{\Omega_{1}} \frac{h_{e}^{0}}{u_{e}^{0}}\nabla \partial_{Y}\omega_{1} \cdot \nabla\partial_{Y}\omega_{2}\mathrm{d}x \mathrm{d}Y \right|
& \lesssim \left\|\frac{h_{e}^{0}}{u_{e}^{0}}\right\|_{L^{\infty}(\Omega_{1})} \|\nabla\partial_{Y}\omega_{1}\|_{L^{2}(\Omega_{1})}\|\nabla \partial_{Y}\omega_{2}\|_{L^{2}(\Omega_{1})},\\
\left|\iint_{\Omega_{1}} \partial_{Y}\left(\frac{h_{e}^{0}}{u_{e}^{0}}\right)
\partial_{Y}^{2}\omega_{2} \cdot \partial_{Y}\omega_{1}\mathrm{d}x \mathrm{d}Y \right|
&\lesssim\left\|\partial_{Y}\left(\frac{h_{e}^{0}}{u_{e}^{0}}\right)\right\|_{L^{\infty}\left(\Omega_{1}\right)}
\|\nabla\partial_{Y}\omega_{2}\|_{L^{2}(\Omega_{1})}
\| \partial_{Y}\omega_{1}\|_{L^{2}(\Omega_{1})} \\
& \lesssim \|\nabla \partial_{Y}\omega_{2}\|_{L^{2}(\Omega_{1})}
\|\omega_{1}\|_{H^{1}(\Omega_{1})},\\
\left|\iint_{\Omega_{1}}\partial_{Y}\left(\frac{h_{e}^{0}}{u_{e}^{0}}\right)\nabla \omega_{2}\cdot \nabla \partial_{Y}\omega_{1} \mathrm{d}x \mathrm{d}Y \right|
&\lesssim \|\nabla\partial_{Y}\omega_{1}\|_{L^{2}(\Omega_{1})}
\|\omega_{2}\|_{H^{1}(\Omega_{1})},\\
\left|\iint_{\Omega_{1}}\partial_{Y}^{2}\left(\frac{h_{e}^{0}}{u_{e}^{0}}\right)
\partial_{Y}\omega_{1}\cdot \partial_{Y}\omega_{2}\mathrm{d}x \mathrm{d}Y \right|
&\lesssim\|\omega_{1}\|_{H^{1}(\Omega_{1})}\|\omega_{2}\|_{H^{1}(\Omega_{1})},\\
\left|-\iint_{\Omega_{1}} M_{e}^{1}\partial_{Y}^{2}\omega_{1}\mathrm{d}x \mathrm{d}Y \right|
&\lesssim \|M_{e}^{1}\|_{L^{2}(\Omega_{1})}\|\nabla \partial_{Y} \omega_{1}\|_{L^{2}(\Omega_{1})},\\
\left|\int_{0}^{L}\left.\partial_{Y}\left(\frac{h_{e}^{0}}{u_{e}^{0}}\right)
\partial_{Y}\omega_{1} \cdot \partial_{Y}\omega_{2}\right|_{Y=0}^{Y=1} \mathrm{d}x\right|
&\lesssim \left\|\partial_{Y}\left(\partial_{Y}\left(\frac{h_{e}^{0}}{u_{e}^{0}}\right)
\partial_{Y}\omega_{1}\right)\right\|_{L^{2}(\Omega_{1})}
\|\partial_{Y}\omega_{2}\|_{L^{2}(\Omega_{1})}\\
& \quad+\left\|\partial_{Y}\left(\frac{h_{e}^{0}}{u_{e}^{0}}\right)
\partial_{Y}\omega_{1}\right\|_{L^{2}(\Omega_{1})}
\|\partial_{Y}^{2}\omega_{2}\|_{L^{2}(\Omega_{1})}\\
&\lesssim  \|\nabla\partial_{Y}\omega_{1}\|_{L^{2}(\Omega_{1})}
\|\omega_{2}\|_{H^{1}(\Omega_{1})}
+\|\omega_{1}\|_{H^{1}(\Omega_{1})}\|\nabla \partial_{Y}\omega_{2}\|_{L^{2}(\Omega_{1})}.
\end{align*}
Combining all the above inequalities, we get
\begin{align}\label{3.21.1.2}
\|\nabla \partial_{Y}\omega_{1}\|_{L^{2}(\Omega_{1})}^{2}
\lesssim &
\left\|\frac{h_{e}^{0}}{u_{e}^{0}}\right\|_{L^{\infty}(\Omega_{1})} \|\nabla \partial_{Y}\omega_{1}\|_{L^{2}(\Omega_{1})}\|\nabla \partial_{Y}\omega_{2}\|_{L^{2}(\Omega_{1})}\nonumber\\
&+\|\nabla\partial_{Y}\omega_{2}\|_{L^{2}(\Omega_{1})}
\|\omega_{1}\|_{H^{1}(\Omega_{1})}
+\|\omega_{2}\|_{H^{1}(\Omega_{1})}\|\nabla \partial_{Y}\omega_{1}\|_{L^{2}(\Omega_{1})} \nonumber\\
&+\|\omega_{1}\|_{H^{1}(\Omega_{1})}\|\omega_{2}\|_{H^{1}(\Omega_{1})}
+\|M_{e}^{1}\|_{L^{2}(\Omega_{1})}\|\nabla\partial_{Y}\omega_{1}\|_{L^{2}(\Omega_{1})}.
\end{align}
For the $H^2$ norm of $\omega_2$, it follows from the second equation of $(\ref{3.17})$ that
\begin{align}\label{3.20.3}
\omega_{2}=\frac{h_e^0}{u_e^0}(B_v+\omega_1)-\frac{b_1^1}{u_e^0}-B_g.
\end{align}
Applying $\nabla \partial_{Y}$ to $(\ref{3.20.3})$, we have
$$
\nabla \partial_Y \omega_{2}=\left(\begin{array}{c}
\partial_{Y}\left(\frac{h_e^0}{u_e^0}\right) \partial_x \omega_1 +\frac{h_e^0}{u_e^0} \partial_{xY} \omega_1 \\
2 \partial_{Y} \left(\frac{h_e^0}{u_e^0}\right) \partial_Y \omega_1+\partial_Y^2 \left(\frac{h_e^0}{u_e^0}\right) \omega_1+\frac{h_e^0}{u_e^0}\partial_{Y}^2 \omega_1
\end{array}\right)+\nabla \partial_{Y}\left(\frac{h_e^0}{u_e^0} B_v-\frac{b_{1}^{1}}{u_e^0}-B_g\right) .
$$
Hence, we get the following estimate
\begin{align*}
\|\nabla \partial_{Y} \omega_2\|_{L^2(\Omega_1)}
& \lesssim \left\|\nabla \partial_{Y}\left(\frac{h_e^0}{u_e^0}B_v\right)\right\|_{L^2(\Omega_1)}
+\left\|\nabla \partial_{Y}\left(\frac{b_{1}^{1}}{u_e^0}\right)\right\|_{L^2(\Omega_1)}
+\|\nabla \partial_{Y}B_g\|_{L^2(\Omega_1)}\\
& \quad +\left\|\frac{h_e^0}{u_e^0}\right\|_{L^{\infty}(\Omega_{1})}\|\nabla \partial_{Y}\omega_1\|_{L^2(\Omega_1)}
+\left(\left\|\partial_{Y} \left(\frac{h_e^0}{u_e^0}\right)\right\|_{L^{\infty}(\Omega_{1})}+\left\|\partial_{Y}^2
\left(\frac{h_e^0}{u_e^0}\right)\right\|_{L^{\infty}(\Omega_{1})}\right)\|\omega_1\|_{H^1(\Omega_1)}\\
&  \lesssim 1+\left\|\frac{h_e^0}{u_e^0}\right\|_{L^{\infty}(\Omega_{1})}\|\nabla \partial_{Y}\omega_1\|_{L^2(\Omega_1)}+\| \omega_1\|_{H^1(\Omega_1)},
\end{align*}
which together with $(\ref{3.21.1.2})$, implies that
$$
\|\nabla \partial_{Y} \omega_{i}\|_{L^2(\Omega_1)} \leq C, \quad i=1,2,
$$
where we have used the $H^{1}$ estimates of $\omega_{i}(i=1,2)$ and the assumption (H).

For $\partial_{xx} \omega_i, ~(i=1,2)$,  using $(\ref{3.21.1})$ and taking $\partial_{xx}$ to $(\ref{3.20.3})$, we get the $L^2$ estimates of $\partial_{xx}\omega_i$.
 For the $L^{\infty}$ estimates of $\omega_{i}$, $(i=1,2)$, we have the the following estimate
 \begin{align*}
\left|\omega_i(x, Y)\right| & \leq \int_0^x\left|\partial_x \omega_i(s, Y)\right| \mathrm{d} s \\
& \lesssim \int_0^x\left(\int_0^Y\left|\partial_x \omega_i \partial_{x Y} \omega_i\right|(s, \theta) \mathrm{d} \theta\right)^{1 / 2} \mathrm{d} s \\
& \lesssim \sqrt{x}\left\|\partial_x \omega_i\right\|_{L^2(\Omega_1)}^{\frac{1}{2}}\left\|\partial_{x Y} \omega_i\right\|_{L^2(\Omega_1)}^{\frac{1}{2}} \lesssim \sqrt{L}.
\end{align*}

\subsubsection*{Step 3: Higher order estimates of $\omega_{i}(i=1,2)$}

Applying  $\partial_{Y}^{k}$ ${(k=1,2)}$ to the first equation in $(\ref{3.21.1})$, we obtain
\begin{align}\label{3.22.1}
\left\{\begin{array}{l}
-\Delta \partial_Y^{k} \omega_1+\partial_{Y}^{k}(\frac{h_e^0}{u_e^0} \Delta \omega_2)=\partial_Y^{k} M_e^{1},\\ 
\omega_2=\frac{h_e^0}{u_e^0}(\omega_1+B_v)-\frac{b_{1}^{1}}{u_e^0}-B_g, \\
\left.\partial_Y^{k} \omega_i\right|_{x=0, L}=\left.\partial_Y^2 \omega_i\right|_{Y=0}=0, \quad \left.\partial_{Y}^{2}\omega_i\right|_{Y=1}=0.
\end{array}\right.
\end{align}
Based on the $H^2$ bounds of $\omega_i$ and the estimates of $E_b, F_e$, we have
$$
\|\partial_{Y}^{k}\omega_i\|_{H^2(\Omega_{1})} \leq C\varepsilon^{-\frac{3k}{16(n-1)}+\frac{3}{32(n-1)}},\quad k=1,2.
$$
For the estimates of $\partial_{x}^{3} \omega_i$ in $L^{2}~\text{and}~ H^{1}$ norms, we take the $x$-derivative of the first equation in $(\ref{3.21.1})$,
$$
\begin{aligned}
-\partial_x^3 \omega_1+\frac{h_e^0}{u_e^0} \partial_x^3 \omega_2
& =\partial_{x Y Y} \omega_1-\frac{h_e^0}{u_e^0} \partial_{x Y Y} \omega_2+(-\Delta \partial_x \omega_1+\frac{h_e^0}{u_e^0} \Delta \partial_x \omega_2) \\
& =\partial_{x Y Y} \omega_1-\frac{h_e^0}{u_e^0} \partial_{x Y Y} \omega_2+\partial_x M_e^{1},
\end{aligned}
$$
and apply $\partial_{x}^3$ to the second equation in $(\ref{3.22.1})$, we can  obtain the $L^2$ and $H^1$ estimates of $\partial_x^3 \omega_i$ $(i=1,2)$. 
Similarly, we can establish the estimates of $H^5$ to $H^{2(n+1)}(n\geq 3)$ on $\omega_i$ $(i=1,2)$.


By the standard elliptic theory \cite{GT1983}, this leads to $W^{k,p}$ estimates of $\omega_{i}(i=1,2)$. This ends the proof of the Proposition.
\end{proof}

Next, we derive the estimates of $v_e^1, g_e^1$. By their definition $(\ref{3.17.1})$ and boundary conditions $(\ref{3.10})$, we obtain that $(v_e^1, g_e^1) \in W^{k,p}(\Omega_1)$ is the unique smooth solution to the system $(\ref{3.14})$ with boundary conditions $(\ref{3.10})$.
It follows from $B_v, B_g \in W^{2(n+1),p}(\Omega_1)$ that
$$
\|(v_e^1, g_e^1)\|_{L^{\infty}(\Omega_{1})}
+\|(v_e^1,g_e^1)\|_{W^{2,q}(\Omega_{1})} \leq C,
\quad \|(v_e^1, g_e^1)\|_{W^{2+j,q}(\Omega_{1})} \leq C \varepsilon^{-\frac{3j}{16(n-1)}+\frac{3}{16(n-1)q}}, \quad j=1,2,\ldots,2n.
$$
Now, we construct the first order ideal MHD correctors $u_e^1, h_e^1 ~\text{and}~ p_e^1$. Using $(\ref{3.9})$ and the divergence free conditions, we have
\begin{align}\label{3.24}
& u_e^1(x, Y)=u_b^1(Y)-\int_0^x v_{e Y}^1(s, Y) \mathrm{d} s, \nonumber\\
& h_e^1(x, Y)=h_b^1(Y)-\int_0^x g_{e Y}^1(s, Y) \mathrm{d} s, \\
& p_e^1(x, y)=\int_Y^1 (u_e^0 v_{e x}^1-h_e^0 g_{e x}^1)(x, \theta) \mathrm{d} \theta-\int_0^x ( u_e^0(1) v_{e Y}^1(s, 1)-h_e^0(1) g_{e Y}^1(s, 1)) \mathrm{d} s,\nonumber
\end{align}
where $u_e^1(0, Y)=u_b^1(Y) ~\text{and}~ h_e^1(0, Y)=h_b^1(Y)$ satisfy $\partial_Y u_b^1(1)=\partial_Y h_b^1(1)=0$. Equation $(\ref{3.14})$ implies that $v_{e Y Y}^{1}(x,1)=g_{e Y Y}^{1}(x,1)=0$, thus, $u_{e Y}^1(x, 1)=h_{e Y}^1(x, 1)=0$.
By  definition of $(u_e^1, h_e^1)$ in $(\ref{3.24})$
and Proposition $\ref{b3.1}$, we get
\begin{align}\label{3.26}
\|(u_e^1,h_e^1)\|_{L^{\infty}(\Omega_{1})} \leq C, \quad
\|(u_e^1, h_e^1)\|_{H^{1+j}(\Omega_{1})} \leq C \varepsilon^{-\frac{3j}{16(n-1)}+\frac{3}{32(n-1)}}, \quad j=1,2,\ldots,2n.
\end{align}

\subsubsection{The $i$-th order ideal inner MHD profiles $(2\leq i \leq n)$}

Combining $v_e^i, g_e^i$ in Section $3.1$ with $(\ref{3.9.1})$, we consider the following system
\begin{align}\label{5.12}
\left\{\begin{array}{l}
-u_e^0 \Delta v_e^i+\partial_Y^2 u_e^0 \cdot v_e^i+(h_e^0 \Delta g_e^i-\partial_Y^2 h_e^0 \cdot g_e^i)=f_{1 Y}^i-f_{2 x}^i, \\
h_e^0 v_e^i-u_e^0 g_e^i=b_i^{1}, \\
(v_e^i,g_e^i)(x,0)=-(v_p^{i-1},g_p^{i-1})(x,0),\quad
(v_e^i, g_e^i)(x,1)=(0,0).
\end{array}\right.
\end{align}
For the system $(\ref{5.12})$, we establish the estimates of $(v_e^i, g_e^i)$ as follows.
\begin{pro}\label{b5.1}
Suppose that the assumption (H) 
is satisfied.
Then, for $2\leq i \leq n$, there exists a unique smooth solution $(v_e^i, g_e^i)$ to the boundary value problems $(\ref{5.12})$, satisfying
\begin{align}\label{5.13}
& \|(v_e^i, g_e^i)\|_{L^{\infty}(\Omega_1)}
+\|(v_e^i, g_e^i)\|_{H^2(\Omega_1)} \leq C \varepsilon^{-\frac{3(i-1)}{16(n-1)}+\frac{3}{32(n-1)}}, \nonumber\\
& \|(v_e^i, g_e^i)\|_{H^{2+j}(\Omega_1)}\leq C \varepsilon^{-\frac{3(i-1+j)}{16(n-1)}+\frac{3}{32(n-1)}}, 
\quad j=1,2,\ldots,2(n-i+1),
\end{align}
where the constant $C>0$ only depends on the given boundary data. Moreover, it holds that
\begin{align}\label{5.14}
& \|(v_e^i, g_e^i)\|_{W^{2, q}(\Omega_1)} \leq C(L)\varepsilon^{-\frac{3(i-1)}{16(n-1)}+\frac{3}{16(n-1)q}}, \nonumber\\
& \|(v_e^i, g_e^i)\|_{W^{2+j, q}(\Omega_1)} \leq C(L) \varepsilon^{-\frac{3(i-1+j)}{16(n-1)}+\frac{3}{16(n-1)q}},
\end{align}
for $q \in[1, \infty)$, $C(L)>0$.
\end{pro}
\begin{proof}
The energy estimates on $(v_e^i,g_e^i)$ under assumption (H).

Using the methods mentioned in Section $3.1.1$, we obtain
$$
\|(v_e^i,g_e^i)\|_{H^1(\Omega_{1})} \leq C \varepsilon^{-\frac{3(i-1)}{16(n-1)}+\frac{3}{32(n-1)}}.
$$
Now, we rewrite the equations $(\ref{5.12})$ in the following form
\begin{align}\label{5.21.1}
\left\{
\begin{array}{l}
-\Delta v_e^i+\frac{h_e^0}{u_e^0} \Delta g_e^i=G_e^i,\\
\left. (v_e^i, g_e^i)
\right|_{\partial{\Omega_1}}=0,
\end{array}\right.
\end{align}
where
$
M_e^i:=\frac{1}{u_e^0}(f_{1Y}^i-f_{2x}^i-\partial_Y^2 u_e^0 \cdot v_e^i+\partial_Y^2 h_e^0 \cdot g_e^i) .
$
Clearly,
 $$\|M_e^{i}\|_{L^2(\Omega_1)} \leq C \varepsilon^{-\frac{3(i-1)}{16(n-1)}+\frac{3}{32(n-1)}}.$$
Following the  method used to estimate $H^2$ norms of $(v_e^1,g_e^1)$, we get
\begin{align*}
&\|(v_e^i,g_e^i)\|_{H^2(\Omega_1)}  \leq C \varepsilon^{-\frac{3(i-1)}{16(n-1)}+\frac{3}{32(n-1)}},\\
&\|(v_e^i, g_e^i)\|_{L^{\infty}(\Omega_{1})}\leq C \varepsilon^{-\frac{3(i-1)}{16(n-1)}+\frac{3}{32(n-1)}},\\
\end{align*}
Similarly, we have the higher order estimates
\begin{align*}
&\|(v_e^i, g_e^i)\|_{H^{2+j}(\Omega_{1})} \leq C \varepsilon^{-\frac{3(i-1+j)}{16(n-1)}+\frac{3}{32(n-1)}},
\end{align*}
where  $j=1,2,\ldots,2(n-i+1)$.
This leads to the desired estimates $(\ref{5.13})$. By the standard elliptic theory, we  obtain $W^{k,p}$ estimates of $v_e^i ~\text{and}~ g_e^i$. This ends the  proof of the proposition.
\end{proof}

Next, we study the estimates of $u_e^i, h_e^i ~\text{and}~p_e^i$.
Using $(\ref{3.9})$ and the divergence free conditions, we have
\begin{align}\label{5.15}
& u_e^i(x, Y)=u_b^i(Y)-\int_0^x v_{e Y}^i(s, Y) \mathrm{d} s, \nonumber\\
& h_e^i(x, Y)=h_b^i(Y)-\int_0^x g_{e Y}^i(s, Y) \mathrm{d} s, \\
& p_e^i(x, y)=\int_Y^1 [u_e^0 v_{e x}^i-h_e^0 g_{e x}^i-f_2^i](x, \theta) \mathrm{d} \theta-\int_0^x[ u_e^0(1) v_{e Y}^i(s, 1)-h_e^0(1) g_{e Y}^i(s, 1)-f_1^i(s,1)] \mathrm{d} s,\nonumber
\end{align}
where $u_e^i(0, Y)=u_b^i(Y) ~\text{and}~ h_e^i(0, Y)=h_b^i(Y)$ satisfy $\partial_Y u_b^i(1)=\partial_Y h_b^i(1)=0$. Hence, we have $u_{e Y}^i(x, 1)=h_{e Y}^i(x, 1)=0$.
By definition of $(u_e^i, h_e^i)$ in $(\ref{5.15})$ and Proposition $\ref{b5.1}$, we get
\begin{align}\label{5.17}
&\|(u_e^i, h_e^i)\|_{L^{\infty}(\Omega_{1})}+\|(u_e^i, h_e^i)\|_{H^1(\Omega_{1})} \leq C \varepsilon^{-\frac{3(i-1)}{16(n-1)}+\frac{3}{32(n-1)}},\nonumber\\
&\|(u_e^i, h_e^i)\|_{H^{1+j}(\Omega_{1})} \leq C \varepsilon^{-\frac{3(i-1+j)}{16(n-1)}+\frac{3}{32(n-1)}}, \quad j=1,2,\ldots,2(n-i+1).
\end{align}


\subsection{The $\varepsilon^{\frac{i}{2}}$-order MHD boundary layer profiles}


In this subsection, we construct the $\varepsilon^{\frac{i}{2}}$-order MHD boundary layer profiles $(u_p^i, v_p^i, h_p^i, g_p^i, p_p^i)$, $(1\leq i \leq n)$. Let us first estimate the $i$-th order boundary layer profiles $(u_p^i, v_p^i, h_p^i, g_p^i, p_p^i)(1\leq i \leq n-1)$.


\subsubsection{The $i$-th order MHD boundary layer profiles $(1\leq i \leq n-1)$}



We now study the $i$-th order MHD boundary layer system
\begin{align}\label{6.1}
\left\{\begin{aligned}
&u^0 \partial_x u_p^{i}+u_p^{i} \partial_x u^0+v_p^{i} \partial_y u^0+(v_p^0+\overline{v_e^1}) \partial_y u_p^{i}-\nu_{1}\partial_y^2 u_p^{i} +\partial_{x}p_p^i\\
& \quad -h^0 \partial_x h_p^{i}-h_p^{i} \partial_x h^0-g_p^{i} \partial_y h^0-(g_p^0+\overline{g_e^1}) \partial_y h_p^{i}=F_{p_{1}}^i, \\
&u^0 \partial_x h_p^{i}+u_p^{i} \partial_x h^0+v_p^{i} \partial_y h^0+(v_p^0+\overline{v_e^1}) \partial_y h_p^{i}- \nu_{3}\partial_y^2 h_p^{i} \\
& \quad-h^0 \partial_x u_p^{i}-h_p^{i} \partial_x u^0-g_p^{i} \partial_y u^0-(g_p^0+\overline{g_e^1}) \partial_y u_p^{i}=F_{p_2}^i ,
\end{aligned}\right.
\end{align}
with the boundary conditions 
\begin{align}\label{6.2}
& (u_p^{i}, h_p^{i})(0, y)=(\tilde{u}_i(y), \tilde{h}_{i}(y)), (u_p^{i}, h_{py}^{i})(x, 0)=-(u_e^i, h_{eY}^{i-1})(x, 0),\nonumber\\
&u_{p y}^{i}(x, \varepsilon^{-1/2})=v_p^{i}(x, \varepsilon^{-1/2})=0, v_p^{i}(x, 0)=-v_e^{i+1}(x,0),\\
&h_{p y}^{i}(x, \varepsilon^{-1/2})=g_p^{i}(x, \varepsilon^{-1/2})=0, g_p^{i}(x, 0)=-g_e^{i+1}(x,0),\nonumber
\end{align}
where
\begin{align*}
F_{p_1}^i=&-\left[\varepsilon^{-\frac{1}{2}}E_{r_1}^{i-1} +u_p^{i-1} \partial_{x}(u_p^1+u_e^1)-h_p^{i-1} \partial_{x}(h_p^1+h_e^1) + v_p^{i-1}(\partial_{Y}u_e^0+\partial_{y}u_p^1)
-g_p^{i-1}(\partial_{Y}h_e^0+\partial_{y}h_p^1)\right.\\
&+\partial_{x}u_{p}^{i-1}(u_e^1+u_p^1)-\partial_{x}h_{p}^{i-1}(h_e^1+h_p^1)
+ v_e^2\partial_{y}u_p^{i-1}-g_e^2\partial_{y}h_p^{i-1} +v_e^i \partial_{y}u_p^1-g_e^i \partial_{y}h_p^1\\
& \left.+u_e^{i}\partial_{x}u_p^0+u_p^0 \partial_{x}u_e^i-h_e^{i}\partial_{x}h_p^0-h_p^0 \partial_{x}h_e^i- \nu_{1}\partial_{x}^{2} u_p^{i-2}\right],\\
F_{p_2}^i=&-\left[\varepsilon^{-\frac{1}{2}}E_{r_2}^{i-1} +u_p^{i-1} \partial_{x}(h_p^1+h_e^1)-h_p^{i-1} \partial_{x}(u_p^1+u_e^1) + v_p^{i-1}(\partial_{Y}h_e^0+\partial_{y}h_p^1)
-g_p^{i-1}(\partial_{Y}u_e^0+\partial_{y}u_p^1)\right.\\
&+\partial_{x}h_{p}^{i-1}(u_e^1+u_p^1)-\partial_{x}u_{p}^{i-1}(h_e^1+h_p^1)
+ v_e^2\partial_{y}h_p^{i-1}-g_e^2\partial_{y}u_p^{i-1} +v_e^i \partial_{y}h_p^1-g_e^i \partial_{y}u_p^1\\
& \left.+u_e^{i}\partial_{x}h_p^0+u_p^0 \partial_{x}h_e^i-h_e^{i}\partial_{x}u_p^0-h_p^0 \partial_{x}u_e^i- \nu_{3}\partial_{x}^{2} h_p^{i-2}\right],
\end{align*}
When $i=1$, 
$\partial_x p_p^1=0$ and
$F_{p_1}^1 ~\text{and}~ F_{p_2}^1$ are as follows
\begin{align*}
F_{p_{1}}^1 = & -\overline{\partial_Y u_e^0}[y \partial_x u_p^0+v_p^0]-y \overline{\partial_Y v_e^1} \partial_y u_p^0-\overline{u_e^1} \partial_x u_p^0-u_p^0 \overline{\partial_x u_e^1}-\overline{v_e^2} \partial_{y}u_p^0  \nonumber\\
& +\overline{\partial_Y h_e^0}[y \partial_x h_p^0+g_p^0]+y \overline{\partial_Y g_e^1} \partial_y h_p^0+\overline{h_e^1} \partial_x h_p^0+h_p^0 \overline{\partial_x h_e^1}+\overline{g_e^2} \partial_{y}h_p^0, \nonumber\\
F_{p_2}^1 = & -\overline{\partial_Y h_e^0} v_p^0-y \overline{\partial_Y u_e^0} \partial_x h_p^0-y \overline{\partial_Y v_e^1} \partial_y h_p^0-\overline{\partial_x h_e^1} u_p^0-\overline{u_e^1} \partial_x h_p^0-\overline{v_e^2} \partial_{y}h_p^0 \\
& +\overline{\partial_Y u_e^0} g_p^0+y \overline{\partial_Y h_e^0} \partial_x u_p^0+y \overline{\partial_Y g_e^1} \partial_y u_p^0+\overline{h_e^1} \partial_x u_p^0+\overline{\partial_x u_e^1} h_p^0 +\overline{g_e^2} \partial_{y}u_p^0.
\end{align*}

Next, we estimate the solution $(u_p^i, v_p^i, h_p^i, g_p^i,
p_p^i)$ to the system $(\ref{6.1})$. We follow the method of Section $2$ and omit the details. First, we extend the domain $I_{\varepsilon}$ to $\mathbb{R}_{+}$ and establish the estimates on $[0,L] \times \mathbb{R}_{+}$.
Second, we cut the domain from $\mathbb{R}_{+}$ to $I_{\varepsilon}$.

\subsubsection*{Step1}

We apply the method used to estimate $(u_{p}^{0, \infty}, v_{p}^{0, \infty}, h_{p}^{0, \infty}, g_{p}^{0, \infty})$ to bound $(u_{p}^{i, \infty}, v_{p}^{i, \infty}, h_{p}^{i, \infty}, g_{p}^{i, \infty})$ on $[0,L] \times \mathbb{R}_{+}$.
\begin{pro}\label{b6.1}
 For $1\leq i\leq n-1$, there exists a smooth solution $(u_p^{i,\infty}, v_p^{i,\infty}, h_p^{i,\infty}, g_p^{i,\infty}, p_p^{i, \infty})$ to the system $(\ref{6.1})$ in $\left[0, L\right] \times[0, \infty)$ such that the following estimates hold
\begin{align}\label{6.4}
&\|(u_p^{i,\infty}, v_p^{i,\infty}, h_p^{i,\infty}, g_p^{i,\infty})\|_{L^{\infty}(0,L;\mathbb{R}_{+})}+\sup _{0 \leq x \leq L}\|\langle y\rangle^l \partial_{y y}(v_p^{i,\infty}, g_p^{i,\infty})\|_{L^2(0, \infty)} \nonumber\\
&\qquad +\|\langle y\rangle^l \partial_{x y}(v_p^{i,\infty}, g_p^{i,\infty})\|_{L^2(0,L;\mathbb{R}_{+})} \leq C(L, \xi) \varepsilon^{-\xi-\frac{3(i-1)}{16(n-1)}},  \nonumber\\
\sup _{0 \leq x \leq L}&\|\langle y\rangle^l \partial_{y y}(v_p^{i, \infty}, g_p^{i, \infty})\|_{{H^j}(0, \infty)}
+\|\langle y\rangle^l \partial_{x y}(v_p^{i, \infty}, g_p^{i, \infty})\|_{{H^j}(0,L;\mathbb{R}_{+})} \leq C(L, \xi) \varepsilon^{-\xi-\frac{3(i-1+j)}{16(n-1)}},
\end{align}
for some small positive constant $\xi$ and $j=1,2,\ldots,2(n-i)+1$.
\end{pro}

\subsubsection*{Step 2}

Now, we cut the solution from $[0,L]\times \mathbb{R}_{+}$ to $\Omega_{\varepsilon}$ as in Section $2$, which leads to
\begin{pro}\label{b6.2}
 Under the assumptions of Theorem $\ref{b1.1}$, there exists smooth functions $(u_p^i, v_p^i, h_p^i, g_p^i, p_p^i)$ satisfying the following inhomogeneous system
\begin{align}\label{6.5}
\left\{\begin{aligned}
& u^0 u_{p x}^i+u_x^0 u_p^i+u_y^0 v_p^i+[v_p^0+\overline{v_e^1}] u_{p y}^i+\partial_{x}p_p^i-\nu_{1}\partial_{y}^2u_p^i\\
& \quad -h^0 h_{p x}^i-h_x^0 h_p^i-h_y^0 g_p^i -[g_p^0+\overline{g_e^1}] h_{p y}^i=R_p^{u, i}, \\
& u^0 h_{p x}^i+h_x^0 u_p^i+h_y^0 v_p^i+[v_p^0+\overline{v_e^1}] h_{p y}^i-\nu_{3}\partial_{y}^2h_{p}^i \\
& \quad -h^0 u_{p x}^i-u_x^0 h_p^i-u_y^0 g_p^i-[g_p^0+\overline{g_e^1}] u_{p y}^i=R_p^{h, i}, \\
& u_{p x}^i+v_{p y}^i=h_{p x}^i+g_{p y}^i=0, \\
&(u_p^i, h_p^i)(0, y)=(\tilde{u}_i(y),\tilde{h}_i(y)), (u_p^i,\partial_y h_p^i)(x, 0)=-(u_e^i, \partial_{Y} h_e^{i-1})(x, 0), \\
&(u_{p y}^i,h_{py}^i)(x, \varepsilon^{-1/2})=(v_p^i,g_p^i)(x, \varepsilon^{-1/2})=0, (v_p^i, g_p^i)(x, 0)=-(v_e^{i+1}, g_e^{i+1})(x,0),
\end{aligned}\right.
\end{align}
such that, for any given $l \in \mathbb{N}$,
\begin{align}\label{6.6}
& \|(u_p^i, v_p^i, h_p^i, g_p^i)\|_{L^{\infty}(\Omega_{\varepsilon})}+\sup _{0 \leq x \leq L}\|\langle y\rangle^l \partial_{y y}
(v_{p}^i,g_{p}^i)\|_{L^2(I_{\varepsilon})}
+\|\langle y\rangle^l \partial_{x y}(v_{p}^i, g_{p}^i)\|_{L^2(\Omega_{\varepsilon})} \leq C(L, \xi) \varepsilon^{-\xi-\frac{3(i-1)}{16(n-1)}},\nonumber\\
&\sup _{0 \leq x \leq L}\|\langle y\rangle^l \partial_{y y}(v_p^i, g_p^i)\|_{{H^j}(I_{\varepsilon})}
+\|\langle y\rangle^l \partial_{x y}(v_p^i, g_p^i)\|_{{H^j}(\Omega_{\varepsilon})} \leq C(L, \xi) \varepsilon^{-\xi-\frac{3(i-1+j)}{16(n-1)}},
\end{align}
where $R_p^{u, i}, R_{p}^{h,i}$ are higher order terms of $\sqrt{\varepsilon}$ and $j=1,2,\ldots,2(n-i)+1$.
\end{pro}

\begin{proof}
 The solution $(u_p^{i,\infty}, v_p^{i,\infty}, h_p^{i,\infty}, g_p^{i,\infty})$ has been constructed in Proposition $\ref{b6.1}$. We define
\begin{align}\label{6.7}
& (u_p^i, h_p^i)(x, y):=\chi(\sqrt{\varepsilon} y) (u_p^{i,\infty}, h_p^{i,\infty})(x, y)-\sqrt{\varepsilon} \chi^{\prime}(\sqrt{\varepsilon} y) \int_0^y ( u_p^{i,\infty}, h_p^{i,\infty} )(x, \theta) \mathrm{d} \theta, \nonumber\\
& (v_p^i, g_p^i)(x, y):=\chi(\sqrt{\varepsilon} y) (v_p^{i,\infty}, g_p^{i,\infty})(x, y).
\end{align}
It is clear that $(u_p^i, v_p^i, h_p^i, g_p^i)$ satisfies $(\ref{6.5})$. Thus, we have
\begin{align*}
R_p^{u,i}=\sqrt{\varepsilon}\triangle_i^1 + \varepsilon \triangle_i^2, \quad R_p^{h,i}=\sqrt{\varepsilon}\triangle_i^3+\varepsilon \triangle_i^4,
\end{align*}
where
\begin{align*}
\triangle_i^1:= &  \chi^{\prime} [u^0 v_p^{i,\infty}- h^0 g_p^{i,\infty}]+ \chi^{\prime} [u_x^0 \int_0^y u_p^{i,\infty} \mathrm{d} \theta-h_x^0 \int_0^y h_p^{i,\infty} \mathrm{d} \theta]-3 \nu_1 \chi^{\prime} u_{p y}^{i,\infty}\\
&+2 \chi^{\prime}
\{[v_p^0+\overline{v_e^1}] u_p^{i,\infty}-[g_p^0+\overline{g_e^1}] h_p^{i,\infty}\}+ F_{p_1}^i \int_0^y \chi^{\prime} \mathrm{d} \theta,\\
\triangle_i^2:=&  \chi^{\prime \prime}[v_p^0+\overline{v_e^1}] \int_0^y u_p^{i,\infty} \mathrm{d} \theta-\nu_1\varepsilon^{1 / 2} \chi^{\prime \prime \prime} \int_0^y u_p^{i,\infty} \mathrm{d} \theta
- \chi^{\prime \prime}[g_p^0+\overline{g_e^1}] \int_0^y h_p^{i,\infty} \mathrm{d} \theta
-3\nu_1 \chi^{\prime \prime} u_p^{i,\infty}, \\
\triangle_i^3:= &  \chi^{\prime} [u^0 g_p^{i,\infty}- h^0 v_p^{i,\infty}]+ \chi^{\prime} [h_x^0 \int_0^y u_p^{i,\infty} \mathrm{d} \theta-u_x^0 \int_0^y h_p^{i,\infty} \mathrm{d} \theta]-3\nu_3 \chi^{\prime} h_{p y}^{i,\infty}\\
&+2  \chi^{\prime}\{[v_p^0+\overline{v_e^1}] h_p^{i,\infty}-[g_p^0+\overline{g_e^1}] u_p^{i,\infty}\}+ F_{p_2}^i \int_0^y \chi^{\prime} \mathrm{d} \theta,\\
\triangle_i^4:=& \chi^{\prime \prime}[v_p^0+\overline{v_e^1}] \int_0^y h_p^{i,\infty} \mathrm{d} \theta -\nu_3\varepsilon^{1 / 2} \chi^{\prime \prime \prime} \int_0^y h_p^{i,\infty} \mathrm{d} \theta
- \chi^{\prime \prime}[g_p^0+\overline{g_e^1}] \int_0^y u_p^{i,\infty} \mathrm{d} \theta
-3\nu_3  \chi^{\prime \prime} h_p^{i,\infty}.
\end{align*}
Clearly,
$$\partial_x u_p^i + \partial_y v_p^i =\partial_x h_p^i + \partial_y g_p^i=0.$$
Next, by using  Proposition $\ref{b6.1}$, we have
$$
\left|\sqrt{\varepsilon} \chi^{\prime}(\sqrt{\varepsilon} y) \int_0^y u_p^{i,\infty}(x, \theta) \mathrm{d} \theta\right| \leq \sqrt{\varepsilon} y\left|\chi^{\prime}(\sqrt{\varepsilon} y)\right|\|u_p^{i,\infty}\|_{L^{\infty}} \leq C(L, \xi) \varepsilon^{-\xi-\frac{3(i-1)}{16(n-1)}},
$$

$$
\left|\sqrt{\varepsilon} \chi^{\prime}(\sqrt{\varepsilon} y) \int_0^y h_p^{i,\infty}(x, \theta) \mathrm{d} \theta\right| \leq \sqrt{\varepsilon} y\left|\chi^{\prime}(\sqrt{\varepsilon} y)\right|\|h_p^{i,\infty}\|_{L^{\infty}} \leq C(L, \xi) \varepsilon^{-\xi-\frac{3(i-1)}{16(n-1)}} .
$$
Hence, applying  Proposition $\ref{b6.1}$, we obtain $(\ref{6.6})$. This completes the proof of the proposition.
\end{proof}


\subsubsection{The $n$-th order MHD boundary layer profile}


 Let us estimate the solution $(u_p^n, v_p^n, h_p^n, g_p^n, p_p^n)$ of the following system
\begin{align}\label{8.1}
\left\{\begin{aligned}
&u^0 \partial_x u_p^{n}+u_p^{n} \partial_x u^0+v_p^{n} \partial_y u^0
 +(v_p^0+\overline{v_e^1}) \partial_y u_p^{n}
 -\nu_{1}\partial_y^2 u_p^{n} + \partial_{x}p_p^n \\
& \quad -h^0 \partial_x h_p^{n}-h_p^{n} \partial_x h^0
 -g_p^{n} \partial_y h^0
 -(g_p^0+\overline{g_e^1}) \partial_y h_p^{n}=F_{p_{1}}^n,\\
&u^0 \partial_x h_p^{n}+u_p^{n} \partial_x h^0+v_p^{n} \partial_y h^0
 +(v_p^0+\overline{v_e^1}) \partial_y h_p^{n}
 -\nu_{3}\partial_y^2 h_p^{n} \\
&\quad-h^0 \partial_x u_p^{n}-h_p^{n} \partial_x u^0
 -g_p^{n} \partial_y u^0
 -(g_p^0+\overline{g_e^1}) \partial_y u_p^{n} =F_{p_2}^n,
\end{aligned}\right.
\end{align}
with the boundary conditions 
\begin{align}\label{8.2}
\left\{\begin{aligned}
& (u_p^{n}, h_p^{n})(0, y)=(\tilde{u}_n(y), \tilde{h}_{n}(y)), (u_p^{n}, h_{py}^{n})(x, 0)=-(u_e^n, h_{eY}^{n-1})(x, 0),\\
&u_{p y}^{n}(x, \varepsilon^{-1/2})=v_p^{n}(x, 0)=v_p^{n}(x, \varepsilon^{-1/2})=0,\\
&h_{p y}^{n}(x, \varepsilon^{-1/2})=g_p^{n}(x, 0)=g_p^{n}(x, \varepsilon^{-1/2})=0,
\end{aligned}\right.
\end{align}
where $F_{p_1}^n$ and $F_{p_2}^n$ are as follows
\begin{align*}
F_{p_1}^n=&-\left[\varepsilon^{-\frac{1}{2}}E_{r_1}^{n-1}+u_p^{n-1}\partial_{x}
(u_p^1+u_e^1)+ (u_e^1+u_p^1)\partial_{x}u_p^{n-1}
+ v_p^{n-1}(\partial_{Y}u_e^0+\partial_{y}u_p^1) + v_e^2\partial_{y}u_p^{n-1} + v_e^n \partial_{y}u_p^1\right.\\
&\quad+\partial_{x} u_e^n u_p^0+ u_e^n \partial_{x}u_p^0-h_p^{n-1}\partial_{x}
(h_p^1+h_e^1)- (h_e^1+h_p^1)\partial_{x}h_p^{n-1}
 -g_p^{n-1}(\partial_{Y}h_e^0+\partial_{y}h_p^1)- g_e^2\partial_{y}h_p^{n-1}\\
&\left.\quad - g_e^n \partial_{y}h_p^1-\partial_{x} h_e^n h_p^0-h_e^n \partial_{x}h_p^0
- \nu_{1}\partial_{x}^{2} u_p^{n-2}\right],\\
F_{p_2}^n=&-\left[\varepsilon^{-\frac{1}{2}}E_{r_2}^{n-1}+u_p^{n-1}\partial_{x}
(h_p^1+h_e^1)+ (u_e^1+u_p^1)\partial_{x}h_p^{n-1}
+ v_p^{n-1}(\partial_{Y}h_e^0+\partial_{y}h_p^1) + v_e^2\partial_{y}h_p^{n-1} + v_e^n \partial_{y}h_p^1\right.\\
&\quad+\partial_{x} h_e^n u_p^0+ u_e^n \partial_{x}h_p^0-h_p^{n-1}\partial_{x}
(u_p^1+u_e^1)- (h_e^1+h_p^1)\partial_{x}u_p^{n-1}
 -g_p^{n-1}(\partial_{Y}u_e^0+\partial_{y}u_p^1)- g_e^2\partial_{y}u_p^{n-1}\\
&\left.\quad - g_e^n \partial_{y}u_p^1-\partial_{x} u_e^n h_p^0-h_e^n \partial_{x}u_p^0
- \nu_{3}\partial_{x}^{2} h_p^{n-2}\right].
\end{align*}

\subsubsection*{Step1}

 We establish the estimates of $(u_p^{n,\infty}, v_p^{n,\infty}, h_p^{n,\infty}, g_p^{n,\infty})$ on $[0,L] \times \mathbb{R}_{+}$ as in Section $3.2.1$ and neglect the proof here.

\begin{pro}\label{b8.1}
 For $n \geq 3$, there exists a smooth solution $(u_p^{n,\infty}, v_p^{n,\infty}, h_p^{n,\infty}, g_p^{n,\infty}, p_{p}^{n,\infty})$ to the problem $(\ref{8.1})$ in $[0, L] \times[0, \infty)$ such that the following estimates hold,
\begin{align}\label{8.4}
\begin{aligned}
&\|(u_p^{n,\infty}, v_p^{n,\infty}, h_p^{n,\infty}, g_p^{n,\infty})\|_{L^{\infty}(0,L;\mathbb{R}_{+})}+\sup _{0 \leq x \leq L}\|\langle y\rangle^l \partial_{yy}(v_p^{n,\infty}, g_p^{n,\infty})\|_{L^2(0, \infty)} \\
& \qquad +\|\langle y\rangle^l \partial_{xy}(v_p^{n,\infty}, g_p^{n,\infty})\|_{L^2(0,L;\mathbb{R}_{+})} \leq C(L, \xi) \varepsilon^{-\xi-\frac{3}{16}}, \\
\sup _{0 \leq x \leq L}&\|\langle y\rangle^l \partial_{xyy}(v_p^{n,\infty}, g_p^{n,\infty})\|_{L^2(0, \infty)}+\|\langle y\rangle^l \partial_{xxy}(v_p^{n,\infty}, g_p^{n,\infty})\|_{L^2(0,L;\mathbb{R}_{+})} \leq C(L) \varepsilon^{-\frac{3}{16}-\frac{3}{16(n-1)}},
\end{aligned}
\end{align}
where $\xi$ is small positive constant.
\end{pro}

\subsubsection*{Step 2}

We cut the domain from $[0,L]\times \mathbb{R}_{+}$ to $\Omega_{\varepsilon}$ and give the following proposition to estimate the solution of the system $(\ref{8.1})$ on $\Omega_{\varepsilon}$.
 \begin{pro}\label{b8.2}
 Under the assumptions of Theorem $\ref{b1.1}$, there exists smooth functions $(u_p^n, v_p^n, h_p^n, g_p^n, p_p^n)$, satisfying the following inhomogeneous system
\begin{align}\label{8.5}
\left\{\begin{aligned}
& u^0 u_{p x}^n+u_x^0 u_p^n+u_y^0 v_p^n+[v_p^0+\overline{v_e^1}] u_{p y}^n+p_{p x}^n-\nu_{1}\partial_{y}^2u_p^n\\
& \quad -h^0 h_{p x}^n-h_x^0 h_p^n-h_y^0 g_p^n-[g_p^0+\overline{g_e^1}] h_{p y}^n=R_p^{u, n}, \\
& u^0 h_{p x}^n+h_x^0 u_p^n+h_y^0 v_p^n+[v_p^0+\overline{v_e^1}] h_{p y}^n-\nu_{3}\partial_{y}^2h_{p}^n\\
& \quad -h^0 u_{p x}^n-u_x^0 h_p^n-u_y^0 g_p^n-[g_p^0+\overline{g_e^1}] u_{p y}^n=R_p^{h, n}, \\
& u_{p x}^n+v_{p y}^n=h_{p x}^n+g_{p y}^n=0, \\
&(u_p^n, h_p^n)(0, y)=(\tilde{u}_n(y),\tilde{h}_n(y)), (u_p^n,\partial_y h_p^n )(x, 0)=-(u_e^n, \partial_{Y} h_e^{n-1})(x, 0), \\
&(u_{p y}^n,h_{py}^n)(x, \varepsilon^{-1/2})
=(v_p^n, g_p^n)(x, 0)=(v_p^n,g_p^n)(x, \varepsilon^{-1/2})=0,
\end{aligned}\right.
\end{align}
such that for any given $l \in \mathbb{N}$,
\begin{align}\label{8.6}
& \|(u_p^n, v_p^n, h_p^n, g_p^n)\|_{L^{\infty}(\Omega_{\varepsilon})}+\sup _{0 \leq x \leq L}\|\langle y\rangle^l \partial_{yy}(v_{p}^n,g_{p}^n )\|_{L^2(I_{\varepsilon})}
+\|\langle y\rangle^l \partial_{xy}(v_{p}^n, g_{p}^n)\|_{L^2(\Omega_{\varepsilon})} \leq C(L, \xi) \varepsilon^{-\xi-\frac{3}{16}}, \nonumber\\
& \sup _{0 \leq x \leq L}\|\langle y\rangle^l \partial_{xyy}( v_{p}^n,g_{p}^n)\|_{L^2(I_{\varepsilon})}
+\|\langle y\rangle^l \partial_{xxy}(v_{p}^n, g_{p}^n)\|_{L^2(\Omega_{\varepsilon})} \leq C(L) \varepsilon^{-\frac{3}{16}-\frac{3}{16(n-1)}},
\end{align}
where $R_p^{u, n}$ and $R_{p}^{h, n}$ are higher order terms.
\end{pro}

\begin{proof}
The solution $(u_p^{n,\infty}, v_p^{n,\infty}, h_p^{n,\infty}, g_p^{n,\infty})$ has been constructed in Proposition $\ref{b8.1}$. We introduce the following definition
\begin{align}\label{8.7}
& (u_p^n, h_p^n )(x, y):=\chi(\sqrt{\varepsilon} y) (u_p^{n,\infty}, h_p^{n,\infty})(x, y)-\sqrt{\varepsilon} \chi^{\prime}(\sqrt{\varepsilon} y) \int_0^y (u_p^{n,\infty}, h_p^{n,\infty})(x, \theta) \mathrm{d} \theta, \nonumber\\
& (v_p^n, g_p^n)(x, y):=\chi(\sqrt{\varepsilon} y) (v_p^{n,\infty}, g_p^{n,\infty} )(x, y).
\end{align}
Clearly, $(u_p^n, v_p^n, h_p^n, g_p^n)$ satisfies $(\ref{8.5})$ and we have
\begin{align}\label{8.8}
R_p^{u, n}:= & \sqrt{\varepsilon} \chi^{\prime} [u^0 v_p^{n,\infty}- h^0 g_p^{n,\infty}]+\sqrt{\varepsilon} \chi^{\prime} [u_x^0 \int_0^y u_p^{n,\infty} \mathrm{d} \theta-h_x^0 \int_0^y h_p^{n,\infty} \mathrm{d} \theta]+2 \sqrt{\varepsilon} \chi^{\prime}[v_p^0+\overline{v_e^1}] u_p^{n,\infty}\nonumber\\
&-2 \sqrt{\varepsilon}\chi^{\prime}[g_p^0+\overline{g_e^1}]h_p^{n,\infty}
-3\nu_1 \sqrt{\varepsilon} \chi^{\prime} u_{p y}^{n,\infty}+\varepsilon \chi^{\prime \prime}\{[v_p^0+\overline{v_e^1}] \int_0^y u_p^{n,\infty} \mathrm{d} \theta-[g_p^0+\overline{g_e^1}] \int_0^y h_p^{n,\infty} \mathrm{d} \theta\}\nonumber\\
&+\sqrt{\varepsilon} F_{p_1}^n \int_0^y \chi^{\prime} \mathrm{d} \theta
-3 \nu_1\varepsilon \chi^{\prime \prime} u_p^{n,\infty}-\nu_1\varepsilon^{3 / 2} \chi^{\prime \prime \prime} \int_0^y u_p^{n,\infty} \mathrm{d} \theta, \\
R_p^{h, n}:= & \sqrt{\varepsilon} \chi^{\prime} [u^0 g_p^{n,\infty}- h^0 v_p^{n,\infty}]+\sqrt{\varepsilon} \chi^{\prime} [h_x^0 \int_0^y u_p^{n,\infty} \mathrm{d} \theta-u_x^0 \int_0^y h_p^{n,\infty} \mathrm{d} \theta]+2 \sqrt{\varepsilon} \chi^{\prime}[v_p^0+\overline{v_e^1}] h_p^{n,\infty}\nonumber\\
&-2 \sqrt{\varepsilon} \chi^{\prime}[g_p^0+\overline{g_e^1}] u_p^{n,\infty}-3 \nu_3\sqrt{\varepsilon} \chi^{\prime} h_{p y}^{n,\infty}+\varepsilon \chi^{\prime \prime}\{[v_p^0+\overline{v_e^1}] \int_0^y h_p^{n,\infty} \mathrm{d} \theta-[g_p^0+\overline{g_e^1}] \int_0^y u_p^{n,\infty} \mathrm{d} \theta \}\nonumber\\
&+\sqrt{\varepsilon} F_{p_2}^n\int_0^y \chi^{\prime} \mathrm{d} \theta
-3 \nu_3\varepsilon \chi^{\prime \prime} h_p^{n,\infty}-\nu_3\varepsilon^{3 / 2} \chi^{\prime \prime \prime} \int_0^y h_p^{n,\infty} \mathrm{d} \theta,\nonumber
\end{align}
where $R_p^{u,n}$ and $R_p^{h,n}$ will be estimated in Proposition $\ref{b8.4}$.
It is clear that
$$\partial_x u_p^n + \partial_y v_p^n =\partial_x h_p^n + \partial_y g_p^n=0.$$
Using the Proposition $\ref{b8.1}$, we obtain
$$
\left|\sqrt{\varepsilon} \chi^{\prime}(\sqrt{\varepsilon} y) \int_0^y u_p^{n,\infty}(x, \theta) \mathrm{d} \theta\right| \leq \sqrt{\varepsilon} y\left|\chi^{\prime}(\sqrt{\varepsilon} y)\right|\|u_p^{n,\infty}\|_{L^{\infty}} \leq C(L, \xi) \varepsilon^{-\xi-\frac{3}{16}},
$$

$$
\left|\sqrt{\varepsilon} \chi^{\prime}(\sqrt{\varepsilon} y) \int_0^y h_p^{n,\infty}(x, \theta) \mathrm{d} \theta\right| \leq \sqrt{\varepsilon} y\left|\chi^{\prime}(\sqrt{\varepsilon} y)\right|\|h_p^{n,\infty}\|_{L^{\infty}} \leq C(L, \xi) \varepsilon^{-\xi-\frac{3}{16}} .
$$
We thus  obtain  $(\ref{8.6})$ through the Proposition $\ref{b8.1}$, which ends   the proof of the proposition.
\end{proof}

\begin{pro}\label{b8.3}
Assume that $(u_p^n, v_p^n, h_p^n, g_p^n, p_p^n)$ is the solution to $(\ref{8.5})$.
Then
\begin{align}\label{8.9}
&\|E_{r_3}^{n-1}\|_{L^2(\Omega_{\varepsilon})}
+\|E_{r_1}^n\|_{L^2(\Omega_{\varepsilon})}
+\|E_{r_2}^n\|_{L^2(\Omega_{\varepsilon})}
+\|R_p^{u, n}\|_{L^2(\Omega_{\varepsilon})}
+\|R_p^{h, n}\|_{L^2(\Omega_{\varepsilon})}\leq C(L, \xi) \varepsilon^{-\xi+\frac{1}{16}},\nonumber\\
&\|p_{p x}^{n+1}\|_{L^2(\Omega_{\varepsilon})} \leq C \varepsilon^{-\frac{7}{16}+\frac{3}{32(n-1)}},
\end{align}
for any $\xi>0$ small enough. 
\end{pro}
\begin{proof}
$E_{r_3}^{n-1}$ and $E_{r_i}^{n}(i=1,2)$ are defined in $(\ref{6.3.1})$ and $(\ref{8.3})$, using the same method as in Section $3.1.1$ and Section  $3.2.1$, we get
$$
\|(E_{r_3}^{n-1},E_{r_i}^{n})\|_{L^2(\Omega_{\varepsilon})}\leq C(L,\xi)\varepsilon^{-\xi+\frac{1}{16}} , \quad i=1,2 .
$$
It follows from the Proposition ${\ref{b8.1}}$ that
\begin{align}\label{8.10}
\|R_p^{u, n}\|_{L^2(\Omega_{\varepsilon})}^2
\lesssim & \sqrt{\varepsilon} L \|(v_p^{n,\infty}, g_p^{n,\infty})\|_{L^{\infty}}^2
+\varepsilon \|u_p^{n,\infty}\|_{L^{\infty}}^2 \iint_{\Omega_{\varepsilon}}\langle y\rangle^2|u_{p x}^0|^2 \mathrm{d}x \mathrm{d}y \nonumber\\
& + \varepsilon \|h_p^{n,\infty}\|_{L^{\infty}}^2 \iint_{\Omega_{\varepsilon}}\langle y\rangle^2|h_{p x}^0|^2 \mathrm{d}x \mathrm{d}y  +\varepsilon \|u_p^{n,\infty}\|_{L^{\infty}}^2 \iint_{\Omega_{\varepsilon}}|v_p^0+\overline{v_e^1}|^2 \mathrm{d}x \mathrm{d}y \nonumber\\
& + \varepsilon\|h_p^{n,\infty}\|_{L^{\infty}}^2 \iint_{\Omega_{\varepsilon}}|g_p^0+\overline{g_e^1}|^2\mathrm{d}x \mathrm{d}y +\varepsilon \iint_{\Omega_{\varepsilon}}|u_{p y}^{n,\infty}|^2\mathrm{d}x \mathrm{d}y \nonumber\\
&+\varepsilon \iint_{\Omega_{\varepsilon}}\langle y\rangle^2|F_{p_1}^n|^2 \mathrm{d}x \mathrm{d}y
+\varepsilon^{\frac{3}{2}}\|u_p^{n,\infty}\|_{L^{\infty}}^2 \nonumber\\
\leq & C(L, \xi) \varepsilon^{-2\xi+\frac{1}{8}} .
\end{align}
Similarly, we have
\begin{align}\label{8.11}
\|R_p^{h, n}\|_{L^2(\Omega_{\varepsilon})}^2
\lesssim & \sqrt{\varepsilon} L\|(v_p^{n,\infty}, g_p^{n,\infty})\|_{L^{\infty}}^2
+\varepsilon \|u_p^{n,\infty}\|_{L^{\infty}}^2 \iint_{\Omega_{\varepsilon}}\langle y\rangle^2|h_{p x}^0|^2 \mathrm{d}x \mathrm{d}y \nonumber\\
& +\varepsilon\|h_p^{n,\infty}\|_{L^{\infty}}^2 \iint_{\Omega_{\varepsilon}}\langle y\rangle^2|u_{p x}^0|^2\mathrm{d}x \mathrm{d}y +\varepsilon \|h_p^{n,\infty}\|_{L^{\infty}}^2 \iint_{\Omega_{\varepsilon}}|v_p^0+\overline{v_e^1}|^2 \mathrm{d}x \mathrm{d}y \nonumber\\
& +\varepsilon\|u_p^{n,\infty}\|_{L^{\infty}}^2 \iint_{\Omega_{\varepsilon}}|g_p^0+\overline{g_e^1}|^2 \mathrm{d}x \mathrm{d}y +\varepsilon \iint_{\Omega_{\varepsilon}}|h_{p y}^{n,\infty}|^2 \mathrm{d}x \mathrm{d}y \nonumber\\
&+\varepsilon \iint_{\Omega_{\varepsilon}}\langle y\rangle^2
|F_{p_2}^n|^2 \mathrm{d}x \mathrm{d}y +\varepsilon^{\frac{3}{2}}\|h_p^{n,\infty}\|_{L^{\infty}}^2 \nonumber\\
\leq & C(L, \xi) \varepsilon^{-2\xi+\frac{1}{8}} .
\end{align}
By  definition of $p_p^{n+1}$, we have
\begin{align*}
p_{px}^{n+1}(x, y)= & \int_y^{1/\sqrt{\varepsilon}}\partial_{x}\left\{\sum_{i_1+i_2=n-1}\left[(u_e^{i_1}+u_p^{i_1}) \partial_x v_p^{i_2}-(h_e^{i_1}+h_p^{i_1}) \partial_x g_p^{i_2}\right]-\nu_{2}\partial_{y}^2v_p^{n-1}\right.\nonumber\\
&+\sum_{i_1+i_2=n-1}\left[(v_p^{i_1}+v_e^{i_1+1}) \partial_y v_p^{i_2}-(g_p^{i_1}+g_e^{i_1+1}) \partial_y g_p^{i_2}\right]
-\nu_2 \partial_x^2 v_p^{n-3}\nonumber\\
&\left.+\sum_{i_1+i_2=n-1}\left[u_p^{i_1} \partial_x v_e^{i_2+1}-h_p^{i_1} \partial_x g_e^{i_2+1}\right]+\sum_{i_1+i_2=n-2}\left[v_p^{i_1} \partial_Y v_e^{i_2+1}-g_p^{i_1} \partial_Y g_e^{i_2+1}\right]
\right\}(x,\theta) \mathrm{d} \theta.
\end{align*}
Note that
\begin{align*}
\int_y^{1 / \sqrt{\varepsilon}}(u_e^{n-1}+u_p^{n-1}) v_{p x x}^0 (x, \theta) \mathrm{d} \theta
\leq & C\langle y\rangle^{-l+1}\|u_e^{n-1}+u_p^{n-1}\|_{L^{\infty}}
\|\langle y\rangle^l v_{p x x}^0\|_{L^2}, \\
\int_y^{1 / \sqrt{\varepsilon}} u_p^{n-1} v_{exx}^1(x, \theta) \mathrm{d} \theta
\leq & C\langle y\rangle^{-l+2}\|\langle y\rangle^l u_{p y}^{n-1}\|_{L^2}
\|v_{exx}^1(x, \sqrt{\varepsilon} \cdot)\|_{L^2},\\
\int_y^{1 / \sqrt{\varepsilon}}\left(v_p^{n-1}+v_e^n\right) v_{p x y}^0 \left(x, \theta\right) \mathrm{d} \theta
\leq & C\langle y\rangle^{-l+1}\|v_p^{n-1}+v_e^n\|_{L^{\infty}}
\|\langle y\rangle^l v_{p x y}^0\|_{L^2}, \\
\int_y^{1 / \sqrt{\varepsilon}} v_{p x y y}^{n-1} (x, \theta) \mathrm{d} \theta
\leq & C\langle y\rangle^{-l+1}\|\langle y\rangle^l v_{pxyy}^{n-1}\|_{L^2}.
\end{align*}
Similarly, we get
\begin{align*}
\int_y^{1 / \sqrt{\varepsilon}}(h_e^{n-1}+h_p^{n-1}) g_{p x x}^0 (x, \theta) \mathrm{d} \theta
\leq & C\langle y\rangle^{-l+1}\|h_e^{n-1}+h_p^{n-1}\|_{L^{\infty}}
\|\langle y\rangle^l g_{p x x}^0\|_{L^2}, \\
\int_y^{1 / \sqrt{\varepsilon}} h_p^{n-1} g_{exx}^1 (x, \theta) \mathrm{d} \theta
\leq & C\langle y\rangle^{-l+2}\|\langle y\rangle^l h_{py}^{n-1}\|_{L^2}
\|g_{exx }^1(x, \sqrt{\varepsilon} \cdot)\|_{L^2},\\
\int_y^{1 / \sqrt{\varepsilon}}(g_p^0+g_e^1) g_{p x y}^{n-1} (x, \theta) \mathrm{d} \theta
\leq & C\langle y\rangle^{-l+1}\|h_e^0+h_p^0\|_{L^{\infty}}
\|\langle y\rangle^l g_{p x y}^{n-1}\|_{L^2}.
\end{align*}
The remaining terms of $p_{px}^{n+1}$ are treated using the same method as above. Hence, taking $l \geq 3$, we have
$$
\|p_{p x}^{n+1}\|_{L^2(\Omega_{\varepsilon})} \leq C \varepsilon^{-\frac{7}{16}+\frac{3}{32(n-1)}} .
$$
This completes the proof of the proposition.
\end{proof}

Combining all the error terms derived from the construction of the ideal MHD profiles and boundary layer profiles, we obtain
\begin{pro}\label{b8.4}
Under the assumptions of Theorem $\ref{b1.1}$, for arbitrarily small $\xi>0$,
\begin{align}\label{8.12}
\|R_{app}^1\|_{L^2(\Omega_{\varepsilon})}
+\|R_{app}^3\|_{L^2(\Omega_{\varepsilon})}
+\sqrt{\varepsilon}(\|R_{a p p}^2\|_{L^2(\Omega_{\varepsilon})}+\|R_{a p p}^4\|_{L^2(\Omega_{\varepsilon})}) \leq C(L, \xi) \varepsilon^{\frac{1}{16}+\frac{n}{2}-\xi}.
\end{align}
\end{pro}
\begin{proof}
Collecting all the error terms in $R_{app}^1$ and $R_{app}^3$, we obtain
\begin{align}\label{8.13}
R_{a p p}^1:= &- \nu_{1}\varepsilon^{\frac{n+1}{2}} \partial_{Y Y} u_e^{n-1}
+ \varepsilon^{\frac{n+1}{2}} \triangle_{n-1}^{2}
+ \varepsilon^{\frac{n}{2}}(R^u_n+R_p^{u,n}+E_{r_1}^n)
+\varepsilon^{\frac{n+1}{2}} \partial_x p_p^{n+1}\nonumber\\
&+\sum_{\substack{i_1,i_2=1\\ i_1+i_2 \geq n}}^{n-1}\sqrt{\varepsilon}^{i_1+i_2+1}[v_e^{i_1+1}\partial_{y}(u_e^{i_2+1}+u_p^{i_2+1})
-g_e^{i_1+1}\partial_{y}(h_e^{i_2+1}+h_p^{i_2+1})]\nonumber\\
&+\sum_{\substack{i_1,i_2=1\\i_1+i_2\geq n+1}}^{n}\sqrt{\varepsilon}^{i_1+i_2}[(u_e^{i_1}+u_p^{i_1}) \partial_x+v_p^{i_1} \partial_y](u_e^{i_2}+u_p^{i_2}) \nonumber\\ &-\sum_{\substack{i_1,i_2=1\\i_1+i_2\geq n+1}}^{n}\sqrt{\varepsilon}^{i_1+i_2}[(h_e^{i_1}+h_p^{i_1}) \partial_x+g_p^{i_1}\partial_y](h_e^{i_2}+h_p^{i_2})\nonumber\\
&-\nu_1\varepsilon^{\frac{n+1}{2}} \partial_x^2[(u_p^{n-1}+u_e^{n-1})
+\sqrt{\varepsilon}(u_e^n+u_p^n)],
\end{align}
and
\begin{align}\label{8.13.1}
R_{a p p}^3:= &- \nu_{3}\varepsilon^{\frac{n+1}{2}} \partial_{Y Y} h_e^{n-1}
+ \varepsilon^{\frac{n+1}{2}} \triangle_{n-1}^{4}
+ \varepsilon^{\frac{n}{2}}(R^h_n+R_p^{h,n}+E_{r_2}^n)\nonumber\\
&+\sum_{\substack{i_1,i_2=1\\ i_1+i_2 \geq n}}^{n-1}\sqrt{\varepsilon}^{i_1+i_2+1}[v_e^{i_1+1}\partial_{y}(h_e^{i_2+1}+h_p^{i_2+1})
-g_e^{i_1+1}\partial_{y}(u_e^{i_2+1}+u_p^{i_2+1})]\nonumber\\
&+\sum_{\substack{i_1,i_2=1\\i_1+i_2\geq n+1}}^{n}\sqrt{\varepsilon}^{i_1+i_2}[(u_e^{i_1}+u_p^{i_1}) \partial_x+v_p^{i_1} \partial_y](h_e^{i_2}+h_p^{i_2}) \nonumber\\ &-\sum_{\substack{i_1,i_2=1\\i_1+i_2\geq n+1}}^{n}\sqrt{\varepsilon}^{i_1+i_2}[(h_e^{i_1}+h_p^{i_1}) \partial_x+g_p^{i_1}\partial_y](u_e^{i_2}+u_p^{i_2})\nonumber\\
&-\nu_3\varepsilon^{\frac{n+1}{2}} \partial_x^2[(h_p^{n-1}+h_e^{n-1})
+\sqrt{\varepsilon}(h_e^n+h_p^n)].
\end{align}
We now estimate $R^u_n$. It follows from the estimates of $v_p^i(i=0,1,\ldots,n) ~\text{and}~ (u_e^j,v_e^j,h_e^j,g_e^j)(j=1,2,\ldots,n)$ that
\begin{align*}
& \|(u_{e}^{n}, h_{e}^{n})\|_{L^{\infty}(\Omega_{\varepsilon})}
\leq C \varepsilon^{-\frac{3}{16}+\frac{3}{32(n-1)}}, \quad \|(u_{e}^{n}, h_{e}^{n})\|_{H^{1}(\Omega_{\varepsilon})} \leq C \varepsilon^{-\frac{7}{16}+\frac{3}{32(n-1)}}, \\
&\|\partial_{Y}^{2}u_{e}^{n}\|_{L^{2}(\Omega_{\varepsilon})} \leq C \varepsilon^{-\frac{1}{4}}\|\partial_{Y}^{2} u_{b}^{n}\|_{L^{2}(0,1)}+C \varepsilon^{-\frac{1}{4}}\|\partial_{Y}^{3}v_{e}^{n}\|_{L^{2}(\Omega_{1})}
\leq C \varepsilon^{-\frac{7}{16}-\frac{3}{32(n-1)}},\\
&\|(v_{p}^{0}+v_{e}^{1})\partial_{Y}u_{e}^{n}\|_{L^{2}(\Omega_{\varepsilon})} \leq \|v_{p}^{0}+v_{e}^{1}\|_{L^{\infty}(\Omega_{\varepsilon})}
\|\partial_{Y}u_{e}^{n}\|_{L^{2}(\Omega_{\varepsilon})} \leq C\varepsilon^{-\frac{7}{16}+\frac{3}{32(n-1)}},\\
&\|(v_{p}^{1}+v_{e}^{2})\partial_{Y}u_{e}^{n-1}\|_{L^{2}(\Omega_{\varepsilon})} \leq \|v_{p}^{1}+v_{e}^{2}\|_{L^{\infty}(\Omega_{\varepsilon})}
\|\partial_{Y}h_{e}^{n-1}\|_{L^{2}(\Omega_{\varepsilon})} \leq C\varepsilon^{-\frac{7}{16}+\frac{9}{32(n-1)}},\\
&\|(v_{p}^{2}+v_{e}^{3})\partial_{Y}u_{e}^{n-2}\|_{L^{2}(\Omega_{\varepsilon})} \leq \|v_{p}^{2}+v_{e}^{3}\|_{L^{\infty}(\Omega_{\varepsilon})}
\|\partial_{Y}h_{e}^{n-2}\|_{L^{2}(\Omega_{\varepsilon})} \leq C\varepsilon^{-\frac{7}{16}+\frac{3}{32(n-1)}}.
\end{align*}
Using similar calculations for the other terms of $R_{n}^{u}$, we obtain
\begin{align*}
\|R^{u}_{n}\|_{L^2(\Omega_{\varepsilon})} \leq C \varepsilon^{\frac{1}{16}}.
\end{align*}
Using similar arguments, we have
\begin{align*}
\|(R^{h}_{n}, R^{v}_{n-1}, R^{g}_{n-1})\|_{L^2(\Omega_{\varepsilon})} \leq C \varepsilon^{\frac{1}{16}}.
\end{align*}
By the above estimates and $(\ref{8.9})$, we directly have
\begin{align*}
&\|-\nu_{1}\varepsilon^{\frac{n+1}{2}} \partial_{Y Y} u_e^{n-1}
+\varepsilon^{\frac{n}{2}}(R^{u}_{n}+R_p^{u, n}+E_{r_1}^n)
+\varepsilon^{\frac{n+1}{2}} \partial_x p_p^{n+1}\|_{L^2({\Omega_{\varepsilon}})} \leq C(L, \xi) \varepsilon^{\frac{n}{2}+\frac{1}{16}-\xi}, \\
&\|-\nu_{3}\varepsilon^{\frac{n+1}{2}} \partial_{Y Y} h_e^{n-1}
+\varepsilon^{\frac{n}{2}}(R^{h}_{n}
+R_p^{h, n}+E_{r_2}^n)\|_{L^2({\Omega_{\varepsilon}})} \leq C(L, \xi) \varepsilon^{\frac{n}{2}+\frac{1}{16}-\xi} .
\end{align*}
Applying the estimates of $(u_e^i, v_e^i, h_e^i, g_e^i)$ and $(u_p^i, v_p^i, h_p^i, g_p^i)(i=1,2,\ldots,n)$, we obtain
\begin{align*}
&\varepsilon^{\frac{n+1}{2}}\|(u_e^n+u_p^n)\partial_x(u_e^1+u_p^1)\|_{L^2(\Omega_{\varepsilon})} \\
\leq & \varepsilon^{\frac{n+1}{2}} (\|u_e^n\|_{L^{\infty}}
+\|u_p^n\|_{L^{\infty}}) \times (\|\partial_x u_e^1\|_{L^2}+\|\partial_x u_p^1\|_{L^2})\leq  C(L, \xi) \varepsilon^{\frac{n}{2}+\frac{1}{16}-\xi}, \\
&\varepsilon^{\frac{n+1}{2}}\|v_p^n \partial_y(u_e^1+u_p^1)\|_{L^2
(\Omega_{\varepsilon})}\\
\leq & \varepsilon^{\frac{n+1}{2}} \|v_p^n\|_{L^{\infty}}(\sqrt{\varepsilon}
\|\partial_Y u_e^1\|_{L^2}+\|\partial_y u_p^1\|_{L^2})\leq  C(L, \xi) \varepsilon^{\frac{n}{2}+\frac{1}{16}-\xi}, \\
&\varepsilon^{\frac{n+1}{2}}\|\partial_x^2[(u_p^{n-1}+u_e^{n-1})+\sqrt{\varepsilon}
(u_e^n+u_p^n)]\|_{L^2(\Omega_{\varepsilon})} \\
\leq & \varepsilon^{\frac{n+1}{2}}\|\partial_x^2 u_p^{n-1}+\partial_x^2 u_e^{n-1}\|_{L^2}
+\varepsilon^{\frac{n+2}{2}}(\|\partial_{xY} v_e^n\|_{L^2}+\|\partial_{xy} v_p^n\|_{L^2})
\leq  C\varepsilon^{\frac{n}{2}+\frac{1}{16}+\frac{9}{32(n-1)}-\xi}.
\end{align*}
Similarly, we have
\begin{align*}
&\varepsilon^{\frac{n+1}{2}}\|(h_e^n+h_p^n) \partial_x(h_e^1+h_p^1)\|_{L^2(\Omega_{\varepsilon})}
\leq  C(L, \xi) \varepsilon^{\frac{n}{2}+\frac{1}{16}-\xi}, \\
&\varepsilon^{\frac{n+1}{2}}\|g_p^n\partial_y(h_e^1+h_p^1)\|_{L^2
(\Omega_{\varepsilon})}
\leq  C(L, \xi) \varepsilon^{\frac{n}{2}+\frac{1}{16}-\xi},\\
&\varepsilon^{\frac{n+1}{2}} \| \partial_x^2[(h_p^{n-1}+h_e^{n-1})+\sqrt{\varepsilon}
(h_e^n+h_p^n)]\|_{L^2(\Omega_{\varepsilon})} \\
\leq & \varepsilon^{\frac{n+1}{2}}\|\partial_x^2 h_p^{n-1}+\partial_x^2 h_e^{n-1}\|_{L^2}+\varepsilon^{\frac{n+2}{2}}(\|\partial_{xY} g_e^n\|_{L^2}+\|\partial_{xy} g_p^n\|_{L^2})
\leq  C\varepsilon^{\frac{n}{2}+\frac{1}{16}+\frac{9}{32(n-1)}-\xi}.
\end{align*}
For the remaining terms in $(\ref{8.13})$, 
we follow the same  approach as previously. Recalling the definition of $\triangle_{n-1}^{2}$, we get
$$
\begin{aligned}
\|\triangle_{n-1}^{2}\|_{L^2(\Omega_{\varepsilon})}
\lesssim & \|u_p^{n-1,\infty}\|_{L^2}
\|\langle y\rangle (v_p^{0}+\overline{v_e^1})\|_{L^{\infty}}
+\|h_p^{n-1,\infty}\|_{L^2}
\|\langle y\rangle (g_p^{0}+\overline{g_e^1})\|_{L^{\infty}}\\
&+\|\chi^{\prime \prime}\|_{L^2}\|u_p^{n-1,\infty}\|_{L^2}
+\|\chi^{\prime \prime}\|_{L^2}\|h_p^{n-1,\infty}\|_{L^2}\\
\leq & C(L, \xi) \varepsilon^{-\xi-\frac{3(n-2)}{16(n-1)}},
\end{aligned}
$$
similarly, we have
$\|\triangle_{n-1}^{4} \|_{L^2(\Omega_{\varepsilon})}\leq C(L,\xi)\varepsilon^{-\xi-\frac{3(n-2)}{16(n-1)}}.$
Therefore, we obtain
$$
\|R_{a p p}^1\|_{L^2(\Omega_{\varepsilon})} \leq C(L, \xi) \varepsilon^{\frac{n}{2}+\frac{1}{16}-\xi}.
$$
Similarly, we have
$$
\|R_{a p p}^3\|_{L^2(\Omega_{\varepsilon})} \leq C(L, \xi) \varepsilon^{\frac{n}{2}+\frac{1}{16}-\xi} .
$$
Next, we turn to estimates of $R_{a p p}^2$ and $R_{a p p}^4$. Collecting all error terms, we have
\begin{align}\label{8.14}
R_{a p p}^2:= & \varepsilon^{\frac{n-1}{2}} R^{v}_{n-1}
-\nu_{2} \varepsilon^{\frac{n}{2}}\Delta v_e^{n-1}
+\sum_{\substack{i_1,i_2=1\\i_1+i_2\geq n}}^{n-1}
\sqrt{\varepsilon}^{i_1+i_2}[(u_e^{i_1}+u_p^{i_1})\partial_{x}
+(v_p^{i_1}+v_e^{i_1+1}) \partial_{y}] (v_p^{i_2}+v_e^{i_2+1})\nonumber\\
&+\sum_{i_2=0}^{n-1}\varepsilon^{\frac{n+i_2}{2}}[(u_e^n+u_p^n)\partial_{x}
+v_p^n \partial_{y}] (v_p^{i_2}+v_e^{i_2+1}) +\varepsilon^{n}[(u_e^n+u_p^n)\partial_{x}+v_p^n\partial_{y}]v_p^n \nonumber\\
&+\sum_{i_1=0}^{n-1}\varepsilon^{\frac{n+i_1}{2}}[(u_e^{i_1}+u_p^{i_1})\partial_{x}
+(v_p^{i_1}+v_e^{i_1+1})\partial_{y}] v_p^n-\varepsilon^{n}[(h_e^n+h_p^n)\partial_{x}+g_p^n\partial_{y}]g_p^n \nonumber\\
&-\sum_{\substack{i_1,i_2=1\\i_1+i_2\geq n}}^{n-1}
\sqrt{\varepsilon}^{i_1+i_2}[(h_e^{i_1}+h_p^{i_1})\partial_{x}
+(g_p^{i_1}+g_e^{i_1+1}) \partial_{y}] (g_p^{i_2}+g_e^{i_2+1})-\nu_{2}\varepsilon^{\frac{n}{2}}\partial_y^2 v_p^n \nonumber\\
&-\sum_{i_2=0}^{n-1}\varepsilon^{\frac{n+i_2}{2}}[(h_e^n+h_p^n)\partial_{x}
+g_p^n \partial_{y}] (g_p^{i_2}+g_e^{i_2+1})-\nu_{2}\varepsilon^{\frac{n}{2}} \partial_x^2(v_p^{n-2}+\sqrt{\varepsilon} v_p^{n-1})\nonumber\\
&-\sum_{i_1=0}^{n-1}\varepsilon^{\frac{n+i_1}{2}}[(h_e^{i_1}+h_p^{i_1})\partial_{x}
+(g_p^{i_1}+g_e^{i_1+1})\partial_{y}] g_p^n-\nu_{2}\varepsilon^{\frac{n}{2}} \partial_x^2(\sqrt{\varepsilon} v_e^n+\varepsilon v_p^n),
\end{align}
and
\begin{align}\label{8.14.1}
R_{a p p}^4:= & -\nu_{4} \varepsilon^{\frac{n}{2}}\Delta g_e^{n-1}
+\sum_{\substack{i_1,i_2=1\\i_1+i_2\geq n}}^{n-1}
\sqrt{\varepsilon}^{i_1+i_2}[(u_e^{i_1}+u_p^{i_1})\partial_{x}
+(v_p^{i_1}+v_e^{i_1+1}) \partial_{y}] (g_p^{i_2}+g_e^{i_2+1})\nonumber\\
&+\sum_{i_2=0}^{n-1}\varepsilon^{\frac{n+i_2}{2}}[(u_e^n+u_p^n)\partial_{x}
+v_p^n \partial_{y}] (g_p^{i_2}+g_e^{i_2+1}) +\varepsilon^{n}[(u_e^n+u_p^n)\partial_{x}+v_p^n\partial_{y}]g_p^n \nonumber\\
&+\sum_{i_1=0}^{n-1}\varepsilon^{\frac{n+i_1}{2}}[(u_e^{i_1}+u_p^{i_1})\partial_{x}
+(v_p^{i_1}+v_e^{i_1+1})\partial_{y}] g_p^n-\varepsilon^{n}[(h_e^n+h_p^n)\partial_{x}+g_p^n\partial_{y}]v_p^n \nonumber\\
&-\sum_{\substack{i_1,i_2=1\\i_1+i_2\geq n}}^{n-1}
\sqrt{\varepsilon}^{i_1+i_2}[(h_e^{i_1}+h_p^{i_1})\partial_{x}
+(g_p^{i_1}+g_e^{i_1+1}) \partial_{y}] (v_p^{i_2}+v_e^{i_2+1})-\nu_{4}\varepsilon^{\frac{n}{2}}\partial_y^2 g_p^n \nonumber\\
&-\sum_{i_2=0}^{n-1}\varepsilon^{\frac{n+i_2}{2}}[(h_e^n+h_p^n)\partial_{x}
+g_p^n \partial_{y}] (v_p^{i_2}+v_e^{i_2+1})-\nu_{4}\varepsilon^{\frac{n}{2}} \partial_x^2(g_p^{n-2}+\sqrt{\varepsilon} g_p^{n-1})\nonumber\\
&-\sum_{i_1=0}^{n-1}\varepsilon^{\frac{n+i_1}{2}}[(h_e^{i_1}+h_p^{i_1})\partial_{x}
+(g_p^{i_1}+g_e^{i_1+1})\partial_{y}] v_p^n-\nu_{4}\varepsilon^{\frac{n}{2}} \partial_x^2(\sqrt{\varepsilon} g_e^n+\varepsilon g_p^n)\nonumber\\
&+\varepsilon^{\frac{n-1}{2}} (R^{g}_{n-1}+E_{r_3}^{n-1}).
\end{align}
Using the estimates of $R^v_{n-1}~\text{and}~  R^g_{n-1} $, we directly have
$$
\|(R^v_{n-1}, R^g_{n-1})\|_{L^2(\Omega_{\varepsilon})} \lesssim \varepsilon^{\frac{1}{16}} .
$$
Next, we estimate the other terms in $(\ref{8.14})$ and get
$$
\begin{aligned}
 & \varepsilon^{\frac{n}{2}} \| [(u_e^0+u_p^0) \partial_x+(v_p^0+v_e^1) \partial_y] v_p^n \|_{L^2(\Omega_{\varepsilon})} \\
\leq & \varepsilon^{\frac{n}{2}}\|(u_e^0, u_p^0, v_p^0, v_e^1)\|_{L^{\infty}}\|\nabla v_p^n\|_{L^2} \\
\leq & C(L, \xi) \varepsilon^{-\xi-\frac{7}{16}+\frac{n}{2}},
\end{aligned}
$$

$$
\begin{aligned}
&\varepsilon^{\frac{n}{2}} \|[(u_e^1+u_p^1) \partial_x+(v_p^1+v_e^2) \partial_y](v_p^{n-1}+v_e^n)\|_{L^2(\Omega_{\varepsilon})}  \\
\leq & \varepsilon^{\frac{n}{2}}\|(u_e^1, u_p^1, v_p^1,v_e^2)\|_{L^{\infty}}(\|\partial_x v_p^{n-1}+\partial_x v_e^n\|_{L^2}+\|\partial_y v_p^{n-1}+\sqrt{\varepsilon} \partial_Y v_e^n\|_{L^2}) \\
\leq & C(L, \xi) \varepsilon^{-\xi-\frac{7}{16}+\frac{n}{2}},\\
&\varepsilon^{\frac{n}{2}} \|[(u_e^{n-1}+u_p^{n-1}) \partial_x+(v_p^{n-1}+v_e^n) \partial_y](v_p^1+v_e^2)\|_{L^2(\Omega_{\varepsilon})}\\
\leq & \varepsilon^{\frac{n}{2}}\|(u_e^{n-1}, u_p^{n-1}, v_p^{n-1},v_e^n)\|_{L^{\infty}}(\|\partial_x v_p^1+\partial_x v_e^2\|_{L^2}+\|\partial_y v_p^1+\sqrt{\varepsilon} \partial_Y v_e^2\|_{L^2}) \\
\leq & C(L, \xi) \varepsilon^{-\xi-\frac{7}{16}+\frac{n}{2}},
\end{aligned}
$$
and
$$
\begin{aligned}
&\varepsilon^{\frac{n}{2}} \|[(h_e^n+h_p^n) \partial_x+g_p^n\partial_y]
(g_p^0+g_e^1)\|_{L^2(\Omega_{\varepsilon})}  \\
 \leq & \varepsilon^{\frac{n}{2}}\|(h_e^n, h_p^n, g_p^n)\|_{L^{\infty}}(\|\partial_x g_p^0+\partial_x g_e^1\|_{L^2}+\|\partial_y g_p^0+\sqrt{\varepsilon} \partial_Y g_e^1\|_{L^2}) \\
\leq & C(L, \xi) \varepsilon^{-\xi-\frac{7}{16}+\frac{n}{2}}.
\end{aligned}
$$
For the last two terms of $(\ref{8.14})$, we have
$$
\|\varepsilon^{\frac{n}{2}} \partial_{y y} v_p^n
+\varepsilon^{\frac{n}{2}} \partial_x^2(v_p^{n-2}+\sqrt{\varepsilon}v_p^{n-1}+\sqrt{\varepsilon}v_e^n
+\varepsilon v_p^n)\|_{L^2(\Omega_{\varepsilon})} \leq C(L) \varepsilon^{-\xi-\frac{7}{16}+\frac{n}{2}},
$$
in which we use the estimates of $v_p^{n-2}, v_p^{n-1}, v_p^n ~\text{and}~ v_e^n$. Hence, we obtain
$$
\|R_{app}^2\|_{L^2(\Omega_{\varepsilon})} \leq C(L, \xi) \varepsilon^{-\xi-\frac{7}{16}+\frac{n}{2}} .
$$
Similar arguments imply that
$$
\|R_{a p p}^4\|_{L^2(\Omega_{\varepsilon})} \leq C(L, \xi) \varepsilon^{-\xi-\frac{7}{16}+\frac{n}{2}} .
$$
Combining all the previous estimates, we obtain $(\ref{8.12})$.
\end{proof}


\section{The proof of the main result}


Each term in the expansions $(\ref{1.7})$ has already been previously constructed.
In this section, we  prove bounds on the remainder  $(u^{\varepsilon}, v^{\varepsilon}, h^{\varepsilon}, g^{\varepsilon}, p^{\varepsilon})$.
Let us define
\begin{align}\label{9.1}
\left\{\begin{aligned}
u_s(x, y)=& U^{\varepsilon}-\varepsilon^{\frac{n}{2}}u_p^n
-\varepsilon^{\gamma+\frac{n}{2}}u^{\varepsilon},
\quad v_s(x, y)=V^{\varepsilon}-\varepsilon^{\frac{n}{2}}v_p^n
-\varepsilon^{\gamma+\frac{n}{2}}v^{\varepsilon},\\
h_s(x, y)=&H^{\varepsilon}-\varepsilon^{\frac{n}{2}}h_p^n
-\varepsilon^{\gamma+\frac{n}{2}}h^{\varepsilon},
\quad g_s(x, y)=G^{\varepsilon}-\varepsilon^{\frac{n}{2}}g_p^n
-\varepsilon^{\gamma+\frac{n}{2}}g^{\varepsilon}.
\end{aligned}\right.
\end{align}
Substituting $(\ref{9.1})$ into $(\ref{1.5})$, we obtain that $(u^{\varepsilon}, v^{\varepsilon}, h^{\varepsilon}, g^{\varepsilon},  p^{\varepsilon})$ satisfy the following system
\begin{align}\label{9.2}
\left\{\begin{aligned}
&u_s u_x^{\varepsilon}+u^{\varepsilon} u_{s x}+v_s u_y^{\varepsilon}+v^{\varepsilon} u_{s y}+p_x^{\varepsilon}-\nu_{1}\Delta_{\varepsilon} u^{\varepsilon}\\
& \quad -h_s h_x^{\varepsilon}-h^{\varepsilon} h_{s x}-g_s h_y^{\varepsilon}-g^{\varepsilon} h_{s y}=R_1(u^{\varepsilon}, v^{\varepsilon},h^{\varepsilon}, g^{\varepsilon}), \\
&u_s v_x^{\varepsilon}+u^{\varepsilon} v_{s x}+v_s v_y^{\varepsilon}+v^{\varepsilon} v_{s y}
+\frac{p_y^{\varepsilon}}{\varepsilon}
-\nu_{2}\Delta_{\varepsilon} v^{\varepsilon}\\
& \quad -h_s g_x^{\varepsilon}-h^{\varepsilon} g_{s x}-g_s g_y^{\varepsilon}-g^{\varepsilon} g_{s y}=R_2(u^{\varepsilon}, v^{\varepsilon},h^{\varepsilon}, g^{\varepsilon}), \\
&u_s h_x^{\varepsilon}+u^{\varepsilon} h_{s x}+v_s h_y^{\varepsilon}+v^{\varepsilon} h_{s y}-\nu_{3}\Delta_{\varepsilon} h^{\varepsilon}\\
& \quad -h_s u_x^{\varepsilon}-h^{\varepsilon} u_{s x}-g_s u_y^{\varepsilon}-g^{\varepsilon} u_{s y}=R_3(u^{\varepsilon}, v^{\varepsilon},h^{\varepsilon}, g^{\varepsilon}), \\
&u_s g_x^{\varepsilon}+u^{\varepsilon} g_{s x}+v_s g_y^{\varepsilon}+v^{\varepsilon} g_{s y}-\nu_{4}\Delta_{\varepsilon} g^{\varepsilon}\\
& \quad -h_s v_x^{\varepsilon}-h^{\varepsilon} v_{s x}-g_s v_y^{\varepsilon}-g^{\varepsilon} v_{s y}=R_4(u^{\varepsilon}, v^{\varepsilon},h^{\varepsilon}, g^{\varepsilon}), \\
&u_x^{\varepsilon}+v_y^{\varepsilon}=h_x^{\varepsilon}+g_y^{\varepsilon}=0.
\end{aligned}\right.
\end{align}
The  source terms $R_i$  $(i=1,2,3,4)$ are defined by
\begin{align}\label{9.3}
 R_1(u^{\varepsilon}, v^{\varepsilon},h^{\varepsilon}, g^{\varepsilon}):=&\varepsilon^{-\gamma-\frac{n}{2}} R_{a p p}^1-\varepsilon^{\frac{n}{2}}\left[(u_p^n+\varepsilon^\gamma u^{\varepsilon}) u_x^{\varepsilon}+u^{\varepsilon} u_{p x}^n+(v_p^n+\varepsilon^\gamma v^{\varepsilon}) u_y^{\varepsilon}+v^{\varepsilon} u_{p y}^n \right. \nonumber\\
 & \left.- (h_p^n+\varepsilon^\gamma h^{\varepsilon}) h_x^{\varepsilon}-h^{\varepsilon} h_{p x}^n-(g_p^n+\varepsilon^\gamma g^{\varepsilon}) h_y^{\varepsilon}-g^{\varepsilon} h_{p y}^n\right],\nonumber\\
 R_2(u^{\varepsilon}, v^{\varepsilon},h^{\varepsilon}, g^{\varepsilon}):=&\varepsilon^{-\gamma-\frac{n}{2}} R_{a p p}^2-\varepsilon^{\frac{n}{2}}\left[(u_p^n+\varepsilon^\gamma u^{\varepsilon}) v_x^{\varepsilon}+u^{\varepsilon} v_{p x}^n+(v_p^n+\varepsilon^\gamma v^{\varepsilon}) v_y^{\varepsilon}+v^{\varepsilon} v_{p y}^n \right. \nonumber\\
 & \left.- (h_p^n+\varepsilon^\gamma h^{\varepsilon}) g_x^{\varepsilon}-h^{\varepsilon} g_{p x}^n
 -(g_p^n+\varepsilon^\gamma g^{\varepsilon}) g_y^{\varepsilon}-g^{\varepsilon} g_{p y}^n\right], \\
 R_3(u^{\varepsilon}, v^{\varepsilon},h^{\varepsilon}, g^{\varepsilon}):=&\varepsilon^{-\gamma-\frac{n}{2}} R_{a p p}^3-\varepsilon^{\frac{n}{2}}\left[(u_p^n+\varepsilon^\gamma u^{\varepsilon}) h_x^{\varepsilon}+u^{\varepsilon} h_{p x}^n+(v_p^n+\varepsilon^\gamma v^{\varepsilon}) h_y^{\varepsilon}+v^{\varepsilon} h_{p y}^n \right.\nonumber\\
 & \left. -(h_p^n+\varepsilon^\gamma h^{\varepsilon}) u_x^{\varepsilon}-h^{\varepsilon} u_{p x}^n-(g_p^n+\varepsilon^\gamma g^{\varepsilon}) u_y^{\varepsilon}-g^{\varepsilon} u_{p y}^n\right], \nonumber\\
 R_4(u^{\varepsilon}, v^{\varepsilon},h^{\varepsilon}, g^{\varepsilon}):=&\varepsilon^{-\gamma-\frac{n}{2}} R_{a p p}^4-\varepsilon^{\frac{n}{2}}\left[(u_p^n+\varepsilon^\gamma u^{\varepsilon}) g_x^{\varepsilon}+u^{\varepsilon} g_{p x}^n
 +(v_p^n+\varepsilon^\gamma v^{\varepsilon}) g_y^{\varepsilon}
 +v^{\varepsilon} g_{p y}^n \right. \nonumber\\
 & \left.-(h_p^n+\varepsilon^\gamma h^{\varepsilon}) v_x^{\varepsilon}-h^{\varepsilon} v_{p x}^n
 -(g_p^n+\varepsilon^\gamma g^{\varepsilon}) v_y^{\varepsilon}-g^{\varepsilon} v_{p y}^n\right],\nonumber
\end{align}
in which the error terms $R_{app}^i (i=1,2,3,4)$ have been estimated in Proposition $\ref{b8.4}$.

By the definition of $u_s, h_s$ in $(\ref{9.1})$ and the previously obtained estimates of $\left(u_p^i, v_p^i, h_p^i, g_p^i\right)(i=0,1,\ldots,n)$ and $(u_e^j, v_e^j,h_e^j,g_e^j)(j=1,2,\ldots,n)$, for small enough $\sigma_0 > 0$, we get
$\|y \partial_{y} (u_s, h_s)\|_{L^{\infty}} < C  \sigma_0$.
In addition, under assumption (H), we obtain $\left\|h_s / u_s\right\|_{L^{\infty}} \ll 1$. These two estimates will be used in the proof of Proposition $\ref{b9.1}$.

Next, we state two propositions which will be used later in the proof of  Theorem $\ref{b1.1}$. 

\begin{pro}\label{b9.1}
For any given $f_i \in L^2(\Omega_{\varepsilon}) (i=1,2,3,4)$, there exists $L>0$ such that the following linear problem
\begin{align}\label{9.4}
\left\{\begin{array}{l}
u_s u_x+u u_{s x}+v_s u_y+v u_{s y}-(h_s h_x+h h_{s x}+g_s h_y+g h_{s y})+p_x-\nu_{1}\Delta_{\varepsilon} u=f_1, \\
u_s v_x+u v_{s x}+v_s v_y+v v_{s y}-(h_s g_x+h g_{s x}+g_s g_y+g g_{s y})+\varepsilon^{-1} p_y -\nu_{2}\Delta_{\varepsilon} v=f_2,\\
u_s h_x+u h_{s x}+v_s h_y+v h_{s y}-(h_s u_x+h u_{s x}+g_s u_y+g u_{s y})-\nu_{3}\Delta_{\varepsilon} h=f_3,  \\
u_s g_x+u g_{s x}+v_s g_y+v g_{s y}-(h_s v_x+h v_{s x}+g_s v_y+g v_{s y})-\nu_{4}\Delta_{\varepsilon} g=f_4,\\
u_x+v_y=h_x+g_y=0,
\end{array}\right.
\end{align}
with boundary conditions
\begin{align}\label{9.5}
\left\{\begin{array}{l}
{(u, v,  h_y, g)(x,0)=0, \quad(u_y, v, h_y, g)(x,\varepsilon^{-1/2})=0,} \\
{(u, v, h, g)(0,y)=0, \quad(p-2 \nu_{1}\varepsilon u_x, u_y+\nu_{2}\varepsilon v_x, h, g_x)(L,y)=0,}
\end{array}\right.
\end{align}
has a unique solution $(u, v, h, g, p)$ in the domain $\Omega_{\varepsilon}$, which satisfies
\begin{align}\label{9.6}
&\|\nabla_{\varepsilon} u\|_{L^2(\Omega_{\varepsilon})}+\|\nabla_{\varepsilon} v\|_{L^2(\Omega_{\varepsilon})}
+\|\nabla_{\varepsilon} h\|_{L^2(\Omega_{\varepsilon})}+\|\nabla_{\varepsilon} g\|_{L^2(\Omega_{\varepsilon})}\nonumber\\
\lesssim &\|f_1\|_{L^2(\Omega_{\varepsilon})}
+\|f_3\|_{L^2(\Omega_{\varepsilon})}
+\sqrt{\varepsilon}(\|f_2\|_{L^2(\Omega_{\varepsilon})}
+\|f_4\|_{L^2(\Omega_{\varepsilon})}).
\end{align}
\end{pro}

\begin{proof}
Under assumption (H), the above proposition is a direct consequence of the following two Lemmas.
\end{proof}
We recall the following two lemmas, proved in \cite{DLX2021}.

\begin{Lemma}\label{b9.1.1}
Let $(u, v, h, g)$ be the solution to the linear problem $(\ref{9.4})$, and assume that $\varepsilon \ll L$, then the following estimate holds
\begin{align}\label{9.1.1.1}
&\nu_{1}\|\nabla_{\varepsilon}u\|_{L^2(\Omega_{\varepsilon})}^{2}
+\nu_{3}\|\nabla_{\varepsilon}h\|_{L^2(\Omega_{\varepsilon})}^{2}
+\int_{x=L} u_s(|u|^{2}+|h|^{2}+\varepsilon |v|^{2}+\varepsilon |g|^{2})\mathrm{d}y \nonumber\\
& \lesssim  L (\|\nabla_{\varepsilon} v\|_{L^2(\Omega_{\varepsilon})}^{2}+\|\nabla_{\varepsilon} g\|_{L^2(\Omega_{\varepsilon})}^{2})
+\|(f_{1},f_{3})\|_{L^2(\Omega_{\varepsilon})}^{2}
+\varepsilon\|(f_{2},f_{4})\|_{L^2(\Omega_{\varepsilon})}^{2}.
\end{align}
\end{Lemma}

\begin{Lemma}\label{b9.1.2}
Let $(u, v, h, g)$ be the solution to the linear problem $(\ref{9.4})$, 
then
\begin{align}\label{9.1.2.1}
&\|\nabla_{\varepsilon}g\|_{L^2(\Omega_{\varepsilon})}^{2} +\|\nabla_{\varepsilon}v\|_{L^2(\Omega_{\varepsilon})}^{2}
+ \int_{x=0}\frac{\varepsilon^{2}(\nu_1 |v_{x}|^2+\nu_3 |g_{x}|^2)}{u_s} \mathrm{d}y
+\varepsilon\int_{x=L}\frac{|v_y|^{2}}{u_s}\mathrm{d}y\nonumber\\
\leq&C\left[\|(f_1,f_3)\|_{L^2(\Omega_{\varepsilon})}^{2}
+\varepsilon\|(f_2,f_4)\|_{L^2(\Omega_{\varepsilon})}^{2}
+\|\nabla_{\varepsilon} u\|_{L^2(\Omega_{\varepsilon})}^{2}
+\|\nabla_{\varepsilon} h\|_{L^2(\Omega_{\varepsilon})}^{2}
\right.\nonumber\\
&\left. +\left(L+\left\|\frac{h_s}{u_s}\right\|_{L^{\infty}}+\|y \partial_{y}(u_s,h_s)\|_{L^{\infty}}\right)\left(\|\nabla_{\varepsilon} v\|_{L^2(\Omega_{\varepsilon})}^{2}
+\|\nabla_{\varepsilon}g\|_{L^2(\Omega_{\varepsilon})}^2\right)
\right].
\end{align}
\end{Lemma}

The following proposition provides  $L^{\infty}$ estimates on the solutions for the nonlinear problem.
\begin{pro}\label{b9.2}
For any given $F_i \in L^2\left(\Omega_{\varepsilon}\right)(i=1,2,3,4)$, suppose that the following system in the domain $\Omega_{\varepsilon}$
\begin{align}\label{9.7}
\left\{\begin{array}{l}
-\nu_{1}\Delta_{\varepsilon} u+p_x=F_1, \\
-\nu_{2}\Delta_{\varepsilon} v+\frac{p_y}{\varepsilon}=F_2, \\
-\nu_{3}\Delta_{\varepsilon} h=F_3,  \\
-\nu_{4}\Delta_{\varepsilon} g=F_4,\\
u_x+v_y=h_x+g_y=0,
\end{array}\right.
\end{align}
has a solution $\left(u, v, h, g\right)$ with the same boundary conditions as that in $(\ref{9.5})$. Then, for any $\gamma>0$,
\begin{align}\label{9.8}
& \|u\|_{L^{\infty}(\Omega_{\varepsilon})}
+\|h\|_{L^{\infty}(\Omega_{\varepsilon})}
+\sqrt{\varepsilon}(\|v\|_{L^{\infty}(\Omega_{\varepsilon})}
+\|g\|_{L^{\infty}(\Omega_{\varepsilon})})\nonumber\\
\lesssim & C_{\gamma, L} \varepsilon^{-\frac{\gamma}{16}}\left[\|\nabla_{\varepsilon} u\|_{L^2(\Omega_{\varepsilon})}+\|\nabla_{\varepsilon} v\|_{L^2(\Omega_{\varepsilon})}+\|\nabla_{\varepsilon} h\|_{L^2(\Omega_{\varepsilon})}+\|\nabla_{\varepsilon} g\|_{L^2(\Omega_{\varepsilon})}\right.\nonumber\\
& \quad \quad \quad \left.+\|F_1\|_{L^2(\Omega_{\varepsilon})}
+\|F_3\|_{L^2(\Omega_{\varepsilon})}
+\sqrt{\varepsilon}(\|F_2\|_{L^2(\Omega_{\varepsilon})}
+\|F_4\|_{L^2(\Omega_{\varepsilon})})\right].
\end{align}
\end{pro}
The detailed proof of this proposition can be obtained by modifying the results of \cite{DLX2021}.

\begin{proof}[Proof of Theorem \ref{b1.1}]
 We will apply the standard contraction mapping principle to prove the existence of the solutions for the nonlinear problem.
First, we define the function space $\mathbb{X}$  by its norm
\begin{align}\label{9.9}
\left\|(u^{\varepsilon}, v^{\varepsilon},h^{\varepsilon}, g^{\varepsilon})\right\|_{\mathbb{X}}:=&\left\|\nabla_{\varepsilon} u^{\varepsilon}\right\|_{L^2(\Omega_{\varepsilon})}
+\left\|\nabla_{\varepsilon}v^{\varepsilon}\right\|_{L^2
(\Omega_{\varepsilon})}+\left\|\nabla_{\varepsilon} h^{\varepsilon}\right\|_{L^2(\Omega_{\varepsilon})}
+\left\|\nabla_{\varepsilon} g^{\varepsilon}\right\|_{L^2(\Omega_{\varepsilon})}\nonumber\\
&+\varepsilon^{\frac{\gamma}{8}}\left\|u^{\varepsilon}\right\|_{L^{\infty}(\Omega_{\varepsilon})}
+\varepsilon^{\frac{\gamma}{8}}\left\|h^{\varepsilon}\right\|_{L^{\infty}(\Omega_{\varepsilon})}
+\varepsilon^{\frac{1}{2}+\frac{\gamma}{8}}(\left\|v^{\varepsilon}\right\|_{L^{\infty}
(\Omega_{\varepsilon})}+\left\|g^{\varepsilon}\right\|_{L^{\infty}
(\Omega_{\varepsilon})}),
\end{align}
where $\nabla_{\varepsilon}:=\partial_y+\sqrt{\varepsilon} \partial_x$.
Next, we select a subspace of $\mathbb{X}$
$$
\mathbb{X}_K:=\left\{\left.(u^{\varepsilon}, v^{\varepsilon},h^{\varepsilon}, g^{\varepsilon}) \in \mathbb{X} \right| \left\|(u^{\varepsilon}, v^{\varepsilon},h^{\varepsilon}, g^{\varepsilon})\right\|_{\mathbb{X}} \leq K\right\},
$$
where $K$ will be fixed later.
For each $(u^{\varepsilon}, v^{\varepsilon},h^{\varepsilon}, g^{\varepsilon}) $ of $\mathbb{X}_K$, we solve the following linear problem:
\begin{align}\label{9.10}
\left\{\begin{array}{l}
u_s \bar{u}_x^{\varepsilon}+\bar{u}^{\varepsilon} u_{s x}+v_s \bar{u}_y^{\varepsilon}+\bar{v}^{\varepsilon} u_{s y}-(h_s \bar{h}_x^{\varepsilon}+\bar{h}^{\varepsilon} h_{s x}+g_s \bar{h}_y^{\varepsilon}+\bar{g}^{\varepsilon} h_{s y})+\bar{p}_x^{\varepsilon}-\nu_{1}\Delta_{\varepsilon} \bar{u}^{\varepsilon}=R_1(u^{\varepsilon}, v^{\varepsilon},h^{\varepsilon}, g^{\varepsilon}), \\
u_s \bar{v}_x^{\varepsilon}+\bar{u}^{\varepsilon} v_{s x}+v_s \bar{v}_y^{\varepsilon}+\bar{v}^{\varepsilon} v_{s y}
-(h_s \bar{g}_x^{\varepsilon}+\bar{h}^{\varepsilon} g_{s x}+g_s \bar{g}_y^{\varepsilon}+\bar{g}^{\varepsilon} g_{s y})+\frac{\bar{p}_y^{\varepsilon}}{\varepsilon}-\nu_{2}\Delta_{\varepsilon} \bar{v}^{\varepsilon}=R_2(u^{\varepsilon},v^{\varepsilon},
h^{\varepsilon},g^{\varepsilon}), \\
u_s \bar{h}_x^{\varepsilon}+\bar{u}^{\varepsilon} h_{s x}+v_s \bar{h}_y^{\varepsilon}+\bar{v}^{\varepsilon} h_{s y}-(h_s \bar{u}_x^{\varepsilon}+\bar{h}^{\varepsilon} u_{s x}+g_s \bar{u}_y^{\varepsilon}+\bar{g}^{\varepsilon} u_{s y})-\nu_{3}\Delta_{\varepsilon} \bar{h}^{\varepsilon}=R_3(u^{\varepsilon}, v^{\varepsilon},h^{\varepsilon}, g^{\varepsilon}), \\
u_s \bar{g}_x^{\varepsilon}+\bar{u}^{\varepsilon} g_{s x}+v_s \bar{g}_y^{\varepsilon}+\bar{v}^{\varepsilon} g_{s y}
-(h_s \bar{v}_x^{\varepsilon}+\bar{h}^{\varepsilon} v_{s x}+g_s \bar{v}_y^{\varepsilon}+\bar{g}^{\varepsilon} v_{s y})-\nu_{4}\Delta_{\varepsilon} \bar{g}^{\varepsilon}=R_4(u^{\varepsilon},v^{\varepsilon},
h^{\varepsilon},g^{\varepsilon}), \\
\bar{u}_x^{\varepsilon}+\bar{v}_y^{\varepsilon}
=\bar{h}_x^{\varepsilon}+\bar{g}_y^{\varepsilon}=0.
\end{array}\right.
\end{align}
By using the Proposition ${\ref{b9.1}}$, we have
\begin{align}\label{9.11}
&\|\nabla_{\varepsilon} \bar{u}^{\varepsilon}\|_{L^2(\Omega_{\varepsilon})}
+\|\nabla_{\varepsilon} \bar{v}^{\varepsilon}\|_{L^2(\Omega_{\varepsilon})}
+\|\nabla_{\varepsilon} \bar{h}^{\varepsilon}\|_{L^2(\Omega_{\varepsilon})}
+\|\nabla_{\varepsilon} \bar{g}^{\varepsilon}\|_{L^2(\Omega_{\varepsilon})}\nonumber\\
\lesssim &\|R_1\|_{L^2(\Omega_{\varepsilon})}
+\|R_3\|_{L^2(\Omega_{\varepsilon})}
+\sqrt{\varepsilon}(\|R_2\|_{L^2(\Omega_{\varepsilon})}
+\|R_4\|_{L^2(\Omega_{\varepsilon})}).
\end{align}
We now estimate $R_i (i=1,2,3,4)$ of $(\ref{9.11})$. Applying the Proposition $\ref{b8.4}$, it is clear that
$$
\varepsilon^{-\gamma-\frac{n}{2}}(\|R_{a p p}^1\|_{L^2(\Omega_{\varepsilon})}
+\|R_{a p p}^3\|_{L^2(\Omega_{\varepsilon})}
+\sqrt{\varepsilon}\|R_{a p p}^2\|_{L^2(\Omega_{\varepsilon})}
+\sqrt{\varepsilon}\|R_{a p p}^4\|_{L^2(\Omega_{\varepsilon})}) \leq C(L, \xi) \varepsilon^{-\gamma-\xi+\frac{1}{16}}.
$$
For $R_1$, the other terms are estimated as follows
$$
\begin{aligned}
& \varepsilon^{\frac{n}{2}}\|(u_p^n+\varepsilon^\gamma u^{\varepsilon}) u_x^{\varepsilon}+(v_p^n+\varepsilon^\gamma v^{\varepsilon}) u_y^{\varepsilon}\|_{L^2(\Omega_{\varepsilon})} \\
\leq & \varepsilon^{\frac{n}{2}}[(\|u_p^n\|_{L^{\infty}}
+\varepsilon^\gamma\|u^{\varepsilon}\|_{L^{\infty}})
\|v_y^{\varepsilon}\|_{L^2}+(\|v_p^n\|_{L^{\infty}}
+\varepsilon^\gamma \|v^{\varepsilon}\|_{L^{\infty}})
\|u_y^{\varepsilon}\|_{L^2}] \\
\leq & \varepsilon^{\frac{n}{2}}\|(u_p^n, v_p^n)\|_{L^{\infty}}\|(u^{\varepsilon}, v^{\varepsilon}, h^{\varepsilon}, g^{\varepsilon})\|_{\mathbb{X}}
+\varepsilon^{\frac{7 \gamma}{8}+\frac{n-1}{2}}\|(u^{\varepsilon}, v^{\varepsilon}, h^{\varepsilon}, g^{\varepsilon})\|_{\mathbb{X}}^2 \\
\leq & C(L, \xi) \varepsilon^{\frac{n}{2}-\frac{3}{16}-\xi} K+\varepsilon^{\frac{7 \gamma}{8}+\frac{n-1}{2}}K^2,\\
& \varepsilon^{\frac{n}{2}}\|(h_p^n+\varepsilon^\gamma h^{\varepsilon}) h_x^{\varepsilon}+(g_p^n+\varepsilon^\gamma g^{\varepsilon}) h_y^{\varepsilon}\|_{L^2(\Omega_{\varepsilon})} \\
\leq & \varepsilon^{\frac{n}{2}}[(\|h_p^n\|_{L^{\infty}}
+\varepsilon^\gamma\|h^{\varepsilon}\|_{L^{\infty}})
\|g_y^{\varepsilon}\|_{L^2}+(\|g_p^n\|_{L^{\infty}}
+\varepsilon^\gamma \|g^{\varepsilon}\|_{L^{\infty}})
\|h_y^{\varepsilon}\|_{L^2}] \\
\leq & \varepsilon^{\frac{n}{2}}\|(h_p^n, g_p^n)
\|_{L^{\infty}}\|(u^{\varepsilon}, v^{\varepsilon}, h^{\varepsilon},g^{\varepsilon})\|_{\mathbb{X}}
+\varepsilon^{\frac{7 \gamma}{8}+\frac{n-1}{2}}\|(u^{\varepsilon}, v^{\varepsilon}, h^{\varepsilon}, g^{\varepsilon})\|_{\mathbb{X}}^2 \\
\leq & C(L, \xi) \varepsilon^{\frac{n}{2}-\frac{3}{16}-\xi} K+\varepsilon^{\frac{7 \gamma}{8}+\frac{n-1}{2}}K^2,
\end{aligned}
$$
in which we use the divergence free conditions and the estimates of $(u_p^n, v_p^n, h_p^n, g_p^n)$,
and
\begin{align*}
&\varepsilon^{\frac{n}{2}}\|u^{\varepsilon} u_{p x}^n
+v^{\varepsilon} u_{p y}^n +h^{\varepsilon} h_{p x}^n+g^{\varepsilon} h_{p y}^n\|_{L^2(\Omega_{\varepsilon})} \\
\leq & C \varepsilon^{\frac{n}{2}}\left[\|u_y^{\varepsilon}
\|_{L^2} \sup_x\|\langle y\rangle^l u_{p x}^n\|_{L^2}+\|v_y^{\varepsilon}\|_{L^2} \sup_x\|\langle y\rangle^l u_{p y}^n\|_{L^2}\right. \\
 & \left.+\|h_y^{\varepsilon}\|_{L^2} \sup_x\|\langle y\rangle^l h_{p x}^n\|_{L^2}
 +\|g_y^{\varepsilon}\|_{L^2} \sup_x\|\langle y\rangle^l h_{p y}^n\|_{L^2}\right]\\
 \leq & C(L, \xi) \varepsilon^{\frac{n}{2}-\frac{3}{16}-\xi}\|(u^{\varepsilon}, v^{\varepsilon},h^{\varepsilon}, g^{\varepsilon})\|_ \mathbb{X} \leq C(L, \xi) \varepsilon^{\frac{n}{2}-\frac{3}{16}-\xi} K,
\end{align*}
in which we have used the fact that $\left|(u^{\varepsilon}, v^{\varepsilon}, h^{\varepsilon}, g^{\varepsilon})\right| \leq \sqrt{y}\left\|(u_y^{\varepsilon}, v_y^{\varepsilon},h_y^{\varepsilon}, g_y^{\varepsilon})\right\|_{L^2(0, \varepsilon^{-1/2})}$.

Using  similar methods to estimate $R_2$, we obtain
$$
\begin{aligned}
& \varepsilon^{\frac{n}{2}}\|(u_p^n+\varepsilon^\gamma u^{\varepsilon}) v_x^{\varepsilon}+(v_p^n+\varepsilon^\gamma v^{\varepsilon}) v_y^{\varepsilon}\|_{L^2(\Omega_{\varepsilon})} \\
\leq & \varepsilon^{\frac{n-1}{2}} (\|u_p^n\|_{L^{\infty}}
+\varepsilon^\gamma\|u^{\varepsilon}\|_{L^{\infty}})
\|\sqrt{\varepsilon} v_x^{\varepsilon}\|_{L^2}
+\varepsilon^{\frac{n}{2}}(\|v_p^n\|_{L^{\infty}}
+\varepsilon^\gamma\|v^{\varepsilon}\|_{L^{\infty}})
\|v_y^{\varepsilon}\|_{L^2} \\
\leq & C \varepsilon^{\frac{n-1}{2}}\|(u_p^n, v_p^n)\|_{L^{\infty}}\|(u^{\varepsilon}, v^{\varepsilon},h^{\varepsilon}, g^{\varepsilon})\|_\mathbb{X}
+\varepsilon^{\frac{7 \gamma}{8}+\frac{n-1}{2}}\|(u^{\varepsilon}, v^{\varepsilon},h^{\varepsilon}, g^{\varepsilon})\|_{\mathbb{X}}^2 \\
\leq & C(L, \xi) \varepsilon^{\frac{n-1}{2}-\frac{3}{16}-\xi} K
+\varepsilon^{\frac{7 \gamma}{8}+\frac{n-1}{2}} K^2,
\end{aligned}
$$

$$
\begin{aligned}
& \varepsilon^{\frac{n}{2}}\|(h_p^n+\varepsilon^\gamma h^{\varepsilon}) g_x^{\varepsilon}+(g_p^n+\varepsilon^\gamma g^{\varepsilon}) g_y^{\varepsilon}\|_{L^2(\Omega_{\varepsilon})} \\
\leq & \varepsilon^{\frac{n-1}{2}} (\|h_p^n\|_{L^{\infty}}
+\varepsilon^\gamma\|h^{\varepsilon}\|_{L^{\infty}})
\|\sqrt{\varepsilon} g_x^{\varepsilon}\|_{L^2}
+\varepsilon^{\frac{n}{2}}(\|g_p^n\|_{L^{\infty}}
+\varepsilon^\gamma\|g^{\varepsilon}\|_{L^{\infty}})
\|g_y^{\varepsilon}\|_{L^2} \\
\leq & C \varepsilon^{\frac{n-1}{2}} \|(h_p^n,g_p^n)\|_{L^{\infty}}\|(u^{\varepsilon}, v^{\varepsilon},h^{\varepsilon}, g^{\varepsilon})\|_\mathbb{X}
+\varepsilon^{\frac{7 \gamma}{8}+\frac{n-1}{2}}\|(u^{\varepsilon}, v^{\varepsilon},h^{\varepsilon},
g^{\varepsilon})\|_{\mathbb{X}}^2 \\
\leq & C(L, \xi) \varepsilon^{\frac{n-1}{2}-\frac{3}{16}-\xi} K
+\varepsilon^{\frac{7 \gamma}{8}+\frac{n-1}{2}} K^2,
\end{aligned}
$$
and
\begin{align*}
& \varepsilon^{\frac{n}{2}}\|u^{\varepsilon} v_{p x}^n+v^{\varepsilon} v_{p y}^n+h^{\varepsilon} g_{p x}^n+g^{\varepsilon} g_{p y}^n\|_{L^2(\Omega_{\varepsilon})} \\
\leq & C \varepsilon^{\frac{n}{2}}\left[\|u^{\varepsilon}\|_{L^{\infty}}
\|v_{p x}^n\|_{L^2}+\|v_y^{\varepsilon}\|_{L^2} \sup_x\|\langle y\rangle^l v_{p y y}^n\|_{L^2}\right. \\
& \left. +\|h^{\varepsilon}\|_{L^{\infty}}
\|g_{p x}^n\|_{L^2}+\|g_y^{\varepsilon}\|_{L^2} \sup_x\|\langle y\rangle^l g_{p y y}^n\|_{L^2}\right]\\
\leq & C(L, \xi) \varepsilon^{\frac{n-1}{2}-\frac{3}{16}-\xi-\frac{\gamma}{8}}\|(u^{\varepsilon}, v^{\varepsilon}, h^{\varepsilon}, g^{\varepsilon})\|_ \mathbb{X} \leq C(L, \xi) \varepsilon^{\frac{n-1}{2}-\frac{3}{16}-\xi-\frac{\gamma}{8}} K.
\end{align*}

Therefore, we obtain
\begin{align}\label{9.12}
& \|R_1(u^{\varepsilon}, v^{\varepsilon},h^{\varepsilon}, g^{\varepsilon})\|_{L^2(\Omega_{\varepsilon})}
+\sqrt{\varepsilon}\|R_2(u^{\varepsilon}, v^{\varepsilon},h^{\varepsilon}, g^{\varepsilon})\|_{L^2(\Omega_{\varepsilon})} \nonumber\\
& \leq C(L, \xi) \varepsilon^{-\xi-\gamma+\frac{1}{16}}+C(L, \xi) \varepsilon^{\frac{n}{2}-\frac{3}{16}-\xi-\frac{\gamma}{8}} K+\varepsilon^{\frac{7 \gamma}{8}+\frac{n-1}{2}} K^2.
\end{align}
Similarly, we have
\begin{align}\label{9.13}
& \|R_3(u^{\varepsilon}, v^{\varepsilon},h^{\varepsilon}, g^{\varepsilon})\|_{L^2(\Omega_{\varepsilon})}
+\sqrt{\varepsilon}\|R_4(u^{\varepsilon}, v^{\varepsilon},h^{\varepsilon}, g^{\varepsilon})\|_{L^2(\Omega_{\varepsilon})} \nonumber\\
& \leq C(L, \xi) \varepsilon^{-\xi-\gamma+\frac{1}{16}}+C(L, \xi) \varepsilon^{\frac{n}{2}-\frac{3}{16}-\xi-\frac{\gamma}{8}} K+\varepsilon^{\frac{7 \gamma}{8}+\frac{n-1}{2}} K^2.
\end{align}
This leads to
\begin{align}\label{9.14}
& \|\nabla_{\varepsilon}\bar{u}^{\varepsilon}\|_{L^2(\Omega_{\varepsilon})}
+\|\nabla_{\varepsilon} \bar{v}^{\varepsilon}\|_{L^2(\Omega_{\varepsilon})}
+\|\nabla_{\varepsilon} \bar{h}^{\varepsilon}\|_{L^2(\Omega_{\varepsilon})}
+\|\nabla_{\varepsilon}\bar{g}^{\varepsilon}\|_{L^2(\Omega_{\varepsilon})}\nonumber\\
\leq & C(u_s, v_s,h_s, g_s, L, \xi) \varepsilon^{-\xi-\gamma+\frac{1}{16}}+C(u_s, v_s, h_s, g_s, L, \xi) \varepsilon^{\frac{n}{2}-\frac{3}{16}-\xi-\frac{\gamma}{8}} K \nonumber\\
&+C(u_s, v_s, h_s, g_s) \varepsilon^{\frac{7 \gamma}{8}+\frac{n-1}{2}} K^2.
\end{align}
We now give  $L^{\infty}$ bounds of $(\bar{u}^{\varepsilon}, \bar{v}^{\varepsilon},\bar{h}^{\varepsilon}, \bar{g}^{\varepsilon})$. From the equations $(\ref{9.2})$ and $(\ref{9.7})$, we get
\begin{align}\label{9.15}
\left\{\begin{aligned}
& F_1:=R_1-(u_s \bar{u}_x^{\varepsilon}+u_{s x} \bar{u}^{\varepsilon}+v_s \bar{u}_y^{\varepsilon}+\bar{v}^{\varepsilon} u_{s y})+(h_s \bar{h}_x^{\varepsilon}+h_{s x} \bar{h}^{\varepsilon}+g_s \bar{h}_y^{\varepsilon}+\bar{g}^{\varepsilon} h_{s y}), \\
& F_2:=R_2-(u_s \bar{v}_x^{\varepsilon}+v_{s x} \bar{u}^{\varepsilon}+v_s \bar{v}_y^{\varepsilon}+\bar{v}^{\varepsilon} v_{s y})+(h_s \bar{g}_x^{\varepsilon}+g_{s x} \bar{h}^{\varepsilon}+g_s \bar{g}_y^{\varepsilon}+\bar{g}^{\varepsilon} g_{s y}),\\
& F_3:=R_3-(u_s \bar{h}_x^{\varepsilon}+h_{s x} \bar{u}^{\varepsilon}+v_s \bar{h}_y^{\varepsilon}+\bar{v}^{\varepsilon} h_{s y})+(h_s \bar{u}_x^{\varepsilon}+u_{s x} \bar{h}^{\varepsilon}+g_s \bar{u}_y^{\varepsilon}+\bar{g}^{\varepsilon} u_{s y}), \\
& F_4:=R_4-(u_s \bar{g}_x^{\varepsilon}+g_{s x} \bar{u}^{\varepsilon}+v_s \bar{g}_y^{\varepsilon}+\bar{v}^{\varepsilon} g_{s y})+(h_s \bar{v}_x^{\varepsilon}+v_{s x} \bar{h}^{\varepsilon}+g_s \bar{v}_y^{\varepsilon}+\bar{g}^{\varepsilon} v_{s y}).
\end{aligned}\right.
\end{align}
According to the Proposition $\ref{b9.2}$, we obtain
\begin{align}\label{9.16}
& \varepsilon^{\frac{\gamma}{8}}\|\bar{u}^{\varepsilon}\|_{L^{\infty}
(\Omega_{\varepsilon})}
+\varepsilon^{\frac{\gamma}{8}}\|\bar{h}^{\varepsilon}\|_{L^{\infty}(\Omega_{\varepsilon})}
+\varepsilon^{\frac{1}{2}+\frac{\gamma}{8}}(\|\bar{v}^{\varepsilon}\|_{L^{\infty}
(\Omega_{\varepsilon})}+\|\bar{g}^{\varepsilon}\|_{L^{\infty}
(\Omega_{\varepsilon})}) \nonumber\\
\leq & C_{\gamma, L} \varepsilon^{\frac{\gamma}{16}}(\|\nabla_{\varepsilon} \bar{u}^{\varepsilon}\|_{L^2}+\|\nabla_{\varepsilon} \bar{v}^{\varepsilon}\|_{L^2}+\|\nabla_{\varepsilon} \bar{h}^{\varepsilon}\|_{L^2}+\|\nabla_{\varepsilon} \bar{g}^{\varepsilon}\|_{L^2}) \nonumber\\
& +C_{\gamma, L} \varepsilon^{\frac{\gamma}{16}}(\|F_1\|_{L^2}
+\|F_3\|_{L^2})+C_{\gamma, L} \varepsilon^{\frac{1}{2}+\frac{\gamma}{16}}(\|F_2\|_{L^2}
+\|F_4\|_{L^2}).
\end{align}
Using $(\ref{9.14})$, we  obtain the $L^2$ estimate on $(\nabla_{\varepsilon} \bar{u}^{\varepsilon}, \nabla_{\varepsilon} \bar{v}^{\varepsilon},\nabla_{\varepsilon} \bar{h}^{\varepsilon}, \nabla_{\varepsilon} \bar{g}^{\varepsilon})$.
 Next, we establaish estimates of $F_i ( i=1,2,3,4)$ in $(\ref{9.16})$.
We have already obtained the estimates of $R_i(i=1,2,3,4)$, hence, we only estimate the remaining terms of $F_i$. For the remaining terms of $F_{1}$  in $(\ref{9.15})$, we have
\begin{align}\label{9.18}
\|u_{s x} \bar{u}^{\varepsilon}+u_{s y} \bar{v}^{\varepsilon}\|_{L^2(\Omega_{\varepsilon})}
\leq &(\|\bar{u}_y^{\varepsilon}\|_{L^2} \sup_x\|\sqrt{y} u_{s x}\|_{L^2}
+\|\bar{v}_y^{\varepsilon}\|_{L^2} \sup_x\|\sqrt{y} u_{s y}\|_{L^2}) \nonumber\\
\leq & C(L, u_{s}, v_{s}, h_{s}, g_{s}) \|(\nabla_{\varepsilon} \bar{u}^{\varepsilon}, \nabla_{\varepsilon} \bar{v}^{\varepsilon})\|_{L^2}, \nonumber\\
\|u_s \bar{u}_x^{\varepsilon}+v_s \bar{u}_y^{\varepsilon}\|_{L^2(\Omega_{\varepsilon})}
\leq & \|(u_s, v_s)\|_{L^{\infty}}
\|(\bar{u}_y^{\varepsilon}, \bar{v}_y^{\varepsilon})\|_{L^2}\nonumber\\
\leq & C (L, u_{s}, v_{s}, h_{s}, g_{s}) \|(\nabla_{\varepsilon} \bar{u}^{\varepsilon}, \nabla_{\varepsilon} \bar{v}^{\varepsilon})\|_{L^2},\\
\|h_{s x} \bar{h}^{\varepsilon}+h_{s y}
\bar{g}^{\varepsilon}\|_{L^2(\Omega_{\varepsilon})}
\leq &(\|\bar{h}_y^{\varepsilon}\|_{L^2} \sup_x\|\sqrt{y} h_{s x}\|_{L^2}
+\|\bar{g}_y^{\varepsilon}\|_{L^2} \sup_x\|\sqrt{y} h_{s y}\|_{L^2}) \nonumber\\
\leq & C (L, u_{s}, v_{s}, h_{s}, g_{s}) \|(\nabla_{\varepsilon} \bar{h}^{\varepsilon}, \nabla_{\varepsilon} \bar{g}^{\varepsilon})\|_{L^2}, \nonumber\\
\|h_s \bar{h}_x^{\varepsilon}+g_s \bar{h}_y^{\varepsilon}\|_{L^2(\Omega_{\varepsilon})} \leq &\|(h_s, g_s)\|_{L^{\infty}}\|(\bar{h}_y^{\varepsilon}, \bar{g}_y^{\varepsilon})\|_{L^2}
\leq C(L, u_{s}, v_{s}, h_{s}, g_{s}) \|(\nabla_{\varepsilon} \bar{h}^{\varepsilon}, \nabla_{\varepsilon} \bar{g}^{\varepsilon})\|_{L^2}.\nonumber
\end{align}
Similarly, we obtain the estimates of the remaining terms in $F_{2}$
\begin{align}\label{9.19}
\|v_{s x} \bar{u}^{\varepsilon}+v_{s y} \bar{v}^{\varepsilon}\|_{L^2(\Omega_{\varepsilon})}
& \leq (\|\bar{u}_y^{\varepsilon}\|_{L^2} \sup_x\|\sqrt{y} v_{s x}\|_{L^2}+\|\bar{v}_y^{\varepsilon}\|_{L^2} \sup_x\|\sqrt{y} u_{s x}\|_{L^2}) \nonumber\\
& \leq C(L, u_{s}, v_{s}, h_{s}, g_{s})  \varepsilon^{-\frac{1}{2}-\xi}\|(\nabla_{\varepsilon} \bar{u}^{\varepsilon}, \nabla_{\varepsilon} \bar{v}^{\varepsilon})\|_{L^2}, \nonumber\\
\|u_s \bar{v}_x^{\varepsilon}+v_s \bar{v}_y^{\varepsilon}\|_{L^2(\Omega_{\varepsilon})}
& \leq \|(u_s, v_s)\|_{L^{\infty}}
\|(\bar{v}_x^{\varepsilon}, \bar{v}_y^{\varepsilon})\|_{L^2}\nonumber\\
& \leq C (L, u_{s}, v_{s}, h_{s}, g_{s}) \varepsilon^{-\frac{1}{2}-\xi}\|(\nabla_{\varepsilon} \bar{u}^{\varepsilon}, \nabla_{\varepsilon} \bar{v}^{\varepsilon})\|_{L^2},\\
\|g_{s x} \bar{h}^{\varepsilon}+g_{s y} \bar{g}^{\varepsilon}\|_{L^2(\Omega_{\varepsilon})}
& \leq(\|\bar{h}_y^{\varepsilon}\|_{L^2} \sup_x\|\sqrt{y} g_{s x}\|_{L^2}
+\|\bar{g}_y^{\varepsilon}\|_{L^2} \sup_x\|\sqrt{y} h_{s x}\|_{L^2}) \nonumber\\
& \leq C (L, u_{s}, v_{s}, h_{s}, g_{s})  \varepsilon^{-\frac{1}{2}-\xi}\|(\nabla_{\varepsilon} \bar{h}^{\varepsilon}, \nabla_{\varepsilon} \bar{g}^{\varepsilon})\|_{L^2}, \nonumber\\
\|h_s \bar{g}_x^{\varepsilon}+g_s \bar{g}_y^{\varepsilon}\|_{L^2(\Omega_{\varepsilon})} & \leq \|(h_s, g_s)\|_{L^{\infty}}\|(\bar{g}_x^{\varepsilon},\bar{g}_y^{\varepsilon})\|_{L^2}
\leq C (L, u_{s}, v_{s}, h_{s}, g_{s})  \varepsilon^{-\frac{1}{2}-\xi}\|(\nabla_{\varepsilon} \bar{h}^{\varepsilon}, \nabla_{\varepsilon} \bar{g}^{\varepsilon})\|_{L^2},\nonumber
\end{align}
where we have used
\begin{align*}
\sup _x\left\|\sqrt{y} u_{s x}\right\|_{L^2\left(I_{\varepsilon}\right)} \leq &\sum_{i=0}^{n-1}\sqrt{\varepsilon}^{i} \sup _x\left\|\langle y\rangle u_{p x}^i\right\|_{L^2}
+\sum_{i=1}^{n}\sqrt{\varepsilon}^{i-1}\sup _x\left\|v_{e Y}^i\right\|_{L^2}\\
\leq &\sum_{i=0}^{n-1}\sqrt{\varepsilon}^{i} \sup _x\left\|\langle y\rangle u_{p x}^i\right\|_{L^2}
+\sum_{i=1}^{n}\sqrt{\varepsilon}^{i-1}(\sup _x\left\|v_{e xY}^i\right\|_{L^2}+\sup _x\left\|\partial_{Y}V_{b 0}^i\right\|_{L^2}),\\
\sup _x\left\|\sqrt{y} u_{s y}\right\|_{L^2\left(I_{\varepsilon}\right)} \leq &\left\|u_{e Y}^0\right\|_{L^2}+\sum_{i=0}^{n-1}\sqrt{\varepsilon}^{i}\sup _x\left\|\langle y\rangle u_{p y}^i\right\|_{L^2}+\sum_{i=1}^{n}\sqrt{\varepsilon}^{i-1}(\left\|u_{b Y}^i\right\|_{L^2}+\left\|v_{e Y Y}^i\right\|_{L^2}),\\
\|(u_s,v_s)\|_{L^{\infty}(\Omega_{\varepsilon})}
\leq &\|u_e^0\|_{L^{\infty}}+\sum_{i=0}^{n-1}\sqrt{\varepsilon}^{i}\|(u_p^i, v_p^i)\|_{L^{\infty}}+\sum_{i=1}^{n}\sqrt{\varepsilon}^{i-1}\|(\sqrt{\varepsilon}u_e^i, v_e^i)\|_{L^{\infty}},
\end{align*}
and
\begin{align*}
\sup_x\|\sqrt{y} v_{s x}\|_{L^2(I_{\varepsilon})}
\leq & \sum_{i=0}^{n-1}\sqrt{\varepsilon}^i\sup_x\|\langle y\rangle v_{p x}^i\|_{L^2}
+\sum_{i=1}^{n}\sqrt{\varepsilon}^{i-{2}} \sup_x\|v_{e x}^i(x,\cdot)\|_{L^2}\\
 \leq & \sum_{i=0}^{n-1}\sqrt{\varepsilon}^i\sup_x\|\langle y\rangle v_{p x}^i\|_{L^2}
+\sum_{i=1}^{n}\sqrt{\varepsilon}^{i-{2}} \|v_{e x}^i(x,\cdot)\|_{L^2}^{\frac{1}{2}}\|v_{exx}^i(x,\cdot)\|_{L^2}^{\frac{1}{2}}.
\end{align*}
We apply the same arguments to these terms $\sup_{x}\|\langle y \rangle h_{s x}\|_{L^2(I_{\varepsilon})}, \sup_{x}\|\langle y \rangle h_{s y}\|_{L^2(I_{\varepsilon})}, \sup_{x}\|\langle y \rangle g_{s x}\|_{L^2(I_{\varepsilon})}$ and $\|(h_{s}, g_{s})\|_{L^{\infty} (\Omega_{\varepsilon})}$ of  $(\ref{9.18})$ and $(\ref{9.19})$. This completes the estimates of $F_1 ~\text{and}~ F_2$. Similarly, we follow the previous method to estimate $F_{3}, F_{4}$.
Combining all the above estimates into $(\ref{9.16})$, we obtain
\begin{align}\label{9.17}
& \varepsilon^{\frac{\gamma}{8}} \|\bar{u}^{\varepsilon}\|_{L^{\infty}(\Omega_{\varepsilon})}
+\varepsilon^{\frac{1}{2}+\frac{\gamma}{8}}\|\bar{v}^{\varepsilon}\|_{L^{\infty}
(\Omega_{\varepsilon})}
+\varepsilon^{\frac{\gamma}{8}}\|\bar{h}^{\varepsilon}\|_{L^{\infty}(\Omega_{\varepsilon})}
+\varepsilon^{\frac{1}{2}+\frac{\gamma}{8}}\|\bar{g}^{\varepsilon}\|_{L^{\infty}
(\Omega_{\varepsilon})}  \nonumber\\
\leq & C (u_s, v_s,h_s,g_s, L, \xi) \varepsilon^{\frac{1}{16}-\xi-\frac{15\gamma}{16}}+C (u_s, v_s, h_s,  g_s, L, \xi) \varepsilon^{\frac{n}{2}-\frac{3}{16}-\xi-\frac{\gamma}{16}} K \nonumber\\
&+C (u_s, v_s, h_s, g_s) \varepsilon^{\frac{15 \gamma}{16}+\frac{n-1}{2}} K^2+C(L,u_s, v_s, h_s, g_s) \varepsilon^{\frac{\gamma}{16}-\xi} \|(\bar{u}^{\varepsilon},\bar{v}^{\varepsilon},
\bar{h}^{\varepsilon},\bar{g}^{\varepsilon})\|_{\mathbb{X}},
\end{align}
where $\xi < \frac{\gamma}{16}$.
Combining $(\ref{9.14})$ with $(\ref{9.17})$, taking $\frac{1}{16}-\xi-\gamma \geq 0$ and $\varepsilon \ll 1$, we have
$$
\|(\bar{u}^{\varepsilon}, \bar{v}^{\varepsilon}, \bar{h}^{\varepsilon}, \bar{g}^{\varepsilon})\|_{\mathbb{X}} \leq C (u_s, v_s, h_s, g_s, L, \xi)+C (u_s, v_s,h_s, g_s, L, \xi) \varepsilon^{\frac{2n-1}{4}} K+C (u_s, v_s, h_s, g_s) \varepsilon^{\frac{7 \gamma}{8}+\frac{n-1}{2}} K^2 .
$$
Let $K:=C (u_s, v_s, h_s, g_s,L, \xi)+1$. To obtain $\|(\bar{u}^{\varepsilon}, \bar{v}^{\varepsilon},\bar{h}^{\varepsilon}, \bar{g}^{\varepsilon})\|_\mathbb{X} \leq K$, for any small $\varepsilon$, we must take
$$
C (u_s, v_s, h_s, g_s, L, \xi) \varepsilon^{\frac{2n-1}{4}} K+C (u_s, v_s, h_s, g_s) \varepsilon^{\frac{7 \gamma}{8}+\frac{n-1}{2}} K^2 \leq 1,\quad n\geq 3.
$$
This leads to $(\bar{u}^{\varepsilon}, \bar{v}^{\varepsilon},\bar{h}^{\varepsilon}, \bar{g}^{\varepsilon}) \in \mathbb{X}_{K}$. Therefore, there exists an operator $\mathbb{T}: (u^{\varepsilon}, v^{\varepsilon}, h^{\varepsilon}, g^{\varepsilon}) \mapsto (\bar{u}^{\varepsilon}, \bar{v}^{\varepsilon}, \bar{h}^{\varepsilon}, \bar{g}^{\varepsilon})$ that maps $\mathbb{X}_K$  to itself.
We now prove that the operator $\mathbb{T}$ is a contraction mapping.
For any two pairs $(u_1^{\varepsilon}, v_1^{\varepsilon},h_1^{\varepsilon}, g_1^{\varepsilon})$ and $(u_2^{\varepsilon}, v_2^{\varepsilon},h_2^{\varepsilon}, g_2^{\varepsilon})$ in $\mathbb{X}_K$, it holds that
\begin{align*}
&\|(\bar{u}_1^{\varepsilon}-\bar{u}_2^{\varepsilon}, \bar{v}_1^{\varepsilon}-\bar{v}_2^{\varepsilon},\bar{h}_1^{\varepsilon}
-\bar{h}_2^{\varepsilon}, \bar{g}_1^{\varepsilon}-\bar{g}_2^{\varepsilon})\|_\mathbb{X}\\ \leq & C (u_s, v_s, h_s, g_s,L,\xi)(\varepsilon^{\frac{n}{2}-\frac{3}{16}-\xi-\frac{\gamma}{8}}
+\varepsilon^{\frac{7 \gamma}{8}+\frac{n-1}{2}} K) \|(u_1^{\varepsilon}-u_2^{\varepsilon}, v_1^{\varepsilon}-v_2^{\varepsilon},h_1^{\varepsilon}-h_2^{\varepsilon}, g_1^{\varepsilon}-g_2^{\varepsilon})\|_\mathbb{X}.
\end{align*}
This indicates that for any small $ \varepsilon $, $\mathbb {T} $ is a contraction mapping. We can prove that the system $(\ref {9.2})$ has the unique solution through the standard contraction mapping theorem, which ends the proof.
\end{proof}



\section*{Appendix A. The derivation of any order ideal inner MHD profiles and MHD boundary layer profiles}
This section is devoted to deriving any order ideal MHD profiles 
and boundary layer profiles. 
Let us first consider the zeroth order boundary layer profile.

\subsection*{A.1 The zeroth order MHD boundary layer profile}
We first define
\begin{align}\label{2.1}\tag{{A.1}}
\left\{\begin{aligned}
u_{a p p}(x, y)=& U^{\varepsilon}(x,y)-\varepsilon^{\frac{n}{2}+\gamma}u^{\varepsilon}(x,y), 
\quad v_{a p p}(x, y)= V^{\varepsilon}(x,y)-\varepsilon^{\frac{n}{2}+\gamma}v^{\varepsilon}(x,y),\\
h_{a p p}(x,y)=& H^{\varepsilon}(x,y)-\varepsilon^{\frac{n}{2}+\gamma}h^{\varepsilon}(x,y),
\quad g_{a p p}(x, y)=G^{\varepsilon}(x,y)-\varepsilon^{\frac{n}{2}+\gamma}g^{\varepsilon}(x,y),\\
p_{a p p}(x, y)=&P^{\varepsilon}(x,y)-\varepsilon^{\frac{n}{2}+\gamma}p^{\varepsilon}(x,y).
\end{aligned}\right.
\end{align}
Substituting $({\ref{2.1}})$ into $(\ref{1.5})$ and matching order of $\varepsilon$, we have
\begin{align*}
&R_{a p p}^1+\varepsilon^{\gamma+\frac{n}{2}}\left[(u^{\varepsilon} \partial_x+v^{\varepsilon} \partial_y) u_{a p p}+(u_{a p p} \partial_x+v_{a p p} \partial_y) u^{\varepsilon}-(h^{\varepsilon} \partial_x+g^{\varepsilon} \partial_y) h_{a p p}
+p_x^{\varepsilon}\right.\notag\\
& \quad \left.-\nu_{1}\Delta_{\varepsilon} u^{\varepsilon}-(h_{a p p} \partial_x+g_{a p p} \partial_y) h^{\varepsilon}\right] +\varepsilon^{2 \gamma+n}\left[(u^{\varepsilon} \partial_x+v^{\varepsilon} \partial_y) u^{\varepsilon}-(h^{\varepsilon} \partial_x+g^{\varepsilon} \partial_y) h^{\varepsilon}\right]=0, \notag\\
& R_{a p p}^2+\varepsilon^{\gamma+\frac{n}{2}}\left[(u^{\varepsilon} \partial_x+v^{\varepsilon} \partial_y) v_{a p p}+(u_{a p p} \partial_x+v_{a p p} \partial_y) v^{\varepsilon}-(h^{\varepsilon} \partial_x+g^{\varepsilon} \partial_y) g_{a p p}
+\frac{p_y^{\varepsilon}}{\varepsilon}\right.\notag\\
& \quad \left.- \nu_{2}\Delta_{\varepsilon} v^{\varepsilon}-(h_{a p p} \partial_x+h_{a p p} \partial_y) g^{\varepsilon}\right] +\varepsilon^{2 \gamma+n}\left[(u^{\varepsilon} \partial_x+v^{\varepsilon} \partial_y) v^{\varepsilon}-(h^{\varepsilon} \partial_x+g^{\varepsilon} \partial_y) g^{\varepsilon}\right]=0, \notag\\
& R_{a p p}^3+\varepsilon^{\gamma+\frac{n}{2}}\left[(u^{\varepsilon} \partial_x+v^{\varepsilon} \partial_y) h_{a p p}+(u_{a p p} \partial_x+v_{a p p} \partial_y) h^{\varepsilon}-(h^{\varepsilon} \partial_x+g^{\varepsilon} \partial_y) u_{a p p}- \nu_{3}\Delta_{\varepsilon} h^{\varepsilon}\right.\notag\\
& \quad \left.-(h_{a p p} \partial_x+g_{a p p} \partial_y) u^{\varepsilon}\right] +\varepsilon^{2 \gamma+n}
\left[(u^{\varepsilon} \partial_x+v^{\varepsilon} \partial_y) h^{\varepsilon}-(h^{\varepsilon} \partial_x+g^{\varepsilon} \partial_y) u^{\varepsilon}\right]=0, \notag\\
& R_{a p p}^4+\varepsilon^{\gamma+\frac{n}{2}}\left[(u^{\varepsilon} \partial_x+v^{\varepsilon} \partial_y) g_{a p p}+(u_{a p p} \partial_x+v_{a p p} \partial_y) g^{\varepsilon}-(h^{\varepsilon} \partial_x+g^{\varepsilon} \partial_y) v_{a p p}-\nu_{4}\Delta_{\varepsilon} g^{\varepsilon}\right.\notag\\
& \quad \left.-(h_{a p p} \partial_x+g_{a p p} \partial_y) v^{\varepsilon}\right] +\varepsilon^{2 \gamma+n}
\left[(u^{\varepsilon} \partial_x+v^{\varepsilon} \partial_y) g^{\varepsilon}-(h^{\varepsilon} \partial_x+g^{\varepsilon} \partial_y) v^{\varepsilon}\right]=0,\notag
\end{align*}
and the following terms
\begin{align}\label{2.3}
 R_{a p p}^1 &:=(u_{a p p} \partial_x+v_{a p p} \partial_y) u_{a p p}-(h_{a p p} \partial_x+g_{a p p} \partial_y) h_{a p p}+\partial_x p_{a p p}-\nu_{1}\Delta_{\varepsilon} u_{a p p}, \notag\\
 R_{a p p}^2 &:=(u_{a p p} \partial_x+v_{a p p} \partial_y) v_{a p p}-(h_{a p p} \partial_x+g_{a p p} \partial_y) g_{a p p}+\frac{\partial_y p_{a p p}}{\varepsilon}-\nu_{2}\Delta_{\varepsilon} v_{a p p}, \notag\\
 R_{a p p}^3 &:=(u_{a p p} \partial_x+v_{a p p} \partial_y) h_{a p p}-(h_{a p p} \partial_x+g_{a p p} \partial_y) u_{a p p}-\nu_{3}\Delta_{\varepsilon} h_{a p p}, \notag\\
 R_{a p p}^4 &:=(u_{a p p} \partial_x+v_{a p p} \partial_y) g_{a p p}-(h_{a p p} \partial_x+g_{a p p} \partial_y) v_{a p p}-\nu_{4}\Delta_{\varepsilon} g_{a p p},\tag{A.2}
\end{align}
where $\Delta_{\varepsilon}=\varepsilon \partial_{x}^2+\partial_{y}^2$.
The zeroth order terms of $(\ref{2.3})$ are
\begin{align*}
\begin{aligned}
R^{1}_0= & [(u_e^0+u_p^0) \partial_x+(v_p^0+v_e^1) \partial_y](u_e^0+u_p^0) \\
&-[(h_e^0+h_p^0) \partial_x+(g_p^0+g_e^1) \partial_y](h_e^0+h_p^0)-\nu_{1}\partial_y^2(u_e^0+u_p^0), \\
R^{3}_0= & [(u_e^0+u_p^0)\partial_x+(v_p^0+v_e^1) \partial_y](h_e^0+h_p^0)\\
& -[(h_e^0+h_p^0) \partial_x+(g_p^0+g_e^1) \partial_y](u_e^0+u_p^0)-\nu_{3}\partial_y^2(h_e^0+h_p^0), \\
R^{4}_0= & [(u_e^0+u_p^0)\partial_x+(v_p^0+v_e^1) \partial_y](g_p^0+g_e^1)\\
& -[(h_e^0+h_p^0) \partial_x+(g_p^0+g_e^1) \partial_y](v_p^0+v_e^1)-\nu_{4}\partial_y^2(g_p^0+g_e^1).
\end{aligned}
\end{align*}
We have
$$
\begin{array}{ll}
(v_p^0+v_e^1) \partial_y u_e^0=\sqrt{\varepsilon}(v_p^0+v_e^1) \partial_Y u_e^0, & (g_p^0+g_e^1) \partial_y h_e^0=\sqrt{\varepsilon}(g_p^0+g_e^1) \partial_Y h_e^0, \\
(v_p^0+v_e^1) \partial_y h_e^0=\sqrt{\varepsilon}(v_p^0+v_e^1) \partial_Y h_e^0, & (g_p^0+g_e^1) \partial_y u_e^0=\sqrt{\varepsilon}(g_p^0+g_e^1) \partial_Y u_e^0, \\
(v_p^0+v_e^1) \partial_y g_e^1=\sqrt{\varepsilon}(v_p^0+v_e^1) \partial_Y g_e^1, & (g_p^0+g_e^1) \partial_y v_e^1=\sqrt{\varepsilon}(g_p^0+g_e^1) \partial_Y v_e^1.
\end{array}
$$
Therefore,
$$
\begin{aligned}
& u_e^0 \partial_x u_p^0+v_e^1 \partial_y u_p^0-h_e^0 \partial_x h_p^0-g_e^1 \partial_y h_p^0 \\
= & \,  u_e \partial_x u_p^0+\overline{v_e^1} \partial_y u_p^0+\sqrt{\varepsilon} y(u_{e Y}^0(\sqrt{\varepsilon} y) \partial_x u_p^0
+v_{e Y}^1(x,\sqrt{\varepsilon} y) \partial_y u_p^0) \\
& -h_e \partial_x h_p^0-\overline{g_e^1} \partial_y h_p^0-\sqrt{\varepsilon} y (h_{e Y}^0(\sqrt{\varepsilon} y) \partial_x h_p^0+g_{e Y}^1(x,\sqrt{\varepsilon} y) \partial_y h_p^0)+E_1^0,\\
& u_e^0 \partial_x h_p^0+v_e^1 \partial_y h_p^0-h_e^0 \partial_x u_p^0-g_e^1 \partial_y u_p^0 \\
= &  \,  u_e \partial_x h_p^0+\overline{v_e^1} \partial_y h_p^0+\sqrt{\varepsilon} y (u_{e Y}^0(\sqrt{\varepsilon} y) \partial_x h_p^0+v_{e Y}^1(x, \sqrt{\varepsilon} y) \partial_y h_p^0) \\
& -h_e \partial_x u_p^0-\overline{g_e^1} \partial_y u_p^0-\sqrt{\varepsilon} y (h_{e Y}^0(\sqrt{\varepsilon} y) \partial_x u_p^0+g_{e Y}^1(x, \sqrt{\varepsilon} y) \partial_y u_p^0)+E_3^0,
\end{aligned}
$$
and
$$
\begin{aligned}
& u_e^0 \partial_x g_p^0+v_e^1 \partial_y g_p^0+u_p^0 \partial_x g_e^1-h_e^0 \partial_x v_p^0-g_e^1 \partial_y v_p^0-h_p^0 \partial_x v_e^1\\
= &  \,  u_e \partial_x g_p^0+\overline{v_e^1} \partial_y g_p^0+\sqrt{\varepsilon} y(u_{e Y}^0(\sqrt{\varepsilon} y) \partial_x g_p^0
+v_{e Y}^1(x, \sqrt{\varepsilon} y) \partial_y g_p^0) \\
& -h_e \partial_x v_p^0-\overline{g_e^1} \partial_y v_p^0-\sqrt{\varepsilon} y(h_{e Y}^0(\sqrt{\varepsilon} y) \partial_x v_p^0
+g_{e Y}^1(x, \sqrt{\varepsilon} y) \partial_y v_p^0) \\
& +u_p^0 \overline{\partial_x g_e^1}-h_p^0 \overline{\partial_x v_e^1}+E_4^0,
\end{aligned}
$$
where $u_e=u_e^0(0), h_e=h_e^{0}(0)$ and
\begin{align*}
E_1^0= & \varepsilon \partial_x u_p^0 \int_0^y \int_y^{y_{1}} \partial_Y^2 u_e^0(\sqrt{\varepsilon} \tau) \mathrm{d} \tau \mathrm{d} {y_{1}}+\varepsilon \partial_y u_p^0 \int_0^y \int_y^{y_{1}} \partial_Y^2 v_e^1(x, \sqrt{\varepsilon} \tau) \mathrm{d} \tau \mathrm{d} {y_{1}} \notag\\
& -\varepsilon \partial_x h_p^0 \int_0^y \int_y^{y_{1}} \partial_Y^2 h_e^0(\sqrt{\varepsilon} \tau) \mathrm{d} \tau \mathrm{d} {y_{1}}-\varepsilon \partial_y h_p^0 \int_0^y \int_y^{y_{1}} \partial_Y^2 g_e^1(x, \sqrt{\varepsilon} \tau) \mathrm{d} \tau \mathrm{d} {y_{1}},\notag\\
E_3^0= & \varepsilon \partial_x h_p^0 \int_0^y \int_y^{y_{1}} \partial_Y^2 u_e^0(\sqrt{\varepsilon} \tau) \mathrm{d} \tau \mathrm{d} {y_{1}}+\varepsilon \partial_y h_p^0 \int_0^y \int_y^{y_{1}} \partial_Y^2 v_e^1(x, \sqrt{\varepsilon} \tau) \mathrm{d} \tau \mathrm{d} {y_{1}} \notag\\
& -\varepsilon \partial_x u_p^0 \int_0^y \int_y^{y_{1}} \partial_Y^2 h_e^0(\sqrt{\varepsilon} \tau) \mathrm{d} \tau \mathrm{d} {y_{1}}-\varepsilon \partial_y u_p^0 \int_0^y \int_y^{y_{1}} \partial_Y^2 g_e^1(x, \sqrt{\varepsilon} \tau) \mathrm{d} \tau \mathrm{d} {y_{1}}, \\
E_4^0= & \varepsilon \partial_x g_p^0 \int_0^y \int_y^{y_{1}} \partial_Y^2 u_e^0(\sqrt{\varepsilon} \tau) \mathrm{d} \tau \mathrm{d} {y_{1}}+\varepsilon \partial_y g_p^0 \int_0^y \int_y^{y_{1}} \partial_Y^2 v_e^1(x, \sqrt{\varepsilon} \tau) \mathrm{d} \tau \mathrm{d} {y_{1}} \notag\\
& -\varepsilon \partial_x v_p^0 \int_0^y \int_y^{y_{1}} \partial_Y^2 h_e^0(\sqrt{\varepsilon} \tau) \mathrm{d} \tau \mathrm{d} {y_{1}}-\varepsilon \partial_y v_p^0 \int_0^y \int_y^{y_{1}} \partial_Y^2 g_e^1(x, \sqrt{\varepsilon} \tau) \mathrm{d} \tau \mathrm{d} {y_{1}} \notag\\
& +\sqrt{\varepsilon} u_p^0 \int_0^y \partial_{x Y} g_e^1(x, \sqrt{\varepsilon} \tau) \mathrm{d} \tau-\sqrt{\varepsilon} h_p^0 \int_0^y \partial_{x Y} v_e^1(x, \sqrt{\varepsilon} \tau) \mathrm{d} \tau.\notag
\end{align*}
This leads to
\begin{align}\label{2.6}\tag{A.3}
\left\{\begin{array}{l}
(u_e+u_p^0) \partial_x u_p^0+(v_p^0+\overline{v_e^1}) \partial_y u_p^0-(h_e+h_p^0) \partial_x h_p^0-(g_p^0+\overline{g_e^1}) \partial_y h_p^0= \nu_{1}\partial_y^2 u_p^0, \\
(u_e+u_p^0) \partial_x h_p^0+(v_p^0+\overline{v_e^1}) \partial_y h_p^0-(h_e+h_p^0) \partial_x u_p^0-(g_p^0+\overline{g_e^1}) \partial_y u_p^0= \nu_{3}\partial_y^2 h_p^0, \\
(u_e+u_p^0) \partial_x(g_p^0+\overline{g_e^1})
+(v_p^0+\overline{v_e^1})\partial_y(g_p^0+\overline{g_e^1}) \\
\quad-(h_e+h_p^0)\partial_x(v_p^0+\overline{v_e^1})
-(g_p^0+\overline{g_e^1})\partial_y(v_p^0+\overline{v_e^1})=\nu_{4}\partial_y^2 g_p^0, \\
(v_p^0, g_p^0)(x, y)=\int_y^{\varepsilon^{-1/2}} \partial_x (u_p^0, h_p^0)(x, z) \mathrm{d} z, \\
(v_e^1, g_e^1)(x, 0)=-\int_0^{\varepsilon^{-1/2}} \partial_x(u_p^0, h_p^0)(x, z) \mathrm{d} z, \\
(u_p^0, \partial_y h_p^0)(x, 0)=(u_b-u_e, 0), \\
(v_p^0, g_p^0)(x, 0)=-(v_e^1, g_e^1)(x, 0),
(u_p^0, h_p^0)(0, y)=(\tilde{u}_0(y), \tilde{h}_0(y)).
\end{array}\right.
\end{align}
We note that the third equality can be deduced from the second equality and the boundary conditions.
Then, the zeroth order terms $R^1_0, R^3_0 ~\text{and}~ R^4_0$ equal
\begin{align*}
\left\{\begin{aligned}
R^{1}_{0}= & \sqrt{\varepsilon}(v_p^0+v_e^1) \partial_Y u_e^0+\sqrt{\varepsilon} y (\partial_Y u_e^0(\sqrt{\varepsilon} y) \partial_x u_p^0+\partial_Y v_e^1 \partial_y u_p^0)-\nu_{1}\varepsilon \partial_Y^2 u_e^0 \\
& -\sqrt{\varepsilon}(g_p^0+g_e^1) \partial_Y h_e^0-\sqrt{\varepsilon} y(\partial_Y h_e^0(\sqrt{\varepsilon} y) \partial_x h_p^0+\partial_Y g_e^1 \partial_y h_p^0)+E_1^0, \\
R^{3}_{0}= & \sqrt{\varepsilon}(v_p^0+v_e^1) \partial_Y h_e^0+\sqrt{\varepsilon} y(\partial_Y u_e^0(\sqrt{\varepsilon} y) \partial_x h_p^0+\partial_Y v_e^1 \partial_y h_p^0)- \nu_{3}\varepsilon \partial_Y^2 h_e^0 \\
& -\sqrt{\varepsilon}(g_p^0+g_e^1) \partial_Y u_e^0-\sqrt{\varepsilon} y(\partial_Y h_e^0(\sqrt{\varepsilon} y) \partial_x u_p^0+\partial_Y g_e^1 \partial_y u_p^0)+E_3^0, \\
R^{4}_{0}= &\sqrt{\varepsilon} y(\partial_Y u_e^0(\sqrt{\varepsilon} y) \partial_x g_p^0+\partial_Y v_e^1 \partial_y g_p^0)
-\sqrt{\varepsilon} y(\partial_Y h_e^0(\sqrt{\varepsilon} y) \partial_x v_p^0+\partial_Y g_e^1 \partial_y v_p^0) + E_4^0.
\end{aligned}\right.
\end{align*}

\subsection*{A.2 The $\sqrt{\varepsilon}$-order ideal inner MHD profile and MHD boundary layer  profile}


In this section, we consider the $\sqrt{\varepsilon}$-order ideal inner MHD profile $(u_e^1, v_e^1, h_e^1, g_e^1, p_e^1)$ and MHD boundary layer profile $(u_p^1, v_p^1, h_p^1, g_p^1, p_p^1)$.


\subsubsection*{A.2.1 The first order ideal inner MHD profile}


Matching the various terms of order  $\sqrt{\varepsilon}$ in $(\ref{2.3})$, $R^{1,0}_{p} , R^{3,0}_{p}~\text{and}~ R^{4,0}_{p}$, we have
\begin{align}\label{3.1}
R^{u}_1= & [(u_e^{1}+u_p^1) \partial_x+(v_p^1+v_e^2) \partial_y](u_e^{0}+u_p^{0})+[(u_e^{0}+u_p^{0}) \partial_x+(v_p^{0}+v_e^1) \partial_y](u_e^{1}+u_p^1) \notag\\
& +\partial_x (p_e^1+p_p^1)- \nu_{1}\partial_y^2(u_e^1+u_p^1)+(y \partial_x u_p^0+v_p^0+v_e^1) \partial_Y u_e^0+y \partial_Y v_e^1 \partial_y u_p^0 \notag\\
& -[(h_e^1+h_p^1) \partial_x+(g_p^1+g_e^2) \partial_y](h_e^0+h_p^0)-[(h_e^0+h_p^0) \partial_x+(g_p^0+g_e^1) \partial_y](h_e^1+h_p^1) \notag\\
& -(y \partial_x h_p^0+g_p^0+g_e^1) \partial_Y h_e^0-y \partial_Y g_e^1 \partial_y h_p^0 + \triangle_0^1, \notag\\
R^{h}_1= & [(u_e^1+u_p^1) \partial_x+(v_p^1+v_e^2) \partial_y](h_e^0+h_p^0)+[(u_e^0+u_p^0) \partial_x+(v_p^0+v_e^1) \partial_y](h_e^1+h_p^1)  \notag\\
& -\nu_{3}\partial_y^2(h_e^1+h_p^1)+(v_p^0+v_e^1) \partial_Y h_e^0+y \partial_Y u_e^0 \partial_x h_p^0+y \partial_Y v_e^1 \partial_y h_p^0 \tag{A.4}\\
& -[(h_e^1+h_p^1) \partial_x+(g_p^1+g_e^2) \partial_y](u_e^0+u_p^0)-[(h_e^0+h_p^0) \partial_x+(g_p^0+g_e^1) \partial_y](u_e^1+u_p^1) \notag\\
& -(g_p^0+g_e^1) \partial_Y u_e^0-y \partial_Y h_e^0 \partial_x u_p^0-y \partial_Y g_e^1 \partial_y u_p^0 + \triangle_0^3,\notag\\
R^{g}_1= & [(u_e^1+u_p^1) \partial_x+(v_p^1+v_e^2) \partial_y](g_e^1+g_p^0)+[(u_e^0+u_p^0) \partial_x+(v_p^0+v_e^1) \partial_y](g_e^2+g_p^1)  \notag\\
& -\nu_{4}\partial_y^2(g_e^2+g_p^1)+(v_p^0+v_e^1) \partial_Y g_e^1+y \partial_Y u_e^0 \partial_x g_p^0+y \partial_Y v_e^1 \partial_y g_p^0 \notag\\
& -[(h_e^1+h_p^1) \partial_x+(g_p^1+g_e^2) \partial_y](v_e^1+v_p^0)-[(h_e^0+h_p^0) \partial_x+(g_p^0+g_e^1) \partial_y](v_e^2+v_p^1) \notag\\
& -(g_p^0+g_e^1) \partial_Y v_e^1-y \partial_Y h_e^0 \partial_x v_p^0-y \partial_Y g_e^1 \partial_y v_p^0.\notag
\end{align}
Note that the inner MHD correctors are always evaluated at $(x, \sqrt{\varepsilon} y)$. Thus, we have
\begin{align*}
 (v_p^1+v_e^2) \partial_y u_e^0=\sqrt{\varepsilon}(v_p^1+v_e^2) \partial_Y u_e^0,
 \quad 
 (v_p^0+v_e^1) \partial_y u_e^1=\sqrt{\varepsilon}(v_p^0+v_e^1) \partial_Y u_e^1.
\end{align*}
The remaining high order terms of $(\ref{3.1})$ have the same form.  By matching the terms of order  $\sqrt{\varepsilon}$, we obtain that the first order ideal MHD profile $(u_e^1,v_e^1, h_e^1, g_e^1, p_e^1)$ satisfies
\begin{align}\label{3.2}\tag{A.5}
\left\{\begin{array}{l}
u_e^0 \partial_x u_e^1+v_e^1 \partial_Y u_e^0-h_e^0 \partial_x h_e^1-g_e^1 \partial_Y h_e^0+\partial_x p_e^1=0, \\
u_e^0 \partial_x h_e^1+v_e^1 \partial_Y h_e^0-h_e^0 \partial_x u_e^1-g_e^1 \partial_Y u_e^0=0,
\end{array}\right.
\end{align}
and the first order MHD boundary layer profile $(u_p^1,v_p^1,h_p^1,g_p^1,p_p^1)$ satisfies
\begin{align}\label{3.3}\tag{A.6}
\begin{aligned}
&(u_e^1+u_p^1) \partial_x u_p^0+u_p^0 \partial_x u_e^1+(u_e^0+u_p^0) \partial_x u_p^1+(v_p^1+v_e^2)\partial_y u_p^0 \notag\\
&\quad+(v_p^0+v_e^1) \partial_y u_p^1+\partial_x p_p^1- \nu_{1}\partial_y^2 u_p^1+(y \partial_x u_p^0+v_p^0) \partial_Y u_e^0+y \partial_Y v_e^1 \partial_y u_p^0 \notag\\
&\quad-(h_e^1+h_p^1) \partial_x h_p^0-h_p^0 \partial_x h_e^1-(h_e^0+h_p^0) \partial_x h_p^1-(g_p^1+g_e^2)\partial_y h_p^0 \notag\\
&\quad -(g_p^0+g_e^1) \partial_y h_p^1-(y \partial_x h_p^0+g_p^0) \partial_Y h_e^0-y \partial_Y g_e^1 \partial_y h_p^0 + \triangle_0^1=0 ,\notag\\
&(u_e^1+u_p^1) \partial_x h_p^0+u_p^0 \partial_x h_e^1+(u_e^0+u_p^0) \partial_x h_p^1+(v_p^1+v_e^2)\partial_y h_p^0\notag\\
&\quad+(v_p^0+v_e^1) \partial_y h_p^1-\nu_{3}\partial_y^2 h_p^1+v_p^0 \partial_Y h_e^0+y \partial_Y u_e^0 \partial_x h_p^0+y \partial_Y v_e^1 \partial_y h_p^0 \notag\\
&\quad -(h_e^1+h_p^1) \partial_x u_p^0-h_p^0 \partial_x u_e^1-(h_e^0+h_p^0) \partial_x u_p^1-(g_p^1+g_e^2)\partial_y u_p^0 \\
&\quad -(g_p^0+g_e^1) \partial_y u_p^1-g_p^0 \partial_Y u_e^0-y \partial_Y h_e^0 \partial_x u_p^0-y \partial_Y g_e^1 \partial_y u_p^0+\triangle_0^3=0,\notag\\
&(u_e^1+u_p^1) \partial_x g_p^0+u_p^0 \partial_x g_e^2+(u_e^0+u_p^0) \partial_x g_p^1+(v_p^1+v_e^2)\partial_y g_p^0\notag\\
&\quad+(v_p^0+v_e^1) \partial_y g_p^1-\nu_{4}\partial_y^2 g_p^1+v_p^0 \partial_Y g_e^1+y \partial_Y u_e^0 \partial_x g_p^0+y \partial_Y v_e^1 \partial_y g_p^0 \notag\\
&\quad -(h_e^1+h_p^1) \partial_x v_p^0-h_p^0 \partial_x v_e^2-(h_e^0+h_p^0) \partial_x v_p^1-(g_p^1+g_e^2)\partial_y v_p^0 \notag\\
&\quad -(g_p^0+g_e^1) \partial_y v_p^1-g_p^0 \partial_Y v_e^1-y \partial_Y h_e^0 \partial_x v_p^0-y \partial_Y g_e^1 \partial_y v_p^0=0. \notag
\end{aligned}
\end{align}
Hence, $R^{u}_1, R^{h}_1 ~\text{and}~ R^{g}_1$ may be rewritten
\begin{align*}
\left\{\begin{aligned}
 R^{u}_1=&\sqrt{\varepsilon}[(v_p^1+v_e^2) \partial_Y u_e^0+(v_p^0+v_e^1) \partial_Y u_e^1-(g_p^1+g_e^2) \partial_Y h_e^0-(g_p^0+g_e^1) \partial_Y h_e^1]-\nu_{1}\varepsilon \partial_Y^2 u_e^1,\\
 R^{h}_1=&\sqrt{\varepsilon}[(v_p^1+v_e^2) \partial_Y h_e^0+(v_p^0+v_e^1) \partial_Y h_e^1-(g_p^1+g_e^2) \partial_Y u_e^0-(g_p^0+g_e^1) \partial_Y u_e^1]-\nu_{3}\varepsilon \partial_Y^2 h_e^1,\\
 R^{g}_1=&u_e^1 \partial_{x}g_e^1+u_e^0 \partial_{x}g_e^2+v_e^1 \partial_{Y}g_e^1-h_e^1 \partial_{x}v_e^1-h_e^0 \partial_{x}v_e^2-g_e^1 \partial_{Y}v_e^1-\nu_{4}\varepsilon \partial_Y^2 g_e^2\\
  &+\sqrt{\varepsilon}[(v_p^1+v_e^2) \partial_Y g_e^1+(v_p^0+v_e^1) \partial_Y g_e^2-(g_p^1+g_e^2) \partial_Y v_e^1-(g_p^0+g_e^1) \partial_Y v_e^2].
  \end{aligned}\right.
\end{align*}
Next, we consider the zeroth order terms in $R^{2}_{app} ~\text{and}~  R^{4}_{app}$
\begin{align*}
R^{v}_0= & [(u_e^0+u_p^0) \partial_x+(v_p^0+v_e^1) \partial_y](v_p^0+v_e^1)+\partial_Y p_e^1+\frac{p_{p y}^1}{\sqrt{\varepsilon}}+\partial_y p_p^2  \notag\\
& -\nu_{2}\partial_y^2 (v_p^0+v_e^1)
-[(h_e^0+h_p^0) \partial_x+(g_p^0+g_e^1) \partial_y](g_p^0+g_e^1), \notag\\
R^{g}_0= & [(u_e^0+u_p^0) \partial_x+(v_p^0+v_e^1)\partial_y](g_p^0+g_e^1)
- \nu_{4}\partial_y^2(g_p^0+g_e^1)  \notag\\
& -[(h_e^0+h_p^0) \partial_x+(g_p^0+g_e^1) \partial_y](v_p^0+v_e^1).
\end{align*}
The zeroth order term in $R^{v}_0$ is $p^1_{py}$. We take $p_{py}^1 =0$, and obtain that $p_p^1$ only depends on the variable $x$, that is
$$p_p^1=p_p^1(x).$$
Next, we match the terms of order $\varepsilon^{\frac{1}{2}}$ in  $R^{v}_0=R^{g}_0=0$. Recall that in Appendix A.1, the terms of $R^g_0$  related to the zeroth order boundary layer profile have been constructed.  Thus,
$$
R^{g}_0=u_e^0 \partial_{x}g_e^1-h_e^0 \partial_{x}v_e^1+\sqrt{\varepsilon}[(v_p^0+v_e^1) \partial_Y g_e^1-(g_p^0+g_e^1) \partial_Y v_e^1]- \nu_{4}\varepsilon \partial_Y^2 g_e^1.
$$
We obtain that $(u_e^1, v_e^1, h_e^1, g_e^1, p_e^1)$ satisfies the following system
\begin{align}\label{3.6}\tag{A.7}
\left\{\begin{array}{l}
u_e^0 \partial_x v_e^1-h_e^0 \partial_x g_e^1+\partial_Y p_e^1=0, \\
u_e^0 \partial_x g_e^1-h_e^0 \partial_x v_e^1=0,
\end{array}\right.
\end{align}
and $p_p^2$ satisfies
\begin{align*}
p_p^2=&\int_{y}^{1/\sqrt{\varepsilon}}\left[(u_e^0+u_p^0) \partial_x v_p^0+u_p^0 \partial_x v_e^1 +(v_p^0+v_e^1) \partial_y v_p^0-(h_e^0+h_p^0) \partial_x g_p^0\right.\notag\\
&\left.-h_p^0 \partial_x g_e^1
-(g_p^0+g_e^1) \partial_y g_p^0-\nu_{2} \partial_y^2 v_p^0\right](x, \theta)\mathrm{d}\theta. 
\end{align*}
Then, 
 $R^{v}_0$ and $R^{g}_0$ may be rewritten as follows
\begin{align}\label{3.8}\tag{A.8}
\left\{\begin{array}{l}
R^{v}_0=\sqrt{\varepsilon}[(v_p^0+v_e^1) \partial_Y v_e^1-(g_p^0+g_e^1) \partial_Y g_e^1]- \nu_{2}\varepsilon \partial_Y^2 v_e^1, \\
R^{g}_0=\sqrt{\varepsilon}[(v_p^0+v_e^1) \partial_Y g_e^1-(g_p^0+g_e^1) \partial_Y v_e^1]- \nu_{4}\varepsilon \partial_Y^2 g_e^1.
\end{array}\right.
\end{align}

\subsubsection*{A.2.2 The first order MHD boundary layer profile}


In this subsection, we derive 
 $(u_p^1, v_p^1, h_p^1, g_p^1, p_p^1)$, which solves $(\ref{3.3})$. For simplicity, we introduce
$$
u^0:=u_e+u_p^0(x, y), \quad h^0:=h_e+h_p^0(x, y).
$$
Let $p_{p x}^{1}=0$, then we rewrite $(\ref{3.3})$
\begin{align}\label{4.1}
&u^0 \partial_x u_p^{1}+u_p^{1} \partial_x u^0+v_p^{1} \partial_y u^0+(v_p^0+\overline{v_e^1}) \partial_y u_p^{1}-\nu_{1}\partial_y^2 u_p^{1} \notag\\
& -h^0 \partial_x h_p^{1}-h_p^{1} \partial_x h^0-g_p^{1} \partial_y h^0-(g_p^0+\overline{g_e^1}) \partial_y h_p^{1} \notag\\
= & -\overline{\partial_Y u_e^0}[y \partial_x u_p^0+v_p^0]-y \overline{\partial_Y v_e^1} \partial_y u_p^0-\overline{u_e^1} \partial_x u_p^0-u_p^0 \overline{\partial_x u_e^1}-\overline{v_e^2} \partial_{y}u_p^0  \notag\\
& +\overline{\partial_Y h_e^0}[y \partial_x h_p^0+g_p^0]+y \overline{\partial_Y g_e^1} \partial_y h_p^0+\overline{h_e^1} \partial_x h_p^0+h_p^0 \overline{\partial_x h_e^1}+\overline{g_e^2} \partial_{y}h_p^0:= F_{p_{1}}^1, \notag\\
&u^0 \partial_x h_p^{1}+u_p^{1} \partial_x h^0+v_p^{1} \partial_y h^0+\left(v_p^0+\overline{v_e^1}\right) \partial_y h_p^{1}- \nu_{3}\partial_y^2 h_p^{1} \notag\\
& -h^0 \partial_x u_p^{1}-h_p^{1} \partial_x u^0-g_p^{1} \partial_y u^0-\left(g_p^0+\overline{g_e^1}\right) \partial_y u_p^{1} \notag\\
= & -\overline{\partial_Y h_e^0} v_p^0-y \overline{\partial_Y u_e^0} \partial_x h_p^0-y \overline{\partial_Y v_e^1} \partial_y h_p^0-\overline{\partial_x h_e^1} u_p^0-\overline{u_e^1} \partial_x h_p^0-\overline{v_e^2} \partial_{y}h_p^0 \tag{A.9}\\
& +\overline{\partial_Y u_e^0} g_p^0+y \overline{\partial_Y h_e^0} \partial_x u_p^0+y \overline{\partial_Y g_e^1} \partial_y u_p^0+\overline{h_e^1} \partial_x u_p^0+\overline{\partial_x u_e^1} h_p^0 +\overline{g_e^2} \partial_{y}u_p^0:=F_{p_2}^1 ,\notag\\
&u^0 \partial_x g_p^{1}+u_p^{1} \partial_x g_p^0+v_p^{1} \partial_y g_p^0+\left(v_p^0+\overline{v_e^1}\right) \partial_y g_p^{1}- \nu_{4}\partial_y^2 g_p^{1} \notag\\
& -h^0 \partial_x v_p^{1}-h_p^{1} \partial_x v_p^0-g_p^{1} \partial_y v_p^0-\left(g_p^0+\overline{g_e^1}\right) \partial_y v_p^{1} \notag\\
= & -\overline{\partial_Y g_e^1} v_p^0-y \overline{\partial_Y u_e^0} \partial_x g_p^0-y \overline{\partial_Y v_e^1} \partial_y g_p^0-\overline{\partial_x g_e^2} u_p^0-\overline{u_e^1} \partial_x g_p^0-\overline{v_e^2} \partial_{y}g_p^0 \notag\\
& +\overline{\partial_Y v_e^1} g_p^0+y \overline{\partial_Y h_e^0} \partial_x v_p^0+y \overline{\partial_Y g_e^1} \partial_y v_p^0+\overline{h_e^1} \partial_x v_p^0+\overline{\partial_x v_e^2} h_p^0 +\overline{g_e^2} \partial_{y}v_p^0:=F_{p_3}^1 .\notag
\end{align}
It is noted that the third equality in $(\ref{4.1})$ can be deduced from the second equality and the boundary conditions. 
Then, the new error terms are given by
\begin{align*}
E_{r_1}^1:= & (u_e^0-u_e) \partial_x u_p^{1}+(v_e^{2}-\overline{v_e^2}) \partial_y u_p^0+(v_e^1-\overline{v_e^1}) \partial_y u_p^{1}-(h_e^0-h_e) \partial_x h_p^{1} \notag\\
& -(g_e^{2}-\overline{g_e^2}) \partial_y h_p^0-(g_e^1-\overline{g_e^1}) \partial_y h_p^{1}+(\partial_Y u_e^0-\overline{\partial_Y u_e^0})(y \partial_x u_p^0+v_p^0) \notag\\
& +y(\partial_Y v_e^1-\overline{\partial_Y v_e^1}) \partial_y u_p^0+(u_e^1-\overline{u_e^1}) \partial_x u_p^0+(\partial_x u_e^1-\overline{\partial_x u_e^1}) u_p^0 \notag\\
& -(\partial_Y h_e^0-\overline{\partial_Y h_e^0})(y \partial_x h_p^0+g_p^0)-y(\partial_Y g_e^1-\overline{\partial_Y g_e^1}) \partial_y h_p^0 \notag\\
& -(h_e^1-\overline{h_e^1}) \partial_x u_p^0-(\partial_x h_e^1-\overline{\partial_x h_e^1}) h_p^0+\triangle_0^1, \notag\\
E_{r_2}^1:= & (u_e^0-u_e) \partial_x h_p^{1}+(v_e^{2}-\overline{v_e^2}) \partial_y h_p^0+(v_e^1-\overline{v_e^1}) \partial_y h_p^{1} \notag\\
& -(h_e^0-h_e) \partial_x u_p^{1}-(g_e^{2}-\overline{g_e^2}) \partial_y u_p^0-(g_e^1-\overline{g_e^1}) \partial_y u_p^{1} \nonumber\\
& +(\partial_Y h_e^0-\overline{\partial_Y h_e^0}) v_p^0+y \partial_x h_p^0(\partial_Y u_e^0-\overline{\partial_Y u_e^0})\notag\\
& +y \partial_y h_p^0(\partial_Y v_e^1-\overline{\partial_Y v_e^1})+(\partial_x h_e^1-\overline{\partial_x h_e^1}) u_p^0+(u_e^1-\overline{u_e^1}) \partial_x h_p^0 \\
& -(\partial_Y u_e^0-\overline{\partial_Y u_e^0}) g_p^0-y(\partial_Y h_e^0-\overline{\partial_Y h_e^0}) \partial_x u_p^0-(h_e^1-\overline{h_e^1}) \partial_x u_p^0 \notag\\
& -y(\partial_Y g_e^1-\overline{\partial_Y g_e^1}) \partial_y u_p^0-(\partial_x u_e^1-\overline{\partial_x u_e^1}) h_p^0+\triangle_0^3,\notag\\
E_{r_3}^1:= & (u_e^0-u_e) \partial_x g_p^{1}+(v_e^{2}-\overline{v_e^2}) \partial_y g_p^0+(v_e^1-\overline{v_e^1}) \partial_y g_p^{1} \notag\\
& -(h_e^0-h_e) \partial_x v_p^{1}-(g_e^{2}-\overline{g_e^2}) \partial_y v_p^0-(g_e^1-\overline{g_e^1}) \partial_y v_p^{1} \notag\\
& +(\partial_Y g_e^1-\overline{\partial_Y g_e^1}) v_p^0+y \partial_x g_p^0(\partial_Y u_e^0-\overline{\partial_Y u_e^0})\notag\\
& +y \partial_y g_p^0(\partial_Y v_e^1-\overline{\partial_Y v_e^1})+(\partial_x g_e^2-\overline{\partial_x g_e^2}) u_p^0+(u_e^1-\overline{u_e^1}) \partial_x g_p^0 \notag\\
& -(\partial_Y v_e^1-\overline{\partial_Y v_e^1}) g_p^0-y(\partial_Y h_e^0-\overline{\partial_Y h_e^0}) \partial_x v_p^0-(v_e^2-\overline{v_e^2}) \partial_x h_p^0 \notag\\
& -y(\partial_Y g_e^1-\overline{\partial_Y g_e^1}) \partial_y v_p^0-(\partial_x v_e^2-\overline{\partial_x v_e^2}) h_p^0.\notag
\end{align*}

\subsection*{A.3 The $\varepsilon^{\frac{i}{2}}$-order ideal inner MHD profiles and boundary layer profiles}

In this subsection, we study the $i$-th order ideal MHD profiles  $(u_e^i, v_e^i, h_e^i, g_e^i, p_e^i)$ and boundary layer profiles $(u_p^i, v_p^i, h_p^i, g_p^i, p_p^i)(2\leq i \leq n-1, n\geq3)$.

\subsubsection*{A.3.1 The $i$-th order ideal MHD profiles}

Let us consider the $i$-th order ideal MHD profiles and  define $R^u_i, R^h_i ~\text{and}~ R^g_i$ as follows
\begin{align}\label{5.1}
R^{u}_i= & \sum_{i_1+i_2=i} [(u_e^{i_1}+u_p^{i_1}) \partial_x+(v_p^{i_1}+v_e^{i_1+1}) \partial_y](u_e^{i_2}+u_p^{i_2})+\partial_x(p_e^i+p_p^i)
- \nu_{1}\partial_y^2(u_e^{i}+u_p^{i})\notag\\
& -\sum_{i_1+i_2=i} [(h_e^{i_1}+h_p^{i_1}) \partial_x+(g_p^{i_1}+g_e^{i_1+1}) \partial_y](h_e^{i_2}+h_p^{i_2}) -\nu_{1}\partial_{x}^{2}u_p^{i-2}\notag\\
&+\sum_{i_1+i_2=i-1}[(v_p^{i_1}+v_e^{i_1+1}) \partial_Y u_e^{i_2}-(g_p^{i_1}+g_e^{i_1+1}) \partial_Y h_e^{i_2}] +(4-n)\frac{E_{1}^0}{\varepsilon}\notag\\
&+\triangle_{i-2}^{2}+\varepsilon^{-\frac{1}{2}}E_{r_1}^{i-1}
+ \triangle_{i-1}^{1}-\nu_{1}\Delta u_e^{i-2},\notag\\
R^{h}_i= & \sum_{i_1+i_2=i} [(u_e^{i_1}+u_p^{i_1}) \partial_x+(v_p^{i_1}+v_e^{i_1+1}) \partial_y](h_e^{i_2}+h_p^{i_2})
- \nu_{3}\partial_y^2(h_e^{i}+h_p^{i})\notag\\
& -\sum_{i_1+i_2=i} [(h_e^{i_1}+h_p^{i_1}) \partial_x+(g_p^{i_1}+g_e^{i_1+1}) \partial_y](u_e^{i_2}+u_p^{i_2}) -\nu_{3}\partial_{x}^{2}h_p^{i-2}\tag{A.10}\\
&+\sum_{i_1+i_2=i-1}[(v_p^{i_1}+v_e^{i_1+1}) \partial_Y h_e^{i_2}-(g_p^{i_1}+g_e^{i_1+1}) \partial_Y u_e^{i_2}] +(4-n)\frac{E_{3}^0}{\varepsilon} \notag\\
&+\triangle_{i-2}^{4}+\varepsilon^{-\frac{1}{2}}E_{r_2}^{i-1}
+ \triangle_{i-1}^{3}-\nu_{3}\Delta h_e^{i-2},\notag\\
R^{g}_i= & \sum_{i_1+i_2=i} [(u_e^{i_1}+u_p^{i_1}) \partial_x+(v_p^{i_1}+v_e^{i_1+1}) \partial_y](g_e^{i_2}+g_p^{i_2})
- \nu_{4}\partial_y^2(g_e^{i+1}+g_p^{i})\notag\\
& -\sum_{i_1+i_2=i} [(h_e^{i_1}+h_p^{i_1}) \partial_x+(g_p^{i_1}+g_e^{i_1+1}) \partial_y](v_e^{i_2+1}+v_p^{i_2}) -\nu_{4}\partial_{x}^{2}g_p^{i-2}\notag\\
&+\sum_{i_1+i_2=i-1}[(v_p^{i_1}+v_e^{i_1+1}) \partial_Y g_e^{i_2+1}-(g_p^{i_1}+g_e^{i_1+1}) \partial_Y v_e^{i_2+1}] +(4-n)\frac{E_{4}^0}{\varepsilon}\notag\\
&+\varepsilon^{-\frac{1}{2}}E_{r_3}^{i-1}
-\nu_{4}\Delta g_e^{i-1},\notag
\end{align}
 if $4-n < 0$, the terms $(4-n)\frac{E_1^0}{\varepsilon}, (4-n)\frac{E_3^0}{\varepsilon} ~\text{and}~ (4-n)\frac{E_4 ^0}{\varepsilon}$ vanish. Since the ideal MHD profiles are always evaluated at $(x, \sqrt{\varepsilon} y)$, we have
\begin{align*}
 [v_p^{i_1}+v_e^{i_1+1}] \partial_y u_e^{i_2}=\sqrt{\varepsilon}[v_p^{i_1}+v_e^{i_1+1}] \partial_Y u_e^{i_2}, \quad [g_p^{i_1}+g_e^{i_1+1}]\partial_y h_e^{i_2}=\sqrt{\varepsilon}[g_p^{i_1}+g_e^{i_1+1}] \partial_Y h_e^{i_2}, \quad \partial_y^2 u_e^{i}=\varepsilon \partial_Y^2 u_e^i.
\end{align*}
The remaining high order terms of $(\ref{5.1})$ are the same form. Matching the order of $\varepsilon$, we obtain the following $i$-th order ideal MHD system
\begin{align}\label{5.2}\tag{A.11}
\left\{\begin{array}{l}
u_e^0 \partial_x u_e^i+v_e^i \partial_Y u_e^0-h_e^0 \partial_x h_e^i-g_e^i \partial_Y h_e^0+\partial_x p_e^i=f_1^{i}, \\
u_e^0 \partial_x h_e^i+v_e^i \partial_Y h_e^0-h_e^0 \partial_x u_e^i-g_e^i \partial_Y u_e^0=f_3^i,
\end{array}\right.
\end{align}
where
\begin{align*}
&-f_1^i=u_e^1 u_{e x}^{i-1}+u_e^{i-1} u_{e x}^{1}+v_e^1 u_{e Y}^{i-1}+ v_e^{i-1}u_{e Y}^{1}-h_e^1 h_{e x}^{i-1}-h_e^{i-1} h_{e x}^{1}-g_e^1 h_{e Y}^{i-1}- g_e^{i-1}h_{e Y}^{1}-\nu_1 \Delta u_e^{i-2},\\
&-f_3^i=u_e^1 h_{e x}^{i-1}+u_e^{i-1} h_{e x}^{1}+v_e^1 h_{e Y}^{i-1}+ v_e^{i-1}h_{e Y}^{1}-h_e^1 u_{e x}^{i-1}-h_e^{i-1} u_{e x}^{1}-g_e^1 u_{e Y}^{i-1}- g_e^{i-1}u_{e Y}^{1}-\nu_3 \Delta h_e^{i-2},
\end{align*}
and the $i$-th order boundary layer system
\begin{align}\label{5.3}
&\sum_{i_1+i_2=i}(u_e^{i_1}+u_p^{i_1}) \partial_x u_p^{i_2}+\sum_{i_1+i_2=i-1}u_p^{i_1} \partial_x u_e^{i_2+1}
+\sum_{i_1+i_2=i}(v_p^{i_1}+v_e^{i_1+1})\partial_y u_p^{i_2}
+\partial_x p_p^{i} - \nu_{1}\partial_y^2 u_p^{i} \notag\\
& \quad -\sum_{i_1+i_2=i}(h_e^{i_1}+h_p^{i_1}) \partial_x h_p^{i_2}-\sum_{i_1+i_2=i-1}h_p^{i_1} \partial_x h_e^{i_2+1}
-\sum_{i_1+i_2=i}(g_p^{i_1}+g_e^{i_1+1})\partial_y h_p^{i_2}-\nu_{1}\partial_{x}^{2}u_p^{i-2} \notag\\
& \quad +\sum_{i_1+i_2=i-1}[v_p^{i_1} \partial_Y u_e^{i_2}-g_p^{i_1} \partial_Y h_e^{i_2}]+\sum_{\substack{i_1+i_2=i,\\i_1,i_2\neq 0,1}}[u_e^{i_1}u_{e x}^{i_2}+v_e^{i_1}u_{eY}^{i_2}-h_e^{i_1}h_{e x}^{i_2}-g_e^{i_1}h_{eY}^{i_2}] \notag\\
& \quad+(4-n)\frac{E_1^0}{\varepsilon} + \triangle_{i-2}^{2}+ \triangle_{i-1}^{1} + \varepsilon^{-\frac{1}{2}}E_{r_1}^{i-1}=0,\notag\\
&\sum_{i_1+i_2=i}(u_e^{i_1}+u_p^{i_1}) \partial_x h_p^{i_2}+\sum_{i_1+i_2=i-1}u_p^{i_1} \partial_x h_e^{i_2+1}
+\sum_{i_1+i_2=i}(v_p^{i_1}+v_e^{i_1+1})\partial_y h_p^{i_2}
 - \nu_{3}\partial_y^2 h_p^{i} \notag\\
& \quad -\sum_{i_1+i_2=i}(h_e^{i_1}+h_p^{i_1}) \partial_x u_p^{i_2}-\sum_{i_1+i_2=i-1}h_p^{i_1} \partial_x u_e^{i_2+1}
-\sum_{i_1+i_2=i}(g_p^{i_1}+g_e^{i_1+1})\partial_y u_p^{i_2}-\nu_{3}\partial_{x}^{2}h_p^{i-2} \tag{A.12}\\
& \quad +\sum_{i_1+i_2=i-1}[v_p^{i_1} \partial_Y h_e^{i_2}-g_p^{i_1} \partial_Y u_e^{i_2}]+\sum_{\substack{i_1+i_2=i,\\i_1,i_2\neq 0,1}}[u_e^{i_1}h_{e x}^{i_2}+v_e^{i_1}h_{eY}^{i_2}-h_e^{i_1}u_{e x}^{i_2}-g_e^{i_1}u_{eY}^{i_2}] \notag\\
& \quad+(4-n)\frac{E_3^0}{\varepsilon} + \triangle_{i-2}^{4}+ \triangle_{i-1}^{3} + \varepsilon^{-\frac{1}{2}}E_{r_2}^{i-1}=0,\notag\\
&\sum_{i_1+i_2=i}(u_e^{i_1}+u_p^{i_1}) \partial_x g_p^{i_2}+\sum_{i_1+i_2=i-1}u_p^{i_1} \partial_x g_e^{i_2+2}
+\sum_{i_1+i_2=i}(v_p^{i_1}+v_e^{i_1+1})\partial_y g_p^{i_2}
 - \nu_{4}\partial_y^2 g_p^{i} \notag\\
& \quad -\sum_{i_1+i_2=i}(h_e^{i_1}+h_p^{i_1}) \partial_x v_p^{i_2}-\sum_{i_1+i_2=i-1}h_p^{i_1} \partial_x v_e^{i_2+2}
-\sum_{i_1+i_2=i}(g_p^{i_1}+g_e^{i_1+1})\partial_y v_p^{i_2}-\nu_{4}\partial_{x}^{2}g_p^{i-2} \notag\\
& \quad +\sum_{i_1+i_2=i-1}[v_p^{i_1} \partial_Y g_e^{i_2+1}-g_p^{i_1} \partial_Y v_e^{i_2+1}]+\sum_{\substack{i_1+i_2=i,\\i_1,i_2\neq 0,1}}[u_e^{i_1}g_{e x}^{i_2+1}+v_e^{i_1}g_{eY}^{i_2+1}-h_e^{i_1}v_{e x}^{i_2+1}-g_e^{i_1}v_{eY}^{i_2+1}] \notag\\
& \quad+(4-n)\frac{E_4^0}{\varepsilon} + \varepsilon^{-\frac{1}{2}}E_{r_3}^{i-1}=0.\notag
\end{align}
Hence, $R^{u}_i, R^{h}_i ~\text{and}~ R^{g}_i$ further reduce to
\begin{align}\label{5.4}\tag{A.13}
\left\{\begin{aligned}
R^{u}_i=&\sqrt{\varepsilon}\sum_{i_1+i_2=i}[(v_p^{i_1}+v_e^{i_1+1}) \partial_Y u_e^{i_2}-(g_p^{i_1}+g_e^{i_1+1}) \partial_Y h_e^{i_2}]-\nu_{1}\varepsilon \partial_Y^2 u_e^{i}, \\
R^{h}_i=&\sqrt{\varepsilon}\sum_{i_1+i_2=i}[(v_p^{i_1}+v_e^{i_1+1}) \partial_Y h_e^{i_2}-(g_p^{i_1}+g_e^{i_1+1}) \partial_Y u_e^{i_2}]-\nu_{3}\varepsilon \partial_Y^2 h_e^{i}, \\
R^{g}_i=&\sqrt{\varepsilon}\sum_{i_1+i_2=i}[(v_p^{i_1}+v_e^{i_1+1}) \partial_Y g_e^{i_2+1}-(g_p^{i_1}+g_e^{i_1+1}) \partial_Y v_e^{i_2+1}]-\nu_{4}\varepsilon \partial_Y^2 g_e^{i+1}\\
&+u_e^i \partial_{x}g_e^1+u_e^1 \partial_{x}g_e^i+u_e^0 \partial_{x}g_e^{i+1}-h_e^i \partial_{x}v_e^1-h_e^1 \partial_{x}v_e^i
-h_e^0 \partial_{x}v_e^{i+1}\\
&+v_e^i \partial_{Y}g_e^1+v_e^1 \partial_{Y}g_e^i-g_e^i \partial_{Y}v_e^1-g_e^1 \partial_{Y}v_e^i-\nu_{4} \Delta g_e^{i-1}.
\end{aligned}\right.
\end{align}
Next, we consider the $(i-1)$-th order terms of $R^{2}_{app} ~\text{and}~  R^{4}_{app}$. Since $R^{g}_{i-1}$ has been given, 
we have
\begin{align}\label{5.5}
 R^{v}_{i-1}= &\sum_{i_1+i_2=i-1} [(u_e^{i_1}+u_p^{i_1}) \partial_x+(v_p^{i_1}+v_e^{i_1+1}) \partial_y](v_p^{i_2}+v_e^{i_2+1})
 +\partial_Y p_e^{i}+\partial_{y} p_{p}^{i+1}\notag\\
 & -\nu_{2}\partial_y^2(v_p^{i-1}+v_e^{i})-\sum_{i_1+i_2=i-1} [(h_e^{i_1}+h_p^{i_1}) \partial_x+(g_p^{i_1}+g_e^{i_1+1}) \partial_y](g_p^{i_2}+g_e^{i_2+1})\notag\\
 & +\sum_{i_1+i_2=i-2}[(v_p^{i_1}+v_e^{i_1+1})\partial_{Y}v_e^{i_2+1}
 -(g_p^{i_1}+g_e^{i_1+1})\partial_{Y}g_e^{i_2+1}]-\nu_{2}\Delta v_e^{i-2}-\nu_2 \partial_{x}^2 v_p^{i-3}, \notag\\
R^{g}_{i-1}=&\sqrt{\varepsilon}\sum_{i_1+i_2=i-1}[(v_p^{i_1}+v_e^{i_1+1}) \partial_Y g_e^{i_2+1}-(g_p^{i_1}+g_e^{i_1+1}) \partial_Y v_e^{i_2+1}]-\nu_{4}\varepsilon \partial_Y^2 g_e^{i}\notag\\
&+u_e^{i-1} \partial_{x}g_e^1+u_e^1 \partial_{x}g_e^{i-1}+u_e^0 \partial_{x}g_e^{i}-h_e^{i-1} \partial_{x}v_e^1-h_e^1 \partial_{x}v_e^{i-1}-h_e^0 \partial_{x}v_e^{i}\notag\\
&+v_e^{i-1} \partial_{Y}g_e^1+v_e^1 \partial_{Y}g_e^{i-1}-g_e^{i-1} \partial_{Y}v_e^1-g_e^1 \partial_{Y}v_e^{i-1}-\nu_{4} \Delta g_e^{i-2}.\tag{A.14}
\end{align}
Matching the terms of order $\varepsilon^{\frac{i-1}{2}}$ in $R^{v}_{i-1}=R^{g}_{i-1}=0$,  $(u_e^i, v_e^i, h_e^i, g_e^i, p_e^i)$ satisfies
\begin{align}\label{5.6}\tag{A.15}
\left\{\begin{array}{l}
u_e^0 \partial_x v_e^i-h_e^0 \partial_x g_e^i+\partial_Y p_e^i=f_2^i,\\
u_e^0 \partial_x g_e^i-h_e^0 \partial_x v_e^i=f_4^i,
\end{array}\right.
\end{align}
where
\begin{align*}
-f_2^i=&v_e^{i-1}\partial_{Y}v_e^1+v_e^{1}\partial_{Y}v_e^{i-1}
       +u_e^{i-1}\partial_{x}v_e^1+u_e^{1}\partial_{x}v_e^{i-1}\\
       &-g_e^{i-1}\partial_{Y}g_e^1-g_e^{1}\partial_{Y}g_e^{i-1}
       -h_e^{i-1}\partial_{x}g_e^1-h_e^{1}\partial_{x}g_e^{i-1}-\nu_{2}\Delta v_e^{i-2},\\
-f_4^i=&v_e^{i-1}\partial_{Y}g_e^1+v_e^{1}\partial_{Y}g_e^{i-1}
       +u_e^{i-1}\partial_{x}g_e^1+u_e^{1}\partial_{x}g_e^{i-1}\\
       &-g_e^{i-1}\partial_{Y}v_e^1-g_e^{1}\partial_{Y}v_e^{i-1}
       -h_e^{i-1}\partial_{x}v_e^1-h_e^{1}\partial_{x}v_e^{i-1}-\nu_{4}\Delta g_e^{i-2}.
\end{align*}
When $i=2$, $f_1^2, f_2^2, f_3^2$ and $f_4^2$ are as follows
\begin{align*}
-f_1^2=&v_e^{1}\partial_{Y}u_e^1+u_e^{1}\partial_{x}u_e^1
       -g_e^{1}\partial_{Y}h_e^1-h_e^{1}\partial_{x}h_e^1-\nu_1\partial_Y^2u_e^0,\\
-f_2^2=&v_e^{1}\partial_{Y}v_e^1+u_e^{1}\partial_{x}v_e^1
       -g_e^{1}\partial_{Y}g_e^1-h_e^{1}\partial_{x}g_e^1,\\
-f_3^2=&v_e^{1}\partial_{Y}h_e^1+u_e^{1}\partial_{x}h_e^1
       -g_e^{1}\partial_{Y}u_e^1-h_e^{1}\partial_{x}u_e^1-\nu_3\partial_Y^2h_e^0,\\
-f_4^i=&v_e^{1}\partial_{Y}g_e^1+u_e^{1}\partial_{x}g_e^1
       -g_e^{1}\partial_{Y}v_e^1-h_e^{1}\partial_{x}v_e^1.
\end{align*}
The $(i+1)$-th order boundary layer of pressure $p_p^{i+1}$ is
\begin{align}\label{5.7}
p_p^{i+1}(x, y)= & \int_y^{1/\sqrt{\varepsilon}}\left\{\sum_{i_1+i_2=i-1}\left[(u_e^{i_1}+u_p^{i_1}) \partial_x v_p^{i_2}-(h_e^{i_1}+h_p^{i_1}) \partial_x g_p^{i_2}\right]+\sum_{i_1+i_2=i-1}\left[u_p^{i_1} \partial_{x}v_e^{i_2+1}-h_p^{i_1}\partial_{x}g_e^{i_2+1}\right]\right.\notag\\
& \left.+\sum_{i_1+i_2=i-1}\left[(v_p^{i_1}+v_e^{i_1+1}) \partial_{y}v_p^{i_2}-(g_p^{i_1}+g_e^{i_1+1}) \partial_{y}g_p^{i_2}\right]+\sum_{i_1+i_2=i-2}\left[v_p^{i_1} \partial_{Y}v_e^{i_2+1}-g_p^{i_1} \partial_{Y}g_e^{i_2+1}\right]\right.\notag\\
& \left.+\sum_{\substack{i_1+i_2=i,\\i_1, i_2 \neq 0,1}}(u_e^{i_1} \partial_x v_e^{i_2}+v_e^{i_1} \partial_{Y}v_e^{i_2}-h_e^{i_1} \partial_x g_e^{i_2}-g_e^{i_1} \partial_{Y}g_e^{i_2})-\nu_{2} \partial_y^2 v_p^{i-1}-\nu_{2}\partial_{x}^2 v_p^{i-3}\right\}(x, \theta) \mathrm{d} \theta. \tag{A.16}
\end{align}
Then, $R^{v}_{i-1}$ and $R^{g}_{i-1}$ may be rewritten as follows
\begin{align}\label{5.8}\tag{A.17}
\left\{
\begin{aligned}
R^{v}_{i-1}=\sqrt{\varepsilon} \sum_{i_1+i_2=i-1}\left[(v_p^{i_1}+v_e^{i_1+1}) \partial_Y v_e^{i_2+1}-(g_p^{i_1}+g_e^{i_1+1}) \partial_Y g_e^{i_2+1}\right] - \nu_{2}\varepsilon \partial_Y^2 v_e^i, \\
R^{g}_{i-1}=\sqrt{\varepsilon} \sum_{i_1+i_2=i-1}\left[(v_p^{i_1}+v_e^{i_1+1}) \partial_Y g_e^{i_2+1}-(g_p^{i_1}+g_e^{i_1+1}) \partial_Y v_e^{i_2+1}\right] - \nu_{4}\varepsilon \partial_Y^2 g_e^i.
\end{aligned}\right.
\end{align}

\subsubsection*{A.3.2 The $i$-th order MHD boundary layer profiles}


We now consider the $i$-th order MHD boundary layer profile $(u_p^i, v_p^i, h_p^i, g_p^i, p_p^i)(2\leq i \leq n-1)$.
System $(\ref{5.3})$ can be rewritten
\begin{align}\label{6.1.1}\tag{A.18}
\left\{\begin{aligned}
&u^0 \partial_x u_p^{i}+u_p^{i} \partial_x u^0+v_p^{i} \partial_y u^0+(v_p^0+\overline{v_e^1}) \partial_y u_p^{i}-\nu_{1}\partial_y^2 u_p^{i} +\partial_{x}p_p^i\notag\\
& \quad -h^0 \partial_x h_p^{i}-h_p^{i} \partial_x h^0-g_p^{i} \partial_y h^0-(g_p^0+\overline{g_e^1}) \partial_y h_p^{i}=F_{p_{1}}^i, \notag\\
&u^0 \partial_x h_p^{i}+u_p^{i} \partial_x h^0+v_p^{i} \partial_y h^0+(v_p^0+\overline{v_e^1}) \partial_y h_p^{i}- \nu_{3}\partial_y^2 h_p^{i} \notag\\
& \quad-h^0 \partial_x u_p^{i}-h_p^{i} \partial_x u^0-g_p^{i} \partial_y u^0-(g_p^0+\overline{g_e^1}) \partial_y u_p^{i}=F_{p_2}^i ,\\
&u^0 \partial_x g_p^{i}+u_p^{i} \partial_x g_p^0+v_p^{i} \partial_y g_p^0+(v_p^0+\overline{v_e^1}) \partial_y g_p^{i}- \nu_{4}\partial_y^2 g_p^{i} \notag\\
& \quad-h^0 \partial_x v_p^{i}-h_p^{i} \partial_x v_p^0-g_p^{i} \partial_y v_p^0-(g_p^0+\overline{g_e^1}) \partial_y v_p^{i}=F_{p_3}^i,\notag
\end{aligned}\right.
\end{align}
where
\begin{align*}
F_{p_1}^i=&-\left[\varepsilon^{-\frac{1}{2}}E_{r_1}^{i-1} +u_p^{i-1} \partial_{x}(u_p^1+u_e^1)-h_p^{i-1} \partial_{x}(h_p^1+h_e^1) + v_p^{i-1}(\partial_{Y}u_e^0+\partial_{y}u_p^1)
-g_p^{i-1}(\partial_{Y}h_e^0+\partial_{y}h_p^1)\right.\\
&+\partial_{x}u_{p}^{i-1}(u_e^1+u_p^1)-\partial_{x}h_{p}^{i-1}(h_e^1+h_p^1)
+ v_e^2\partial_{y}u_p^{i-1}-g_e^2\partial_{y}h_p^{i-1} +v_e^i \partial_{y}u_p^1-g_e^i \partial_{y}h_p^1\\
& \left.+u_e^{i}\partial_{x}u_p^0+u_p^0 \partial_{x}u_e^i-h_e^{i}\partial_{x}h_p^0-h_p^0 \partial_{x}h_e^i- \nu_{1}\partial_{x}^{2} u_p^{i-2}\right],\\
F_{p_2}^i=&-\left[\varepsilon^{-\frac{1}{2}}E_{r_2}^{i-1} +u_p^{i-1} \partial_{x}(h_p^1+h_e^1)-h_p^{i-1} \partial_{x}(u_p^1+u_e^1) + v_p^{i-1}(\partial_{Y}h_e^0+\partial_{y}h_p^1)
-g_p^{i-1}(\partial_{Y}u_e^0+\partial_{y}u_p^1)\right.\\
&+\partial_{x}h_{p}^{i-1}(u_e^1+u_p^1)-\partial_{x}u_{p}^{i-1}(h_e^1+h_p^1)
+ v_e^2\partial_{y}h_p^{i-1}-g_e^2\partial_{y}u_p^{i-1} +v_e^i \partial_{y}h_p^1-g_e^i \partial_{y}u_p^1\\
& \left.+u_e^{i}\partial_{x}h_p^0+u_p^0 \partial_{x}h_e^i-h_e^{i}\partial_{x}u_p^0-h_p^0 \partial_{x}u_e^i- \nu_{3}\partial_{x}^{2} h_p^{i-2}\right],\\
F_{p_3}^i=&-\left[\varepsilon^{-\frac{1}{2}}E_{r_3}^{i-1} +u_p^{i-1} \partial_{x}(g_p^1+g_e^2)-h_p^{i-1} \partial_{x}(v_p^1+v_e^2) + v_p^{i-1}(\partial_{Y}g_e^1+\partial_{y}g_p^1)
-g_p^{i-1}(\partial_{Y}v_e^1+\partial_{y}v_p^1)\right.\\
&+\partial_{x}g_{p}^{i-1}(u_e^1+u_p^1)-\partial_{x}v_{p}^{i-1}(h_e^1+h_p^1)
+ v_e^2\partial_{y}g_p^{i-1}-g_e^2\partial_{y}v_p^{i-1} +v_e^i \partial_{y}g_p^1-g_e^i \partial_{y}v_p^1\\
& \left.+u_e^{i}\partial_{x}g_p^0+u_p^0 \partial_{x}g_e^{i+1}-h_e^{i}\partial_{x}v_p^0-h_p^0 \partial_{x}v_e^{i+1}- \nu_{4}\partial_{x}^{2} g_p^{i-2}\right].
\end{align*}
Therefore,
\begin{align}\label{6.3.1}
E^{i}_{ r_1}:= & (u_e^0-u_e) \partial_x u_p^{i}+(v_e^1-\overline{v_e^1}) \partial_y u_p^{i}-(h_e^0-h_e) \partial_x h_p^{i}
-(g_e^1-\overline{g_e^1}) \partial_y h_p^{i}+\Delta_{i-2}^{2}+\Delta_{i-1}^{1}\notag\\
&+v_e^{i+1}\partial_{y}u_p^0 +u_e^{i-1}\partial_{x}u_p^1+u_p^1 \partial_{x}u_e^{i-1}+v_p^1\partial_{y}u_p^{i-1}
+(4-n)\frac{E_1^0}{\varepsilon}+v_p^0\partial_{Y}u_e^{i-1}
+v_p^{i-2}\partial_{Y}u_e^1\notag\\
& +\sum_{\substack{i_1+i_2=i\\i_1,i_2\neq 0,1}}(u_e^{i_1}+u_p^{i_1})\partial_{x}(u_e^{i_2}+u_p^{i_2})+\sum_{\substack{i_1+i_2=i\\i_1,i_2 \neq0,1}}(v_p^{i_1}+v_e^{i_1+1})\partial_{y}u_p^{i_2}+\sum_{\substack{i_1+i_2=i-1\\i_1\neq0,i_2\neq0,1}}
(v_p^{i_1}+v_e^{i_1+1})\partial_{Y}u_e^{i_2}\notag\\
& -\sum_{\substack{i_1+i_2=i\\i_1,i_2\neq 0,1}}(h_e^{i_1}+h_p^{i_1})\partial_{x}(h_e^{i_2}+h_p^{i_2})-\sum_{\substack{i_1+i_2=i\\i_1,i_2 \neq0,1}}(g_p^{i_1}+g_e^{i_1+1})\partial_{y}h_p^{i_2}-\sum_{\substack{i_1+i_2=i-1\\i_1\neq0,i_2\neq0,1}}
(g_p^{i_1}+g_e^{i_1+1})\partial_{Y}h_e^{i_2}\notag\\
&-g_e^{i+1}\partial_{y}h_p^0-h_e^{i-1}\partial_{x}h_p^1-h_p^1 \partial_{x}h_e^{i-1}-g_p^1\partial_{y}h_p^{i-1}
-g_p^0\partial_{Y}h_e^{i-1}-g_p^{i-2}\partial_{Y}h_e^1,\notag\\
E^{i}_{ r_2}:= & (u_e^0-u_e) \partial_x h_p^{i}+(v_e^1-\overline{v_e^1}) \partial_y h_p^{i}-(h_e^0-h_e) \partial_x u_p^{i}
-(g_e^1-\overline{g_e^1}) \partial_y u_p^{i}+\Delta_{i-2}^{4}+\Delta_{i-1}^{3}\notag\\
&+v_e^{i+1}\partial_{y}h_p^0 +u_e^{i-1}\partial_{x}h_p^1+u_p^1 \partial_{x}h_e^{i-1}+v_p^1\partial_{y}h_p^{i-1}
+(4-n)\frac{E_3^0}{\varepsilon}+v_p^0\partial_{Y}h_e^{i-1}
+v_p^{i-2}\partial_{Y}h_e^1\notag\\
& +\sum_{\substack{i_1+i_2=i\\i_1,i_2\neq 0,1}}(u_e^{i_1}+u_p^{i_1})\partial_{x}(h_e^{i_2}+h_p^{i_2})+\sum_{\substack{i_1+i_2=i\\i_1,i_2 \neq0,1}}(v_p^{i_1}+v_e^{i_1+1})\partial_{y}h_p^{i_2}+\sum_{\substack{i_1+i_2=i-1\\i_1\neq0,i_2\neq0,1}}
(v_p^{i_1}+v_e^{i_1+1})\partial_{Y}h_e^{i_2}\notag\\
& -\sum_{\substack{i_1+i_2=i\\i_1,i_2\neq 0,1}}(h_e^{i_1}+h_p^{i_1})\partial_{x}(u_e^{i_2}+u_p^{i_2})-\sum_{\substack{i_1+i_2=i\\i_1,i_2 \neq0,1}}(g_p^{i_1}+g_e^{i_1+1})\partial_{y}u_p^{i_2}-\sum_{\substack{i_1+i_2=i-1\\i_1\neq0,i_2\neq0,1}}
(g_p^{i_1}+g_e^{i_1+1})\partial_{Y}u_e^{i_2}\notag\\
&-g_e^{i+1}\partial_{y}u_p^0-h_e^{i-1}\partial_{x}u_p^1-h_p^1 \partial_{x}u_e^{i-1}-g_p^1\partial_{y}u_p^{i-1}
-g_p^0\partial_{Y}u_e^{i-1}-g_p^{i-2}\partial_{Y}u_e^1,\notag\\
E^{i}_{ r_3}:= & (u_e^0-u_e) \partial_x g_p^{i}+(v_e^1-\overline{v_e^1}) \partial_y g_p^{i}-(h_e^0-h_e) \partial_x v_p^{i}
-(g_e^1-\overline{g_e^1}) \partial_y v_p^{i}+v_e^{i+1}\partial_{y}g_p^0 +u_e^{i-1}\partial_{x}g_p^1\notag\\
&+u_p^1 \partial_{x}g_e^{i}+v_p^1\partial_{y}g_p^{i-1}
+(4-n)\frac{E_4^0}{\varepsilon}+v_p^0\partial_{Y}g_e^{i}
+v_p^{i-2}\partial_{Y}g_e^2-g_e^{i+1}\partial_{y}v_p^0-h_e^{i-1}\partial_{x}v_p^1\nonumber\\
& +\sum_{\substack{i_1+i_2=i\\i_1,i_2\neq 0,1}}(u_e^{i_1}+u_p^{i_1})\partial_{x}(g_e^{i_2+1}+g_p^{i_2})+\sum_{\substack{i_1+i_2=i\\i_1,i_2 \neq0,1}}(v_p^{i_1}+v_e^{i_1+1})\partial_{y}g_p^{i_2}+\sum_{\substack{i_1+i_2=i-1\\i_1\neq0,i_2\neq0,1}}
(v_p^{i_1}+v_e^{i_1+1})\partial_{Y}g_e^{i_2+1}\notag\\
& -\sum_{\substack{i_1+i_2=i\\i_1,i_2\neq 0,1}}(h_e^{i_1}+h_p^{i_1})\partial_{x}(v_e^{i_2+1}+v_p^{i_2})-\sum_{\substack{i_1+i_2=i\\i_1,i_2 \neq0,1}}(g_p^{i_1}+g_e^{i_1+1})\partial_{y}v_p^{i_2}-\sum_{\substack{i_1+i_2=i-1\\i_1\neq0,i_2\neq0,1}}
(g_p^{i_1}+g_e^{i_1+1})\partial_{Y}v_e^{i_2+1}\notag\\
&-h_p^1 \partial_{x}v_e^{i}-g_p^1\partial_{y}v_p^{i-1}
-g_p^0\partial_{Y}v_e^{i}-g_p^{i-2}\partial_{Y}v_e^2.\tag{A.19}
\end{align}
We note that in the system $(\ref{6.1.1})$,  the third equality can be deduced from the second equality and the boundary conditions.


\subsection*{A.4 The $\varepsilon^{\frac{n}{2}}$-order ideal inner MHD profile and MHD boundary layer profile}


In this subsection, we consider $(u_e^n, v_e^n, h_e^n, g_e^n, p_e^n)$ and $(u_p^n, v_p^n, h_p^n, g_p^n, p_p^n)$. First, we study the $n$-th order ideal MHD flow.


\subsubsection*{A.4.1 The $n$-th order ideal inner MHD profile}


Let us define $R^u_n ~\text{and}~ R^h_n$ by
\begin{align}\label{7.1}
R^{u}_n= & \sum_{\substack{i_1+i_2=n\\i_2\neq0}}[(u_e^{i_1}+u_p^{i_1}) \partial_x+(v_p^{i_1}+v_e^{i_1+1}) \partial_y](u_e^{i_2}+u_p^{i_2})
 +[(u_e^{n}+u_p^{n})\partial_{x}+v_p^{n}\partial_{y}](u_e^{0}+u_p^{0})\notag\\
&- \sum_{\substack{i_1+i_2=n\\i_2\neq0}}[(h_e^{i_1}+h_p^{i_1}) \partial_x+(g_p^{i_1}+g_e^{i_1+1}) \partial_y](h_e^{i_2}+h_p^{i_2})
 -[(h_e^{n}+h_p^{n})\partial_{x}+g_p^{n}\partial_{y}](h_e^{0}+h_p^{0})\notag\\
&+\sum_{i_1+i_2=n-1}[(v_p^{i_1}+v_e^{i_1+1})\partial_{Y}u_e^{i_2}
-(g_p^{i_1}+g_e^{i_1+1})\partial_{Y}h_e^{i_2}]
+\partial_x(p_e^n+p_p^n)
- \nu_{1}\partial_y^2(u_e^{n}+u_p^{n})\notag\\
&-\nu_{1}\partial_{x}^{2}u_p^{n-2}+ \triangle_{n-2}^{2}
+ \varepsilon^{-\frac{1}{2}}E_{r_1}^{n-1}+ \triangle_{n-1}^{1}-\nu_{1}\Delta u_e^{n-2},\notag\\
R^{h}_n= & \sum_{\substack{i_1+i_2=n\\i_2\neq0}}[(u_e^{i_1}+u_p^{i_1}) \partial_x+(v_p^{i_1}+v_e^{i_1+1}) \partial_y](h_e^{i_2}+h_p^{i_2})
 +[(u_e^{n}+u_p^{n})\partial_{x}+v_p^{n}\partial_{y}](h_e^{0}+h_p^{0}) \tag{A.20}\\
&- \sum_{\substack{i_1+i_2=n\\i_2\neq0}}[(h_e^{i_1}+h_p^{i_1}) \partial_x+(g_p^{i_1}+g_e^{i_1+1}) \partial_y](u_e^{i_2}+u_p^{i_2})
 -[(h_e^{n}+h_p^{n})\partial_{x}+g_p^{n}\partial_{y}](u_e^{0}+u_p^{0})\notag\\
&+\sum_{i_1+i_2=n-1}[(v_p^{i_1}+v_e^{i_1+1})\partial_{Y}h_e^{i_2}
-(g_p^{i_1}+g_e^{i_1+1})\partial_{Y}u_e^{i_2}]
- \nu_{3}\partial_y^2(h_e^{n}+h_p^{n})\notag\\
&-\nu_{3}\partial_{x}^{2}h_p^{n-2}+ \triangle_{n-2}^{4}
+ \varepsilon^{-\frac{1}{2}}E_{r_2}^{n-1}+ \triangle_{n-1}^{3}-\nu_{3}\Delta h_e^{n-2}.\notag
\end{align}
As the inner correctors are always calculated at $(x, \sqrt{\varepsilon} y)$, we have the following high order terms
\begin{align*}
& (v_p^{i_1}+v_e^{i_1+1}) \partial_y u_e^{i_2}=\sqrt{\varepsilon}(v_p^{i_1}+v_e^{i_1+1}) \partial_Y u_e^{i_2}, \quad v_p^n\partial_y (u_e^0+u_p^0)=\sqrt{\varepsilon}v_p^n \partial_Y (u_e^0+u_p^0), \quad \partial_y^2 u_e^n=\varepsilon \partial_Y^2 u_e^n.
\end{align*}
The remaining high order terms of $(\ref{7.1})$ are of the same form. Hence, they will be omitted here.
By matching the terms of order  $\varepsilon^{\frac{n}{2}}$, we obtain the following $n$-th order ideal MHD system
\begin{align}\label{7.2}\tag{A.21}
\left\{\begin{array}{l}
u_e^0 \partial_x u_e^n+v_e^n \partial_Y u_e^0-h_e^0 \partial_x h_e^n-g_e^n \partial_Y h_e^0+\partial_x p_e^n=f_1^{n}, \\
u_e^0 \partial_x h_e^n+v_e^n \partial_Y h_e^0-h_e^0 \partial_x u_e^n-g_e^n \partial_Y u_e^0=f_3^n,
\end{array}\right.
\end{align}
where
\begin{align*}
-f_1^n=&u_e^1 u_{e x}^{n-1}+u_e^{n-1} u_{e x}^1+v_e^1 u_{e Y}^{n-1}+v_e^{n-1} u_{e Y}^{1}\\
& -h_e^1 h_{e x}^{n-1}-h_e^{n-1} h_{e x}^1-g_e^1 h_{e Y}^{n-1}-g_e^{n-1} h_{e Y}^{1}-\nu_{1} \Delta u_e^{n-2},\\
-f_3^n=&u_e^1 h_{e x}^{n-1}+u_e^{n-1} h_{e x}^1+v_e^1 h_{e Y}^{n-1}+v_e^{n-1} h_{e Y}^{1}\\
& -h_e^1 u_{e x}^{n-1}-h_e^{n-1} u_{e x}^1-g_e^1 u_{e Y}^{n-1}-g_e^{n-1} u_{e Y}^{1}-\nu_{3} \Delta h_e^{n-2},
\end{align*}
and  the $n$-th order MHD boundary layer system
\begin{align}\label{7.3}
&\sum_{i_1+i_2=n}(u_e^{i_1}+u_p^{i_1}) \partial_x u_p^{i_2}+\sum_{i_1+i_2=n-1}u_p^{i_1} \partial_x u_e^{i_2+1}
+\sum_{i_1+i_2=n-1}(v_p^{i_1}+v_e^{i_1+1})\partial_y u_p^{i_2+1}
 +v_p^n \partial_y(u_e^0+u_p^0) \notag\\
&\quad-\sum_{i_1+i_2=n}(h_e^{i_1}+h_p^{i_1}) \partial_x h_p^{i_2}-\sum_{i_1+i_2=n-1}h_p^{i_1} \partial_x h_e^{i_2+1}
-\sum_{i_1+i_2=n-1}(g_p^{i_1}+g_e^{i_1+1})\partial_y h_p^{i_2+1}\notag\\
&\quad+\sum_{i_1+i_2=n-1}[v_p^{i_1}\partial_{Y}u_e^{i_2}
-g_p^{i_1}\partial_{Y}h_e^{i_2}]+\sum_{\substack{i_1+i_2=n\\i_1,i_2\neq0,1}}
[u_e^{i_1}\partial_{x}u_e^{i_2}+v_e^{i_1}\partial_{Y}u_e^{i_2}
-h_e^{i_1}\partial_{x}h_e^{i_2}-g_e^{i_1}\partial_{Y}h_e^{i_2}]\notag\\
&\quad-g_p^n\partial_{y}(h_e^0+h_p^0)+\partial_x p_p^n - \nu_{1}\partial_y^2 u_p^n-\nu_{1}\partial_{x}^{2}u_p^{n-2}+\triangle_{n-2}^{2}
 +\triangle_{n-1}^{1} + \varepsilon^{-\frac{1}{2}}E_{r_1}^{n-1}=0 ,\notag\\
&\sum_{i_1+i_2=n}(u_e^{i_1}+u_p^{i_1}) \partial_x h_p^{i_2}+\sum_{i_1+i_2=n-1}u_p^{i_1} \partial_x h_e^{i_2+1}
+\sum_{i_1+i_2=n-1}(v_p^{i_1}+v_e^{i_1+1})\partial_y h_p^{i_2+1}
 +v_p^n \partial_y(h_e^0+h_p^0) \notag\\
&\quad-\sum_{i_1+i_2=n}(h_e^{i_1}+h_p^{i_1}) \partial_x u_p^{i_2}-\sum_{i_1+i_2=n-1}h_p^{i_1} \partial_x u_e^{i_2+1}
-\sum_{i_1+i_2=n-1}(g_p^{i_1}+g_e^{i_1+1})\partial_y u_p^{i_2+1}\notag\\
&\quad+\sum_{i_1+i_2=n-1}[v_p^{i_1}\partial_{Y}h_e^{i_2}
-g_p^{i_1}\partial_{Y}u_e^{i_2}]+\sum_{\substack{i_1+i_2=n\\i_1,i_2\neq0,1}}
[u_e^{i_1}\partial_{x}h_e^{i_2}+v_e^{i_1}\partial_{Y}h_e^{i_2}
-h_e^{i_1}\partial_{x}u_e^{i_2}-g_e^{i_1}\partial_{Y}u_e^{i_2}]\notag\\
&\quad-g_p^n\partial_{y}(u_e^0+u_p^0) - \nu_{3}\partial_y^2 h_p^n-\nu_{3}\partial_{x}^{2}h_p^{n-2}+\triangle_{n-2}^{4}
 +\triangle_{n-1}^{3} + \varepsilon^{-\frac{1}{2}}E_{r_2}^{n-1}=0. \tag{A.22}
\end{align}
Hence,  $R^{u}_n$ and $R^{h}_n$ equal
\begin{align}\label{7.4}\tag{A.23}
\left\{\begin{aligned}
R^{u}_n=&\sqrt{\varepsilon}\sum_{i_1+i_2=n-1}[(v_p^{i_1}+v_e^{i_1+1}) \partial_Y u_e^{i_2+1}-(g_p^{i_1}+g_e^{i_1+1}) \partial_Y h_e^{i_2+1}]-\nu_{1}\varepsilon \partial_Y^2 u_e^n,\\
R^{h}_n=&\sqrt{\varepsilon}\sum_{i_1+i_2=n-1}[(v_p^{i_1}+v_e^{i_1+1}) \partial_Y h_e^{i_2+1}-(g_p^{i_1}+g_e^{i_1+1}) \partial_Y u_e^{i_2+1}]-\nu_{3}\varepsilon \partial_Y^2 h_e^n.
\end{aligned}\right.
\end{align}
Next, we consider the $(n-1)$-th order terms of $R^{2}_{app} ~\text{and}~  R^{4}_{app}$. $R^{g}_{n-1}$ has been given in $(\ref{5.4})$, thus, we have
\begin{align}\label{7.5}
 R^{v}_{n-1}= & \sum_{i_1+i_2=n-1}[(u_e^{i_1}+u_p^{i_1}) \partial_x+(v_p^{i_1}+v_e^{i_1+1}) \partial_y](v_p^{i_2}+v_e^{i_2+1})+\partial_Y p_e^n
 +\partial_y p_p^{n+1}-\nu_{2}\partial_y^2(v_p^{n-1}+v_e^{n}) \notag\\
& -\sum_{i_1+i_2=n-1}[(h_e^{i_1}+h_p^{i_1}) \partial_x+(g_p^{i_1}+g_e^{i_1+1}) \partial_y](g_p^{i_2}+g_e^{i_2+1})-\nu_2\partial_{x}^{2}v_p^{n-3}\notag\\
&+\sum_{i_1+i_2=n-2}[(v_p^{i_1}+v_e^{i_1+1})\partial_{Y}v_e^{i_2+1}
-(g_p^{i_1}+g_e^{i_1+1})\partial_{Y}g_e^{i_2+1} ] -\nu_{2}\Delta v_e^{n-2},\notag\\
R^{g}_{n-1}=&\sqrt{\varepsilon}\sum_{i_1+i_2=n-1}[(v_p^{i_1}+v_e^{i_1+1}) \partial_Y g_e^{i_2+1}-(g_p^{i_1}+g_e^{i_1+1}) \partial_Y v_e^{i_2+1}]-\nu_{4}\varepsilon \partial_Y^2 g_e^{n}\notag\\
&+u_e^{n-1} \partial_{x}g_e^1+u_e^1 \partial_{x}g_e^{n-1}+u_e^0 \partial_{x}g_e^{n}-h_e^{n-1} \partial_{x}v_e^1-h_e^1 \partial_{x}v_e^{n-1}-h_e^0 \partial_{x}v_e^{n}\notag\\
&+v_e^{n-1} \partial_{Y}g_e^1+v_e^1 \partial_{Y}g_e^{n-1}-g_e^{n-1} \partial_{Y}v_e^1-g_e^1 \partial_{Y}v_e^{n-1}-\nu_{4} \Delta g_e^{n-2}.\tag{A.24}
\end{align}
Writing $R^{v}_{n-1}=R^{g}_{n-1}=0$ and matching the various terms of order $\varepsilon^{\frac{n-1}{2}}$, the $n$-th order ideal MHD profile $(u_e^n, v_e^n, h_e^n, g_e^n, p_e^n)$ satisfies
\begin{align}\label{7.6}\tag{A.25}
\left\{\begin{array}{l}
u_e^0 \partial_x v_e^n-h_e^0 \partial_x g_e^n+\partial_Y p_e^n=f_2^n,\\
u_e^0 \partial_x g_e^n-h_e^0 \partial_x v_e^n=f_4^n,
\end{array}\right.
\end{align}
where
\begin{align*}
-f_2^n=&u_e^1\partial_{x}v_e^{n-1}+u_e^{n-1}\partial_{x}v_e^1
       +v_e^1\partial_{Y}v_e^{n-1}+v_e^{n-1}\partial_{Y}v_e^1 \\
       &-h_e^1\partial_{x}g_e^{n-1}-h_e^{n-1}\partial_{x}g_e^1
       -g_e^1\partial_{Y}g_e^{n-1}-g_e^{n-1}\partial_{Y}g_e^1
       -\nu_{2}\Delta v_e^{n-2},\\
-f_4^n=&u_e^1\partial_{x}g_e^{n-1}+u_e^{n-1}\partial_{x}g_e^1
       +v_e^1\partial_{Y}g_e^{n-1}+v_e^{n-1}\partial_{Y}g_e^1\\
       &-h_e^1\partial_{x}v_e^{n-1}-h_e^{n-1}\partial_{x}v_e^1
       -g_e^1\partial_{Y}v_e^{n-1}-g_e^{n-1}\partial_{Y}v_e^1
       -\nu_{4}\Delta g_e^{n-2},
\end{align*}
and $p_p^{n+1}$ is obtained through the first equation of $(\ref{7.5})$
\begin{align}\label{7.7}
p_p^{n+1}(x, y)= & \int_y^{1/\sqrt{\varepsilon}}\left\{\sum_{i_1+i_2=n-1}\left[(u_e^{i_1}+u_p^{i_1}) \partial_x v_p^{i_2}-(h_e^{i_1}+h_p^{i_1}) \partial_x g_p^{i_2}\right]+\sum_{i_1+i_2=n-1}\left[u_p^{i_1} \partial_x v_e^{i_2+1}-h_p^{i_1} \partial_x g_e^{i_2+1}\right]\right.\notag\\
&+\sum_{i_1+i_2=n-1}\left[(v_p^{i_1}+v_e^{i_1+1}) \partial_y v_p^{i_2}-(g_p^{i_1}+g_e^{i_1+1}) \partial_y g_p^{i_2}\right]+\sum_{i_1+i_2=n-2}\left[v_p^{i_1} \partial_Y v_e^{i_2+1}-g_p^{i_1} \partial_Y g_e^{i_2+1}\right] \notag\\
&\left.+\sum_{\substack{i_1+i_2=n,\\ i_1, i_2 \neq 0,1}}(u_e^{i_1} \partial_x v_e^{i_2}+v_e^{i_1} \partial_{Y}v_e^{i_2}-h_e^{i_1} \partial_x g_e^{i_2}-g_e^{i_1} \partial_{Y}g_e^{i_2})\right.\notag\\
& \left.-\nu_{2}\partial_{y}^2v_p^{n-1}
-\nu_2 \partial_x^2 v_p^{n-3}
\right\}(x,\theta) \mathrm{d} \theta. \tag{A.26}
\end{align}
Then, $R^{v}_{n-1}$ and $R^{g}_{n-1}$ equal
\begin{align}\label{7.8}\tag{A.27}
\begin{aligned}
R^{v}_{n-1}=&\sqrt{\varepsilon}\sum_{i_1+i_2=n-1}\left[(v_p^{i_1}+v_e^{i_1+1}) \partial_Y v_e^{i_2+1}-(g_p^{i_1}+g_e^{i_1+1}) \partial_Y g_e^{i_2+1}\right]
- \nu_{2}\varepsilon \partial_Y^2 v_e^n, \\
R^{g}_{n-1}=&\sqrt{\varepsilon}\sum_{i_1+i_2=n-1}\left[(v_p^{i_1}+v_e^{i_1+1}) \partial_Y g_e^{i_2+1}-(g_p^{i_1}+g_e^{i_1+1}) \partial_Y v_e^{i_2+1}\right]
- \nu_{4}\varepsilon \partial_Y^2 g_e^n.
\end{aligned}
\end{align}


\subsubsection*{A.4.2 The $n$-th order MHD boundary layer profile}


Now, we discuss the $n$-th order  boundary layer profile $(u_p^n, v_p^n, h_p^n, g_p^n, p_p^n)$, which solve $(\ref{7.3})$.
Then the system $(\ref{7.3})$ is written as
\begin{align}\label{8.1.1}
&u^0 \partial_x u_p^{n}+u_p^{n} \partial_x u^0+v_p^{n} \partial_y u^0
 +(v_p^0+\overline{v_e^1}) \partial_y u_p^{n}
 -\nu_{1}\partial_y^2 u_p^{n} + \partial_{x}p_p^n \notag\\
& \quad -h^0 \partial_x h_p^{n}-h_p^{n} \partial_x h^0
 -g_p^{n} \partial_y h^0
 -(g_p^0+\overline{g_e^1}) \partial_y h_p^{n}=F_{p_{1}}^n,\notag\\
&u^0 \partial_x h_p^{n}+u_p^{n} \partial_x h^0+v_p^{n} \partial_y h^0
 +(v_p^0+\overline{v_e^1}) \partial_y h_p^{n}
 -\nu_{3}\partial_y^2 h_p^{n} \notag\\
&\quad-h^0 \partial_x u_p^{n}-h_p^{n} \partial_x u^0
 -g_p^{n} \partial_y u^0
 -(g_p^0+\overline{g_e^1}) \partial_y u_p^{n} =F_{p_2}^n,\tag{A.28}
\end{align}
where $F_{p_1}^n$ and $F_{p_2}^n$ are as follows
\begin{align*}
F_{p_1}^n=&-\left[\varepsilon^{-\frac{1}{2}}E_{r_1}^{n-1}+u_p^{n-1}\partial_{x}
(u_p^1+u_e^1)+ (u_e^1+u_p^1)\partial_{x}u_p^{n-1}
+ v_p^{n-1}(\partial_{Y}u_e^0+\partial_{y}u_p^1) + v_e^2\partial_{y}u_p^{n-1} + v_e^n \partial_{y}u_p^1\right.\\
&\quad+\partial_{x} u_e^n u_p^0+ u_e^n \partial_{x}u_p^0-h_p^{n-1}\partial_{x}
(h_p^1+h_e^1)- (h_e^1+h_p^1)\partial_{x}h_p^{n-1}
 -g_p^{n-1}(\partial_{Y}h_e^0+\partial_{y}h_p^1)- g_e^2\partial_{y}h_p^{n-1}\\
&\left.\quad - g_e^n \partial_{y}h_p^1-\partial_{x} h_e^n h_p^0-h_e^n \partial_{x}h_p^0
- \nu_{1}\partial_{x}^{2} u_p^{n-2}\right],\\
F_{p_2}^n=&-\left[\varepsilon^{-\frac{1}{2}}E_{r_2}^{n-1}+u_p^{n-1}\partial_{x}
(h_p^1+h_e^1)+ (u_e^1+u_p^1)\partial_{x}h_p^{n-1}
+ v_p^{n-1}(\partial_{Y}h_e^0+\partial_{y}h_p^1) + v_e^2\partial_{y}h_p^{n-1} + v_e^n \partial_{y}h_p^1\right.\\
&\quad+\partial_{x} h_e^n u_p^0+ u_e^n \partial_{x}h_p^0-h_p^{n-1}\partial_{x}
(u_p^1+u_e^1)- (h_e^1+h_p^1)\partial_{x}u_p^{n-1}
 -g_p^{n-1}(\partial_{Y}u_e^0+\partial_{y}u_p^1)- g_e^2\partial_{y}u_p^{n-1}\\
&\left.\quad - g_e^n \partial_{y}u_p^1-\partial_{x} u_e^n h_p^0-h_e^n \partial_{x}u_p^0
- \nu_{3}\partial_{x}^{2} h_p^{n-2}\right].
\end{align*}
Therefore, the new error terms  are
\begin{align}\label{8.3}
E^n_{ r_1}:= &\sum_{\substack{i_1+i_2=n\\i_1,i_2\neq0,1}}(v_p^{i_1}+v_e^{i_1+1})
\partial_{y}u_p^{i_2}
+\sum_{\substack{i_1+i_2=n\\i_1,i_2\neq0,1}}(u_e^{i_1}+u_p^{i_1})
\partial_{x}(u_e^{i_2}+u_p^{i_2})+\sum_{\substack{i_1+i_2=n-1\\i_1\neq0,i_2\neq0,1}}(v_p^{i_1}+v_e^{i_1+1})\partial_{Y} u_e^{i_2}\notag\\
&-\sum_{\substack{i_1+i_2=n\\i_1,i_2\neq0,1}}(g_p^{i_1}+g_e^{i_1+1})
\partial_{y}h_p^{i_2}
-\sum_{\substack{i_1+i_2=n\\i_1,i_2\neq0,1}}(h_e^{i_1}+h_p^{i_1})
\partial_{x}(h_e^{i_2}+h_p^{i_2})-\sum_{\substack{i_1+i_2=n-1\\i_1\neq0,i_2\neq0,1}}(g_p^{i_1}+g_e^{i_1+1})\partial_{Y} h_e^{i_2}\notag\\
&+(u_e^0-u_e) \partial_x u_p^{n}+(v_e^1-\overline{v_e^1})\partial_{y}u_p^n+v_p^n\partial_y u_e^0
+u_e^{n-1}\partial_{x} u_p^{1}+u_p^1\partial_{x} u_e^{n-1}+v_p^1\partial_{y}u_p^{n-1}\notag\\
&+v_p^0\partial_{Y} u_e^{n-1}
+v_p^{n-2}\partial_{Y}u_e^1+\triangle_{n-1}^1+\triangle_{n-2}^{2}-(h_e^0-h_e) \partial_x h_p^{n}-(g_e^1-\overline{g_e^1})\partial_{y}h_p^n\notag\\
&-g_p^n\partial_y h_e^0
-h_e^{n-1}\partial_{x} h_p^{1}-h_p^1\partial_{x} h_e^{n-1}-g_p^1\partial_{y}h_p^{n-1}-g_p^0\partial_{Y} h_e^{n-1}
-g_p^{n-2}\partial_{Y}h_e^1,\notag\\
E^n_{ r_2}:= &\sum_{\substack{i_1+i_2=n\\i_1,i_2\neq0,1}}(v_p^{i_1}+v_e^{i_1+1})
\partial_{y}h_p^{i_2}
+\sum_{\substack{i_1+i_2=n\\i_1,i_2\neq0,1}}(u_e^{i_1}+u_p^{i_1})
\partial_{x}(h_e^{i_2}+h_p^{i_2})+\sum_{\substack{i_1+i_2=n-1\\i_1\neq0,i_2\neq0,1}}(v_p^{i_1}+v_e^{i_1+1})\partial_{Y} h_e^{i_2}\notag\\
&-\sum_{\substack{i_1+i_2=n\\i_1,i_2\neq0,1}}(g_p^{i_1}+g_e^{i_1+1})
\partial_{y}u_p^{i_2}
-\sum_{\substack{i_1+i_2=n\\i_1,i_2\neq0,1}}(h_e^{i_1}+h_p^{i_1})
\partial_{x}(u_e^{i_2}+u_p^{i_2})-\sum_{\substack{i_1+i_2=n-1\\i_1\neq0,i_2\neq0,1}}(g_p^{i_1}+g_e^{i_1+1})\partial_{Y} u_e^{i_2}\notag\\
&+(u_e^0-u_e) \partial_x h_p^{n}+(v_e^1-\overline{v_e^1})\partial_{y}h_p^n+v_p^n\partial_y h_e^0
+u_e^{n-1}\partial_{x} h_p^{1}+u_p^1\partial_{x} h_e^{n-1}+v_p^1\partial_{y}h_p^{n-1}\notag\\
&+v_p^{n-2}\partial_{Y}h_e^1+\triangle_{n-1}^3+\triangle_{n-2}^{4}-(h_e^0-h_e) \partial_x u_p^{n}-(g_e^1-\overline{g_e^1})\partial_{y}u_p^n
-g_p^n\partial_y u_e^0\notag\\
&-h_e^{n-1}\partial_{x} u_p^{1}-h_p^1\partial_{x} u_e^{n-1}-g_p^1\partial_{y}u_p^{n-1}-g_p^0\partial_{Y} u_e^{n-1}
-g_p^{n-2}\partial_{Y}u_e^1.\tag{A.29}
\end{align}
This concludes the derivation of any order ideal MHD profiles and MHD boundary layer profiles.

\bigskip
\noindent {\bf {Acknowledgments  }}\\
This work is supported by NSF of China under the Grant 12271032.


\small

\begin{thebibliography}{10}

\bibitem{CRWZ2020} D.X. Chen, S.Q. Ren, Y.X. Wang, Z.F. Zhang, Global well-posedness of the $2$D magnetic Prandtl model in the Prandtl-Hartmann regime. Asymptotic Analysis, 1-21, 2020.

\bibitem{DLN2021} S.J. Ding, Z.L. Lin, D.J. Niu, Stability of the boundary layer expansion for the $3$D plane parallel MHD flow. J. Math. Phys., 62(2), 021510, 2021.

\bibitem{DLX2021} S.J. Ding, Z.L. Lin, F. Xie, Verification of Prandtl boundary layer ansatz for the steady electrically conducting fluids with a moving physical boundary. SIAM J. Math. Anal., 53(5), 4997-5059, 2021.

\bibitem{DJL2021} S.J. Ding, Z.J. Ji, Z.L. Lin, Validity of Prandtl layer theory for steady magnetohydrodynamics over a moving plate with nonshear outer ideal MHD flows. J. Differ. Equations, 278, 220-293, 2021.

\bibitem{DJL2022} S.J. Ding, Z.J. Ji, Z.L. Lin, Global-in-$x$ stability of Prandtl layer expansions for steady magnetodynamics flows over a moving plate. 	arXiv:2203.00380, 2022.


\bibitem{DW2023} S.J. Ding, C.Y. Wang, Validity of Prandtl layer expansions for steady magnetohydrodynamics over a rotating disk. J. Math. Phys., 64(2), 021501, 2023.

\bibitem{GVP2017} D. G\'{e}rard-Varet, M. Prestipino, Formal derivation and stability analysis of boundary layer models in MHD. Z. Angew. Math. Phys., 68(76), 2017.

\bibitem{GLY2024} J.C. Gao, M.L. Li, Z.A. Yao, Higher regularity and asymptotic behavior of 2D magnetic Prandtl model in the Prandtl-Hartmann regime. J. Differ. Equations, 386, 294-367,  2024.

\bibitem{GGH2017} J.C. Gao, B.L. Guo, D.W. Huang, Local-in-time well-posedness of boundary layer system for the full
    incompressible MHD equations by energy methods. arXiv:1704.06766,   2017.


\bibitem{GT1983} D. Gilbarg, N.S. Trudinger, Elliptic partial difffferential equations of second order. Second edition. Grundlehren der Mathematischen Wissenschaften 224, Springer-Verlag, Berlin, 1983.

\bibitem{GJ2021} L.H. Guo, Z.J. Ji, Validity of boundary layer theory for the 3D plane-parallel nonhomogeneous electrically conducting flows. Math. Methods Appl. Sci., 44(11), 8862-8882, 2021.

\bibitem{GI2023CPAM} Y. Guo, S. Iyer, Validity of steady Prandtl layer expansions. Comm. Pure Appl. Math., 76(11), 3150-3232, 2023.


\bibitem{GI2023QAM} Y. Guo, S. Iyer, Steady Prandtl layer expansions with external forcing. Quart. Appl. Math., 81(2), 375-411, 2023.

\bibitem{GI2021} Y. Guo, S. Iyer, Regularity and expansion for steady Prandtl equations. Comm. Math. Phys., 382(3), 1403-1447, 2021.

\bibitem{GN2017} Y. Guo, T. Nguyen, Prandtl boundary layer expansions of steady Navier-Stokes flows over a moving plate. Ann. PDE, 3(10), 1-58,  2017.

\bibitem{HLY2019} Y.T. Huang, C.J. Liu, T. Yang, Local-in-time well-posedness for compressible MHD boundary layer. J. Differ. Equations, 266(6), 2978-3013, 2019.

\bibitem{Iyer2017} S. Iyer, Steady Prandtl boundary layer expansions over a rotating disk. Arch. Ration. Mech. Anal., 224(2), 421-469, 2017.

\bibitem{Iyer2019}S. Iyer, Global steady Prandtl expansion over a moving boundary I. Peking Math. J., 2(2), 155-238, 2019.

\bibitem{Iyer2019-1} S. Iyer, Global steady Prandtl expansion over a moving boundary II. Peking Math. J., 2, 353-437, 2019.

\bibitem{Iyer2020} S. Iyer, Global steady Prandtl expansion over a moving boundary III. Peking Math. J., 3(1), 47-102, 2020. 

\bibitem{Iyer2019-2} S. Iyer, Steady Prandtl layers over a moving boundary: non-shear Euler flows. SIAM J. Math. Anal., 51(3), 1657-1695, 2019.


\bibitem{LD2020} Q.R. Li, S.J. Ding, Symmetrical Prandtl boundary layer expansions of steady Navier-Stokes equations on bounded domain. J. Differ. Equations, 268(4), 1771-1819, 2020. 

\bibitem{LLZ2015} F.H. Lin, X. Li, P. Zhang, Global small solutions of $2$D incompressible MHD system. J. Differ. Equations, 259, 5440-5485,  2015.

\bibitem{LZ2018} X.Y. Lin, T. Zhang, Almost global existence for 2D magnetohydrodynamics boundary layer system. Math. Methods Appl. Sci., 41(17), 7530-7553, 2018.


\bibitem{LXY2019} C.J. Liu, F. Xie, T. Yang, MHD boundary layers theory in Sobolev spaces without monotonicity, I. Well-posedness theory. Comm. Pure Appl. Math., 72(1), 63-121, 2019.

\bibitem{LXY2017} C.J. Liu, F. Xie, T. Yang, MHD boundary layers in Sobolev spaces without monotonicity, II. Convergence theory. arXiv:1704.00523, 2017.



\bibitem{LYZ2023} C.J. Liu, T.Yang, Z. Zhang, Validity of Prandtl expansions for steady MHD in the Sobolev framework. SIAM J. Math. Anal., 55(3), 2377-2410, 2023.




\bibitem{Prandtl1904} L. Prandtl, Uber ussigkeits-bewegung bei sehr kleiner reibung. In Verhandlungen des III Internationalen Mathematiker-Kongresses, Heidelberg. Teubner, Leipzig, 484-491, 1904. English Translation:" Motion of fluids with very little viscosity," Technical Memorandum No. 452 by National Advisory Committee for Aeuronautics. 




\bibitem{WM2018} J. Wang, S.X. Ma,  On the steady Prandtl type equations with magnetic effects arising from $2$D incompressible MHD equations in a half plane. J. Math. Phys., 59(12), 121508, 2018.

\bibitem{WX2017} S. Wang, Z.P. Xin, Boundary layer problems in the viscosity-diffusion vanishing limits for the incompressible MHD systems.  Sci. Sin. Math., 47(10), 1303-1326, 2017.

\bibitem{WW2019AMS} N. Wang, S. Wang, The boundary layer for MHD equations in a plane-parallel channel. Acta Math. Sci., 39A(4), 738-760, 2019.


\bibitem{WW2019MMAS} Z.L. Wu and S. Wang, Viscosity vanishing limit of the nonlinear pipe magnetohydrodynamic flow with diffusion.  Math. Methods Appl. Sci., 42(1), 161-174, 2019.

\bibitem{XLL2014} X.Q. Xie, L. Luo, C.M. Li, Boundary layer for MHD equations with the noncharacteristic boundary conditions. Chin. Ann. Math., 35A(2), 171-192, 2014.


\bibitem{XY2018SJMA} F. Xie, T. Yang, Global-in-time stability of 2D MHD boundary layer in the Prandtl-Hartmann regime. SIAM J. Math. Anal., 50(6), 5749-5760, 2018.




\end{thebibliography}
\end{document}